\documentclass{amsart}
\usepackage{a4wide}
\usepackage{bbm}
\usepackage{lipsum}
\usepackage{enumerate}
\usepackage{showlabels}
\usepackage{amsmath,amsthm,amsfonts,amssymb,amscd,mathrsfs}
\usepackage[pdftex, colorlinks=true]{hyperref}

\newtheorem{theo}{Theorem}[section]
\newtheorem{prop}{Proposition}[section]
\newtheorem{lemma}{Lemma}[section]
\newtheorem{cor}{Corollary}[section]
\newtheorem{defi}{Definition}[section]
\newtheorem{ques}{Question}[section]
\newtheorem{rem}{Remark}[section]

\newtheorem{CClaim}{Claim}

\usepackage{graphicx}
 \usepackage[all]{xy}
 
\usepackage{xcolor}
\usepackage{boxhandler}
\usepackage{caption}
\usepackage{lipsum}
\usepackage{xcolor}
\usepackage{boxhandler}
\usepackage{caption}
\usepackage{lipsum}
\def \diam {\mathsf{diam}}

\def \Id {\mathsf{Id}}

\def \ocap{\mathsf{ocap}}
\def \dim{\mathsf{dim}}
\def \codim{\mathsf{codim}}

\def \Int{\mathsf{Int}}

\begin{document}


\title{Symbolic extensions and uniform generators for topological regular flows.}

\author{DAVID BURGUET}
\address{LPSM - CNRS UMR 8001 \\
\newline
Universite Paris 6 \\
 75252 Paris Cedex 05 FRANCE\\
 \newline
david.burguet@upmc.fr}


\begin{abstract}Building on the theory of symbolic extensions and uniform generators for discrete transformations we develop a similar theory for topological regular flows. In this context a symbolic extension is given by a suspension flow over a subshift. 
\end{abstract}
\maketitle
\tableofcontents
\section{Introduction}
Given an invertible dynamical system $(X,f)$ a \textit{generator} is a finite partition $P$, which   ``generates" the system in the sense that 
the map from $(X,f)$ to $(P^\mathbb{Z},\sigma)$, with $\sigma$ being the usual shift, which associates to any $x\in X$ its $P$-name $(P(f^kx))_{k\in \mathbb{Z}}$
defines an embedding (where $P(x)$ denotes the atom of $P$ containing $x\in X$). The nature of the embedding depends on the structure of the system, e.g. if we consider  a measure preserving system (resp. Borel system, resp. topological system) we require the embedding to be measure theoretical  (resp. Borel, resp. topological). A necessary set-theoretical condition for the existence of a generator is given by the cardinality of periodic points which has to be smaller than in a full shift over  a finite alphabet, i.e.  $\sup_{n\in \mathbb{N}\setminus \{0\}}\frac{1}{n}\log \sharp \{x, \ f^nx=x\}<+\infty$. There are other  conditions, of dynamical nature, namely entropic and expansive properties. \\

 Generators in ergodic theory have a long history. As the  entropy is preserved by isomorphism, an ergodic system with a generator  has finite entropy.  W.Krieger  showed the converse in  \cite{kri} : any ergodic system on a Lebesgue space with finite entropy admits a (finite) generator. For topological systems, expansiveness (which implies finite topological entropy) completely characterizes systems with a generator (see \cite{ky} and \cite{kr}). New developments appear recently for Borel
systems. Namely M.Hochman has proved in  \cite{Ho,Hoc} that a Borel system  admits a generator if and only if the  entropy of its ergodic invariant measures is  bounded from above. 

When considering  the time $t$-map $\phi_t$ of a (regular) flow $(X,\Phi)$ we are interrested in generators $P$ whose atoms are \textit{towers} associated to a cross-section, i.e of the form $\{\phi_t(x), \ x\in A \text{ and }0\leq t<t_S(x)\}$ with $A\subset S$ for a cross-section $S$ and its  return time $t_S$.  As a first step we aim to represent the system as a suspension flow or equivalently to build a  (global) cross-section. For aperiodic ergodic flows it was achieved by W.Ambrose \cite{amb}, whereas V.M.Wagh \cite{wa}  obtained the analogous result in the Borel case. As shown by D.Rudolph \cite{rud} the roof function in this representation of ergodic flows may be always assumed to be two-valued ; moreover for finite entropy flows  the two-partition consisting of the towers 
with constant return time is generating not only  for the flow, but also for the ergodic time $t$-maps with  $t$ less than the minimal return time. Recently K.Slutsky \cite{KS} built for  Borel systems a Borel cross-section with a two-valued return map.\\

In the present paper we are interested in  generators for topological systems.  M.Boyle and  T.Downarowicz \cite{BD} have developped a new theory of entropy revealing fine  properties of expansiveness of discrete topological systems. They introduce new entropy invariants which allow in particular to know whether the system may be encoded with a finite alphabet or not. Formally such a code is given by a topological extension by a subshift over a finite alphabet, also called a \textit{symbolic extension}. More recently  T.Downarowicz and the author \cite{bdo} have related the theory of symbolic extensions with  a  Krieger-like generators problem.  For a discrete topological system $(X,T)$ they introduced \textit{uniform generators} as Borel partitions $P$ of $X$  whose iterated partitions $P_T^{[-n,n]}:=\bigvee_{k=-n}^nT^kP$ have a diameter going to zero with $n$, in other terms $\sup_{y\in P_T^{[-n,n]}(x)}d(y,x)$ goes to zero uniformly in $x$ when $n$ goes to infinity (with $d$ being the distance on $X$). By Theorem 1 in \cite{bdo} a uniform generator is given by a symbolic extension with a Borel  embedding and vice versa.  For aperiodic systems the existence of symbolic extensions is equivalent to the existence of uniform generators whereas the presence of periodic points  generates other constraints to build uniform generators (see Theorem 55 in \cite{bdo}). \\

We aim to develop such theories for topological flows. Regarding the discrete systems given by the time $t$-maps of a real flow, M.Boyle and  T.Downarowicz  have proved that for $t\neq  0$ the time $t$-map  admits a symbolic extension if and only so does the time $1$-map  (Theorem 3.4 in \cite{bdd}).  
 Nevertheless we  consider here the flow in its own, not only through its time-$t$ maps. We call \textit{symbolic extension} of a topological flow any topological extension given by  the suspension flow over a subshift with a positive  continuous roof function, whereas a \textit{uniform generator} for the flow   is a symbolic extension of the flow with a Borel embedding.\\

 For Axiom A flows R.Bowen \cite{BOW} built   a finite-to-one symbolic  extension. Moreover the roof function may be chosen to be H\"{o}lder continuous and the subshift is of finite type. R.Bowen also introduced a notion of expansiveness for flows (satisfied in particular by Axiom A flows). The existence of symbolic extensions preserving entropy for expansive discrete dynamical systems is now well-known. For expansive flows, R.Bowen and P.Walters \cite{BW} built a symbolic extension but they ask whether this extension preserves the entropy. Their construction involves closed cross-sections and they wonder if one could choose carefully the closed cross-sections so that the associated symbolic extension has the same topological  entropy. We will give  a  positive answer to this question for $C^2$ expansive flows.\\
 
A first step in the theory of symbolic extensions  or uniform generators consists in reducing the problem to zero-dimensional systems. This is done by considering partitions with \textit{small boundary}, i.e. with boundaries having zero measure for any invariant probability measure. 
 The existence of  partitions with small boundary and arbitrarily small diameter, known as the small boundary property,  is related with the  deep theory of mean dimension \cite{lin}. In the Section 2 we introduce a small boundary property for flows. This property  is always satisfied for $C^2$ smooth  flows when the set of periodic orbits  with period less than $T$ is finite for any $T>0$. 
   For flows with the \textit{small flow boundary} property we may build, by a similar construction as R.Bowen and P.Walters,  a topological extension preserving entropy given by a suspension  flow over a zero-dimensional discrete system.  
 Then a symbolic extension of the flow may be built from a symbolic extension of this discrete system. 
  Our main results may be stated as follows:


\begin{theo}\label{intr}
Let $(X,\Phi)$ be a (resp. aperiodic)  regular topological flow with the small flow boundary property. It admits a symbolic  extension (resp. a uniform generator) if and only if for some (any) $t\neq 0$, the time $t$-map admits a symbolic  extension. Moreover this property is invariant under orbit equivalence. 
\end{theo}

 For a  $C^\infty$ smooth aperiodic regular flow on a compact manifold, the time-$t$ maps is also $C^\infty$ smooth and thus  admits a symbolic extension. Therefore the flow admits a uniform generator. In fact  the symbolic extension (associated to the uniform generator) is in this case an isomorphic extension (see Section 2.3 for the definitions). 
 
 \begin{cor}
 Any $C^\infty$ smooth aperiodic flow admits an isomorphic symbolic extension.
 \end{cor}

For the time $t$-map $\phi_t$, $t\neq 0$, of a topological flow $(X,\Phi)$ we define a weak notion of uniform generators as follows. For $\alpha>0$ a partition $P$ is said to be \emph{an $\alpha$-uniform generator of $\phi_t$} when 
$\sup_{y\in P_{\phi_t}^{[-n,n]}(x)}d\left(y, \phi_{[-\alpha,\alpha]}(x)\right)$ with $\phi_{[-\alpha,\alpha]}(x)=\{\phi_s(x), \ |s|\leq \alpha\}$ goes to zero uniformly in $x\in X$. 

\begin{theo}\label{deux}
Let $(X,\Phi)$ be a regular aperiodic flow  admitting a uniform generator. Then for any $t\neq 0$ small enough (depending only on the topological entropy of $\Phi$) and  for any $\alpha>0$ there  is a $3$-partition of a Borel global cross-section such that the associated partition of $X$ in towers  is an $\alpha$-uniform generator of $\phi_t$. 
\end{theo}



\section{Zero-dimensional suspension flows as models}
Following R.Bowen and P.Walters we  build from  cross-sections an extension  given by a suspension flow over a zero-dimensional system. When the  cross-sections have small flow boundaries, this extension is isomorphic.
\subsection{Generalities on topological flows}
A pair $(X,\Phi)$ is called a \textit{topological flow}, when $(X,d)$ is a compact metric space and $\Phi=X\times \mathbb{R} \rightarrow X$ is a continuous flow on $X$, i.e. $\Phi$ is continuous, $x\mapsto \Phi(x,0)$ is the identity map on $X$ and $\Phi( \Phi(x,t),s)=\Phi(x,t+s)$ for all $t,s\in \mathbb{R}, \,x\in X$. For $t\in \mathbb{R}$ we let $\phi_t$ be the homeomorphism of $X$ given by $x\mapsto \Phi(x,t)$ and we will denote the flow by $\Phi=(\phi_t)_{t\in \mathbb{R}}$.  The flow is said to be \textit{singular} when there is (at least) a  point $x\in X$
fixed by the flow, i.e. $\phi_t(x)=x$ for all  $t$. Otherwise the flow is said to be \textit{regular}.  In the present paper the flow is always assumed to be regular.

\subsubsection{Cross-sections}
 Following the pioneering works of Poincar\'e, we consider the return maps to cross-sections in order  to study the flow.

\begin{defi}
 Let $(X,\Phi)$ be a topological flow. A \textit{cross-section} $S$ of time $\eta>0$  is a  subset $S$ of $X$ such that the restriction of  $\Phi:(x,t)\mapsto \phi_t(x)$ to  $S\times [-\eta,\eta]$ is one-to-one. 
\end{defi}
Any  subset of a cross-section is itself a cross-section. The cross-section $S$ is  \textit{global} when there is $\xi>0$ with $\Phi(S\times [-\xi,\xi])=X$. Obviously any cross-section has an empty interior. Moreover any Borel cross-section has zero measure for any probability Borel measure invariant by the flow. \\

For an interval $I$ of $\mathbb{R}$ and a subset $E$ of $X$ we  denote by $\phi_I(E)$ the subset $\Phi(E\times I)=\{\phi_t(x), \ t\in I \text{ and }x\in E\}$.  Let $S$ be a cross-section of time $\eta$. For $0<\zeta\leq \eta$   the set $S_\zeta:=\phi_{[-\zeta, \zeta]}(S)$ is called the \textit{$\zeta$-cylinder} associated to $S$. For a subset $E$ of $X$ we denote the interior of $E$   by $\Int (E)$, its closure by $\overline{E}$  and its boundary by $\partial E$.   A cross-section $S$ of time $\eta$ is said \textit{weakly extendable}  when its closure $\overline{S}$ is  itself a  cross-section of time $\eta$. To shorten the notations we will write \textit{w.e.c.} for  weakly extendable cross-section. Again any subset of a w.e.c. is itself a w.e.c. and any closed cross-section is obviously a w.e.c.. We recall below a notion of interior  adapted to w.e.c.'s (introduced in \cite{BW} for closed cross-sections).  

\begin{defi}
Let $S$ be a w.e.c. of time $\eta$ and let $0<\zeta\leq \eta$. The \textit{flow interior} $\Int^\Phi(S)$  of $S$ is defined as follows 
\[\Int^\Phi (S):=\Int(S_{\zeta})\cap S.\]
We also define  the flow boundary of $S$ as $\partial^\Phi S=\overline{S}\setminus \Int^\Phi (S)$ and the \textit{flow boundary of the cylinder } $S_{\zeta}$ as $\partial^\Phi S_\zeta:=\phi_{[-\zeta, \zeta]}(\partial^\Phi S)$.  
\end{defi}

The flow boundary $\partial^{\Phi}S$ of a w.e.c. is always closed since  it may be written as $\partial^{\Phi}S=\overline{S}\setminus \Int(S_{\zeta})$.  The definition of $\Int^\Phi S$ does not depend on $0<\zeta\leq \eta$.  Indeed if $x\in\Int (S_\eta)\setminus \Int (S_\zeta)$ for some $0<\zeta< \eta$  then there is a sequence $(x_n)_n$ in the complement of $S_\zeta$ converging to $x$. As $x$ belongs to $\Int (S_\eta)$ so does $x_n$ for $n$ large enough, in particular $x_n=\phi_{t_n}(y_n)$ for some  $y_n\in S$ and $t_n\in ]\zeta,  \eta]$. By extracting subsequences there exist $y\in \overline{S}$ and $t\in [\zeta, \eta]$ with $x=\phi_t(y)$. Therefore, $x$ does not lie in $S$, because  the closure  $\overline{S}$ of $S$ is a cross-section of time $\eta$. \\

 Consider a flow associated to a smooth  nonvanishing vector field $\mathsf X$ on a compact $(d+1)$-manifold $M$. Then any embedded $d$-disc  transverse to $\mathsf X$ defines a closed cross-section. We denote respectively  by $\mathsf B_d $ and $\mathsf S_d$ the unit ball and the unit sphere in $\mathbb{R}^d$ for the Euclidean norm. If we let $h:\mathsf B_d\rightarrow M $ denote a smooth embedding  satisfying $h(\mathsf B_d)=S$ then $\partial^\Phi S$ is the set $h(\mathsf S_d)$. \\
 
  For a topological flow $(X,\Phi)$ any point belongs to the flow interior of a closed cross-section with arbitrarily small diameter (see \cite{Wi} p. 270).

\subsubsection{Properties of the flow boundary}

The interior boundary of a w.e.c. may be characterized as follows :

\begin{lemma}\label{comp} Let $S$ be a w.e.c. of time $\eta$. The restriction of $\Phi$ to $\Int^\Phi (S)\times ]-\eta,\eta[$ defines a homeomorphism onto $\Int(S_\eta)$. In particular $\Int^\Phi (\Int^\Phi (S))=\Int^\Phi (S)$. Moreover the boundary $\partial S_\zeta$ of $S_\zeta$ (in $X$) is the union of $\partial^\Phi S_\zeta$, $\phi_{-\zeta}(\overline{S})$ and $\phi_{\zeta}(\overline{S})$ for $0\leq  \zeta\leq \eta$. 
\end{lemma}

\begin{proof}As the restriction of $\Phi$ to $\overline{S}\times [-\eta,\eta]$ is a homeomorphism onto its image it is enough  to check $\Phi(\Int^\Phi (S)\times ]-\eta,\eta[)=\Int(S_\eta)$.

Take $x\in \Int(S_\eta)$. Let $y\in S$ with $\phi_{\zeta}(y)=x$ for some $\zeta$  with $|\zeta|\leq \eta$. Necessarily $|\zeta|<\eta$. If $\zeta=\eta$ we would have $\phi_t(y)\in \phi_{[-\eta,0]}S\subset  S_{\eta}$  for $t>\eta$ close to $\eta$. Indeed for such $t$ we have $\phi_t(y)\notin \phi_{[0,\eta]}S$ by injectivity of  $\Phi$   on $S\times [0,2\eta]$.
Then any limit of $\phi_t(y)$ when $t$ goes to $\eta$ should belong to $\phi_{[-\eta,0]}\overline{S}$ but by continuity of the flow such a limit is necessarily equal to $\phi_{\eta}(y)$ contradicting the injectivity of $\Phi$ on $\overline{S}\times [-\eta, \eta]$. We argue similarly for $\zeta=-\eta$. It is thus enough to show $y\in \Int^{\phi}(S)$. Without loss of generality we can assume $\zeta<0$. Then we have $y\in S\cap \phi_{-\zeta}(\Int(S_\eta))=S\cap \Int(\phi_{-\zeta}(S_\eta))= S\cap \Int (S_{\eta+\zeta})$ where the last equality follows again from the injectivity of $\Phi$ on $S\times [0,2\eta]$.

 Conversely let $y\in \Int^\Phi (S)$ and $\zeta \in  ]-\eta,\eta[$. We can assume $\zeta\geq 0$.  Take $0<\xi<\eta-\zeta$. 
By definition of the flow interior the point $y$ belongs to $\Int(S_\xi)$ so that $x=\phi_{\zeta}(y)$ is in $\phi_\zeta(\Int(S_\xi))\subset \Int(S_\eta)$. Thus $\Phi(\Int^\Phi (S)\times ]-\eta,\eta[)=\Int(S_\eta)$. In particular 
$\Int((\Int^{\Phi}S)_\eta)= \Int (S_\eta)$, and  by taking the intersection with $\Int^{\Phi}(S)$ on both sides we get $\Int^\Phi (\Int^\Phi (S))=\Int^\Phi (S)$.

Finally we have for $0\leq  \zeta\leq \eta$ \begin{eqnarray*}
\partial S_\zeta& =& \overline{S_\zeta}\setminus \Int(S_\zeta),\\
&=&\Phi\left(\left(\overline{S}\times [-\zeta,\zeta]\right)\setminus \left(\Int^\Phi (S)\times ]-\zeta,\zeta[\right) \right),\\
&=&\partial^\Phi S_\zeta\cup \phi_{-\zeta}(\overline{S})\cup \phi_{\zeta}(\overline{S}).
\end{eqnarray*}

\end{proof}

We give  below some topological properties of the flow boundary and the flow interior of a subset $A$ of a w.e.c. $B$ with respect to   the induced topology on $B$.

\begin{lemma}\label{zer}Let $B$ be a w.e.c. of time $\eta$ and let $A\subset B$.
\begin{enumerate}
 \item  $\Int^\Phi( A)$ is open in $B$,
 \item  $\partial_B A\subset\partial^{\Phi}A\subset \partial_BA\cup \partial^{\Phi}B$ with $\partial_BA$ being the frontier of $A$ in $B$.
\end{enumerate} 
\end{lemma}
\begin{proof} 
\begin{enumerate}
\item  By definition  we have $\Int^{\Phi}A=A \cap \Int(A_{\eta})$. As $B$ is a cross-section of time $\eta$ containing $A$, then $\Int^{\Phi}(A)$  coincides with $B\cap \Int(A_{\eta})$. The set $\Int(A_{\eta})$ being open in $X$, the set $\Int^{\Phi}(A)$ is open in $B$ for the induced topology. 
\item The first inclusion follows directly from (1). We  show  the second one. Let $x\in \partial^{\Phi}A\setminus \partial^{\Phi}B\subset \Int^{\Phi}(B)$. If $x$ did not belong $\partial_BA$ then there would be an open subset $x\in O\subset \Int^{\Phi}(B)$ of $B$  with either $O\cap A=\emptyset$ or $O\subset A$. By Lemma \ref{comp} the set $\Phi_{]-\eta,\eta[}O$ is an open neighborhood of $x$ in $X$. In the first case this open neighborhood lies in the complement of $A$ contradicting  $x\in \overline{A}$, whereas in the second case it lies in the interior of $A_\eta$ contradicting $x\notin \Int^{\Phi}(A)$.  
\end{enumerate}
\end{proof}

The flow boundary behaves with respect to intersection, union and complement in a similar way to the usual boundary. 
\begin{lemma}\label{ter}Let $A$ and $B$ be w.e.c.'s of time $\eta$. 
\begin{enumerate}
\item  When $A\cup B$ defines a w.e.c. of time $\eta$, we have  $$\partial^{\Phi}(A\cup B)\subset \partial^{\Phi}A\cup  \partial^{\Phi}B.$$ If $A$ and $B$ are disjoint and closed, then the equality holds. 
\item $$\partial^{\Phi}(A\cap B)\subset \partial^{\Phi}A\cup\partial^{\Phi}B.$$
\item  $$\partial^{\Phi}(B\setminus A)\subset \partial^{\Phi}B\cup \partial^{\Phi}A.$$
\end{enumerate}
\end{lemma}
\begin{proof}\begin{enumerate}
\item The inclusion $\Int^{\Phi}(A)\cup\Int^{\Phi}(B)\subset \Int^{\Phi}(A\cup B)$ follows clearly from $ \Int(A_\eta)\cup\Int (B_\eta)\subset\Int( (A\cup B)_\eta)$. Then $\partial^{\Phi}(A\cup B)=\overline{A\cup B}\setminus \Int^{\Phi}(A\cup B) \subset \left(\overline{A}\setminus \Int^{\Phi}(A)\right)\cup\left(\overline{B} \setminus \Int^{\Phi}B\right)$.
When $A$ and $B$ are closed and disjoint,  the cylinder $(A\cup B)_\eta$ is the disjoint union of the closed cylinders $A_\eta$ and $B_\eta$. Thus  we have $\Int( (A\cup B)_\eta)=\Int(A_\eta)\cup\Int (B_\eta)$ and this easily implies the required equalities. 
\item Using cylinders as above we get easily   $\Int^{\Phi}(A\cap B)\subset \Int^{\Phi}(A)\cap \Int^{\Phi}(B)$ so that $\partial^{\Phi}(A\cap B)=\overline{A\cap B}\setminus \Int^{\Phi}(A\cap B) \subset \left(\overline{A}\setminus \Int^{\Phi}(A)\right)\cap\left(\overline{B} \setminus \Int^{\Phi}(B)\right)$.
\item Let $x\in \partial^{\Phi}(B\setminus A)\setminus \partial^{\Phi}B\subset \Int^{\Phi}(B)$.
As $\Int^\Phi (A)\subset A$ is open in $\overline{B}$ then $\overline{B\setminus A}\cap \Int^\Phi (A)=\emptyset$ and  $x\notin \Int^\Phi (A)$. 
If $x$ belongs to $\overline{A}$ then $x\in \partial^{\Phi}A$. If not,  $x$ would belong to $\Int^{\Phi}(B)\setminus \overline{A}$ which is open in $\Int^{\Phi}(B)$. In particular $\phi_{]-\eta,\eta[}(\Int^{\Phi}(B)\setminus \overline{A})$ is open in $X$ according to Lemma \ref{comp}. But this last open set contains $x$ and it is a subset of  $(B\setminus A)_\eta$. Therefore $x$ should belong to $\Int^{\Phi}(B\setminus A)$. Contradiction.

\end{enumerate}

\end{proof}

\subsubsection{Closed  cross-sections}
We focus in this subsection on closed cross-sections and especially on global closed cross-sections. 
\begin{lemma}\label{closed}
\begin{enumerate} Let $S$ be a closed cross-section of time $\eta$.
\item The cylinder $S_\eta$ is closed, therefore $\partial S_\eta\supset \partial^{\Phi}S_\eta$ have  an empty interior.
\item $\Int^{\Phi}(\partial^{\Phi}S)=\emptyset$. In particular $\partial^{\Phi}\partial^{\Phi}\partial^{\Phi}=\partial^{\Phi}\partial^{\Phi}$. 
\item When $S$ is global, then $\overline{\Int^{\Phi}(S)}$ is also a global closed cross-section.
\end{enumerate}
\end{lemma}

\begin{proof}
(1) and (2) follow easily from the  definitions. Let us check (3). By Lemma \ref{comp} the restriction of $\Phi$ to 
$\Int^{\Phi}(S)\times ]k\eta, (k+1)\eta[$ is an homeomorphism onto $\Int(\phi_{]k\eta, (k+1)\eta[}S)$ for any integer $k$. But $S$ being global we have $X=\Phi(S \times [-K\eta, K\eta])$ for some $K\in \mathbb{N}$. Thus the open set 
$$\bigcup_{  k=-K,\cdots, K-1}\Phi(\Int^{\Phi}(S)\times ]k\eta, (k+1)\eta[)= \bigcup_{   k=-K,\cdots, K-1}\Int(\phi_{]k\eta, (k+1)\eta[}S)$$ is contained in $\Phi(\Int^{\Phi}(S) \times [-K\eta, K\eta])$ and is dense in $X$ by (1). Therefore 
$\Phi\left(\overline{\Int^{\Phi}(S)} \times [-K\eta, K\eta]\right)=\overline{\Phi(\Int^{\Phi}(S) \times [-K\eta, K\eta])}=X$.
\end{proof}

 Now we consider a global closed cross-section $S$ of time $\eta$ and we let $\xi>0$ with $\Phi(S\times [-\xi,\xi])=X$. The first return time $t_S$ in  $S$ defines a   lower semicontinuous positive  function  as the  cross-section $S$  is closed. Moreover $t_S$ is bounded from above by $2\xi$ and from below by $2\eta$.  Let $\mathcal{C}_S\subset S$ be the (residual) subset of continuity points of $t_S$. The first return map in $ S$, denoted by $T_{S}: S\rightarrow  S$, $x\mapsto \phi_{t_S(x)}(x)$,  is also continuous at any point of $\mathcal{C}_{S}$.  In fact we may describe more precisely the continuity properties of $t_S$  and $T_S$. 

\begin{lemma}\label{piece}
The first return time  $t_S$ in $S$ is a piecewise continuous map, i.e. there is  a finite partition  $(C_k)_k$ of  $S$  into   w.e.c.'s such that   $t_S$ is uniformly continuous on each $C_k$. Moreover the boundaries in $S$ of the $C_k$'s have an empty interior in $S$. 
\end{lemma}
\begin{proof}

The set $Y$ defined by $Y:= \{(x,t)\in S\times \mathbb{R}^+, \ \phi_{t}(x)\in S \}$ is a closed subset of $S\times \mathbb{R}^+$. Let $\delta\in ]0,\inf_{x\in S}t_S(x)[$. 
For any positive integer $k$ the closed intersection $Y\cap\left(S\times [k\delta, (k+1)\delta]\right)$ is the  
graph of a continuous nonnegative  function defined on a closed subset  $B_k$ of $\mathfrak S$. Let denote this function by $f_k:B_k\rightarrow \mathbb{R}^+$. Then the return time $t_S$ coincides with $f_1$ on $C_1:=B_1$ (which may be the empty set) 
and   with   $f_k$ on $C_k:=B_k\setminus \left(\bigcup_{l<k}B_l\right)$ for  every $k>1$. Moreover observe that $
\bigcup_{1\leq k\leq K }C_k=\bigcup_{1\leq k\leq K }B_k=S$ with $K=\lceil\frac{2\xi}{\delta}\rceil$.  The $B_k$'s  being closed, the  boundaries of the $C_k$'s in $S$ have an empty interior in $S$ (indeed the class of subsets, whose boundary has an empty interior, is closed under complement, finite unions and intersections and it contains the closed subsets).
\end{proof}

The set $\mathcal{C}_S$ of continuity points is not only residual, it contains the  open and dense subset of $S$ given   by the union of the interior sets in $S$ of the w.e.c.'s  $C_k$ by Lemma \ref{piece}. We relate below the set of discontinuity points of $t_S$ with the flow boundary of $S$. 

\begin{lemma}\label{firt}
Let $x\in S\setminus \mathcal{C}_S$. Then there exists $t\in [0,2\xi]$  with $\phi_t(x)\in \partial^\Phi S$.
\end{lemma}
\begin{proof}
Assume by contradiction that there is no $t\in [0,2\xi]$ with $\phi_t(x)\in \partial^\Phi S$. Let us show $x$ belongs to $\mathcal{C}_S$. If not there would be a   sequence $(x_n)_n$  of $S$ converging to $x$ with $2\xi\geq \lim_nt_S(x_n)> t_S(x)$. When $n$ is large enough,  $\phi_{t_S(x)}(x_n)$  belongs to the complement of  $\Int(S_\zeta)$ for some small  $\zeta>0$  and thus so does $\phi_{t_S(x)}(x)$, in particular $\phi_{t_S(x)}(x)\in \partial^\Phi S$ contradicting  our  hypothesis.
\end{proof}

\subsubsection{Complete family of closed  cross-sections}
A finite family $\mathcal{S}$ of disjoint closed  cross-sections  $S$ of time $\eta_\mathcal{S}>0$ is said \textit{complete} when  the  cylinders $S_{\eta_\mathcal{S}/2}$ are covering  $X$. In particular the set $\mathfrak{S}=\bigcup_{S\in \mathcal{S}}S$ defines a global  closed cross-section and the first return time $t_\mathfrak{S}$ is bounded from above by   $\eta_\mathcal{S}$. The \textit{diameter} of such a family $\mathcal{S}$ is the maximum of the diameters of $S\in \mathcal{S}$. Any topological regular flow admits a complete family of closed  cross-sections with arbitrarily small diameter (see Lemma 7 in \cite{BW}). 
  Let us consider such a complete family $\mathcal{S}$ of  cross-sections. For the closed global cross-section $\mathfrak S$, the conclusion of Lemma \ref{piece} holds for the partition $T_{\mathfrak{S}}^{-1}\mathcal{S}$ :

\begin{lemma}\label{osc}
 For every $S\in \mathcal{S}$  the first return time $t_\mathfrak{S}$ is  uniformly continuous on the set $T_{\mathfrak{S}}^{-1}S$.
\end{lemma}

\begin{proof}
We argue by contradiction. Let $(x_n)_n$ and $(y_n)_n$ be two sequences in $ T_{\mathfrak{S}}^{-1}S$ with $\lim_nd(x_n,y_n)=0$ and $t_1:=\lim_n t_\mathfrak{S}(x_n)> t_2:=\lim_nt_\mathfrak{S}(y_n)>0$. By extracting subsequences we may assume $(x_n)_n$ (and thus $(y_n)_n$) is converging in $\mathfrak{S}$, say to $x$. Then $\phi_{t_1}(x)$ and $\phi_{t_2}(x)$ both belong to $S$. But  we have also $t_1-t_2< \sup_{y\in \mathfrak{S}}t_\mathfrak{S}(y)\leq \eta_\mathcal{S}$.  This contradicts the fact that $S$ is a cross-section of time $\eta_\mathcal{S}$.
\end{proof}

By Lemma \ref{ter} (1) we have $\partial^{\Phi}\mathfrak{S}=\bigcup_{S\in \mathcal{S}}\partial^{\Phi}S$. 
The flow boundary of $T_{\mathfrak{S}}^{-1}S$, for $S\in \mathcal{S}$, satisfies the following property  :
\begin{lemma}\label{tour}
 For every $S\in \mathcal{S}$ we have 
 $$\partial^{\Phi}(T_{\mathfrak{S}}^{-1}S)\subset \partial^{\Phi}\mathfrak{S}_{\eta_\mathcal{S}}\cup T_{\mathfrak{S}}^{-1}(\partial^{\Phi}S).$$
\end{lemma}
\begin{proof}
It is enough to show $\Int^{\Phi}(T_{\mathfrak{S}}^{-1}S)\supset T_{\mathfrak{S}}^{-1}(\Int^{\Phi}(S)) \cap \Int^{\Phi}(\mathfrak{S})\cap \mathcal{C}_\mathfrak{S}$. Indeed this implies 
\begin{eqnarray*}
\partial^{\Phi}(T_{\mathfrak{S}}^{-1}S)&=& \overline{T_{\mathfrak{S}}^{-1}S}\setminus \Int^{\Phi}(T_{\mathfrak{S}}^{-1}S),\\
&\subset & \left(\overline{T_{\mathfrak{S}}^{-1}S}\setminus T_{\mathfrak{S}}^{-1}S \right)\cup 
T_{\mathfrak{S}}^{-1}(\partial^{\Phi}S)\cup \partial^{\Phi}\mathfrak S \cup \left(\mathfrak S\setminus \mathcal{C}_\mathfrak{S}\right),
\end{eqnarray*}
but the set $\overline{T_{\mathfrak{S}}^{-1}S}\setminus T_{\mathfrak{S}}^{-1}S $  is contained in $\mathfrak S\setminus \mathcal{C}_\mathfrak{S}$, which by Lemma \ref{firt} is a subset of $\partial^{\Phi}\mathfrak{S}_{\eta_\mathcal{S}}$.

 Let $x\in T_{\mathfrak{S}}^{-1}(\Int^{\Phi}(S)) \cap \Int^{\Phi}(\mathfrak{S})\cap \mathcal{C}_\mathfrak{S}$. Then $T_{\mathfrak{S}}(x)=\phi_{t_\mathfrak{S}(x)}(x)$ lies in $\Int (S_{\zeta})$ for any small $\zeta>0$. By continuity of  the flow we have also $\phi_{t_\mathfrak{S}(x)}(y)\in \Int (S_{\zeta})$ for $y\in \mathfrak S$ close enough to $x$. Therefore such points $y$ return in $S$ in a time close to $t_\mathfrak{S}(x)$. As $x$ belongs to $\mathcal{C}_\mathfrak{S}$ this correspond to their first return time. But  $x$ also belongs to $\Int^{\Phi}(\mathfrak S)$ so that  there is an open subset $x\in O \subset \Int^{\Phi}(\mathfrak S)$  of $\mathfrak S$ contained in $T_{\mathfrak{S}}^{-1}S$. Therefore  $x$ belongs to  $\Int^{\Phi}(T_{\mathfrak{S}}^{-1}S)$ by Lemma \ref{comp}.
\end{proof}

\subsubsection{Suspension flows}\label{sssup}
Let $(X,T)$ be a topological discrete system. Let $r:X\rightarrow \mathbb{R}^+$ be  a positive continuous function. Consider the quotient space
$${\displaystyle X_r=\{(x,t)\, : \, 0 \leq t\leq r(x),\, x\in X \text{ and }  (x,r(x))\sim(Tx,0) \}}.$$
This quotient space $X_r$ is compact and metrizable (see Section 4 in \cite{BW}). The\textit{ suspension flow} over  $(X,T)$ under the  roof function $r$ is the flow $\Phi_r$ on $ X_r$ 
 induced by the time translation $T_t$ on $X\times \mathbb{R}$  defined by $T_t(x,s)= (x,s+t)$.

We call  \textit{Poincar\'e cross-section} of a topological flow $(X,\Phi)$ any closed cross-section $S$  such that the flow map $\Phi:(x,t)\mapsto\phi_t(x)$ is a surjective local homeomorphism from $S\times \mathbb{R}$ to $X$.

\begin{lemma}\label{poinc}
A cross-section  is a Poincar\'e cross-section if and only if it is a global closed cross-section with empty flow boundary.
\end{lemma}

\begin{proof}
Let $S$ be a Poincar\'e cross-section. By compactness of $X=\Phi(S\times \mathbb{R})=\bigcup_{\zeta>0}\Phi(S\times ]-\zeta, \zeta[)$ there is $\xi>0$ with $\Phi(S\times [-\xi,\xi])=X$ and $S$ is thus global. The sets $\phi_{]-\zeta, \zeta[}(S)$ for $\zeta>0$ being open, we have $S=\Int^{\Phi}S$ and thus $\partial^{\Phi}S=\emptyset$ as $S$ is closed.
Conversely, if $S$ is a  closed cross-section with empty flow boundary, then from Lemma \ref{comp} the flow map  $\Phi:S\times \mathbb{R}\rightarrow X$ defines a local homeomorphism onto its image. When the cross-section $S$ is moreover global, this map is then also surjective.
\end{proof}

As the roof function $r$ of the suspension flow $(X_r,\Phi_r)$ does not vanish, the flow is regular and the subset $X\times  \{0\}\subset X_r$ defines a  Poincar\'e cross-section of the suspension flow $X_r$. In fact any topological flow admits  a  Poincar\'e cross-section  if and only if it is topologically conjugate to a suspension  regular flow. When the flow space is one-dimensional, there always exists a  Poincar\'e cross-section. Indeed any closed cross-section is zero-dimensional, so that  any point belongs to a closed   cross-section with empty flow boundary. Bowen-Walters construction then provides a complete family $\mathcal{S}$ of closed cross-sections with empty flow boundary. The union $\mathfrak{S}=\bigcup_{S\in \mathcal{S}}S$ defines therefore in this case a  Poincar\'e cross-section.  Here we consider topological suspension flows, but we may also define similarly  a Borel (resp. ergodic) suspension flow  over a discrete Borel (resp. ergodic) system with a bounded Borel (resp. integrable) roof function. \\


For the suspension flow $(X_r,\Phi_r)$  the $\Phi_r$-invariant measures are related with the $T$-invariant measures as follows. For a discrete topological system $(X,T)$ (resp. topological flow $(X,\Phi)$) we denote by $\mathcal{M}(X,T)$ (resp. $\mathcal{M}(X,\Phi)$) the set of $T$-invariant (resp. $\Phi$-invariant) Borel probability measures. Let  $\lambda$ denote the Lebesgue measure on $\mathbb{R}$. For any $\mu\in \mathcal{M}(X,T)$ the product measure $\mu\times \lambda$ induces a finite $\Phi_r$-invariant measure on $X_r$. This defines a homeomorphism  between $\mathcal{M}(X_r, \Phi_r)$ and $\mathcal{M}(X, T)$. More precisely  the map
\begin{eqnarray*}
\Theta: & \mathcal{M}(X, T) \rightarrow & \mathcal{M}(X_r, \Phi_r),\\
& \mu  \mapsto & \frac{(\mu\times \lambda)|_{X_r}}{ \int r \, d\mu}
\end{eqnarray*}
is a homeomorphism (not affine in general), which preserves  ergodicity. 

\subsubsection{$\Phi$-invariant  and $\phi_t$-invariant measures}
For any $t>0$ we let $i_t$ be the inclusion of $\mathcal{M}(X,\Phi)$ in $\mathcal{M}(X,\phi_t)$. In general this inclusion does not preserve ergodicity : for an ergodic $\mu\in \mathcal{M}(X,\Phi)$, the measure $i_t(\mu)$ need not be ergodic (however, by the spectral theory, for a fixed measure $\mu$, this may occur for at most countably many $t\in \mathbb{R}$). 
  
  For $t>0$ the map
\begin{eqnarray*}
\theta_t: & \mathcal{M}(X,\phi_t)\rightarrow &\mathcal{M}(X,\Phi),\\
& \mu\mapsto & \frac{1}{t}\int_{0}^t\phi_s\mu\,  ds
\end{eqnarray*}  
 defines  a continuous affine map of $\mathcal{M}(X,\phi_t)$ onto $\mathcal{M}(X,\Phi)$, which is a retraction i.e. $\theta_t \circ i_t=\Id_{\mathcal{M}(X,\Phi)}$. As this last identity is immediate, we only check that $\theta_t(\mu)$ belongs to  $\mathcal{M}(X,\Phi)$ for any $\mu\in \mathcal{M}(X,\phi_t)$. Clearly it is enough to show $\phi_u\left(\theta_t(\mu)\right)=\theta_t(\mu)$ for any $u\in ]0,t[$. This follows from the following equalities :

\begin{align*}
 \phi_u\left( \frac{1}{t}\int_{0}^t\phi_s\mu\,  ds\right)&= \frac{1}{t}\int_{0}^t\phi_{s+u}\mu\,  ds,\\
 &=\frac{1}{t}\int_{u}^{t+u}\phi_s\mu\,  ds,\\
&=\frac{1}{t}\left(\int_{u}^{t}\phi_s\mu\,  ds+\int_{0}^{u}\phi_{s+t}\mu\,  ds\right)=\frac{1}{t}\int_{0}^{t}\phi_s\mu\,  ds.
\end{align*}

\subsubsection{Orbit equivalence} 

Two topological flows $(X,\Phi)$ and $(Y,\Psi)$ are \textit{orbit equivalent} when 
there is a homeomorphism $\Lambda$ from $X$ onto $Y$ mapping $\Phi$-orbits to $\Psi$-orbits, preserving their orientation. Any flow obtained by a change of the time scale of a topological flow $(X,\Phi)$ is orbit equivalent to $(X,\Phi)$. 

In the following we are interested in dynamical properties invariant under orbit equivalence. In general the topological entropy is not preserved by  orbit equivalence. But for regular topological flows zero and infinite entropy are invariant \cite{ohn}.

\subsection{The small  boundary property for flows}
\subsubsection{Definitions}
For a topological discrete system $(X,T)$ (resp. topological flow $(X,\Phi)$), a subset $E$ has a 
\textit{ small boundary}  when its boundary is a \textit{null set}, i.e. it  has zero measure for any $T$-invariant (resp. $\Phi$-invariant) Borel probability measure (similarly a Borel subset is said to be a \textit{full set} when its complement is a null set).  \\

We define now an adapted notion of small boundary for w.e.c.'s. 

\begin{defi}
Let $(X,\Phi)$ be a topological flow. A  w.e.c. $S$  of time $\eta$   has a \textit{small flow boundary } when $\partial^\Phi S_{\eta}$  is a null set. 
\end{defi}

 For a w.e.c. $S$, the closure $\overline{S}$ is also a cross-section and it is thus transverse to the flow, so that the subset $\phi_{-\eta}\overline{S}\cup  \phi_{\eta}\overline{S}$  of $\partial S_\eta$ has zero measure  for any $\Phi$-invariant Borel probability measure. Therefore in the above definition we may replace the  flow boundary $\partial^{\Phi} S_{\eta}$ of $S_\eta$ by its usual boundary $\partial S_{\eta}$ according to Lemma \ref{comp}.  Moreover the small flow  boundary property of  $S$ does not depend of the time $\eta$.

\begin{lemma}\label{sss}Let $(X,\Phi)$ be a \textit{topological flow} and let $S$ be a  w.e.c.  of time $\eta>0$. The following properties are equivalent :
\begin{enumerate}[i)]
\item  $S$ has a small flow boundary,  
\item $\lim_{\zeta\rightarrow 0} \frac{\mu(\partial^{\Phi} S_\zeta)}{\zeta}=0$ for any  $\mu\in \mathcal{M}(X,\Phi)$,
\item $\frac{1}{\mathsf T}\sharp \{0<t<\mathsf T, \ \phi_t(x)\in \partial^{\Phi} S \}\xrightarrow{T\rightarrow +\infty}0$  uniformly in $x\in X$.
\end{enumerate}
\end{lemma}
\begin{proof}
We have trivially $i) \Rightarrow ii)$. Assume $ii)$ and let us prove $i)$. For a $\Phi$-invariant Borel probability measure $\mu$, we have for all positive integers $n$ and for all $\eta>0$:
\begin{align*}
\mu(\partial^{\Phi} S_{\eta})&=\mu\left(\phi_{[-\eta,\eta]}\partial^{\Phi}S\right),\\
&=\mu\left(\bigcup_{k=-n}^{n-1}\phi_{[k\eta /n,(k+1)\eta/n]}(\partial^{\Phi}S )\right),\\
&=2n\times \mu(\partial^{\Phi} S_{\eta/2n})=\eta\times \frac{\mu(\partial^{\Phi} S_{\eta/2n})}{\eta/2n}\xrightarrow{n}0.
\end{align*}
Finally we show $ii)$ and $iii)$ are equivalent. We recall that $\lambda$ denotes the Lebesgue measure on $\mathbb{R}$.  By Birkhof ergodic theorem we have for $0<\zeta<\eta/2$ and for any ergodic $\Phi$-invariant Borel probability measure $\mu$ 
 \begin{align*}
\forall \mu \text{ a.e. }x, \ \ \  \mu(\partial^{\Phi} S_\zeta)&=\lim_{\mathsf T\rightarrow +\infty}\frac{1}{\mathsf T}\lambda \left( \{0<t<\mathsf T, \ \phi_t(x)\in \partial^{\Phi} S_\zeta\}\right),\\
& = \lim_{\mathsf T\rightarrow +\infty}\frac{2\zeta}{\mathsf T}\sharp \{0<t<\mathsf T, \ \phi_t(x)\in \partial^{\Phi} S \},\\
\text{ thus }\mu(\partial^{\Phi} S_\zeta)&\leq  \limsup_{\mathsf T\rightarrow +\infty}\frac{2\zeta}{\mathsf T}\sup_{x\in X}\sharp \{0<t<\mathsf T, \ \phi_t(x)\in \partial^{\Phi} S \}. \end{align*}
By the ergodic decomposition this last inequality holds in fact for any $\Phi$-invariant Borel probability measure and therefore $iii) \Rightarrow ii)$. 

We will use a Krylov-Bogolyubov's like argument to  show  $ii)\Rightarrow iii)$. For any $\mathsf T>0$ take $x_{\mathsf T}\in X$ maximizing the function $x\mapsto \sharp \{0<t<\mathsf T, \ \phi_t(x)\in \partial^{\Phi} S \}$ on $X$. Let $\psi_{\mathsf T}:\mathbb{R}\rightarrow M$, $t\mapsto \phi_t(x_{\mathsf T})$ and let $\mu_{\mathsf T}:=\psi_{\mathsf T}(\lambda_{[0,\mathsf T]})$ with $\lambda_{[0,\mathsf T]}:=\frac{\lambda(\cdot \cap [0,\mathsf T])}{\mathsf T}$. 
For any $0<\zeta<\eta/2$ we have
\begin{align*}\mu_{\mathsf T}(\partial^{\Phi} S_\zeta)&=\frac{\lambda \left( \{0<t<\mathsf T, \ \phi_t(x_{\mathsf T})\in \partial^{\Phi} S_\zeta\}\right)}{\mathsf T},\\
&\geq \frac{2\zeta}{\mathsf T}\left( \sharp \{0<t<\mathsf T, \ \phi_t(x_T)\in \partial^{\Phi} S \}-1\right),\\
&\geq \frac{2\zeta}{\mathsf T}\left( \sup_{x\in X}\sharp \{0<t<\mathsf T, \ \phi_t(x)\in \partial^{\Phi} S \}-1\right).
\end{align*}
As $\partial^{\Phi} S_\zeta$ is closed, any weak-$*$ limit $\mu$ of $(\mu_{\mathsf T})_{\mathsf T}$, when $\mathsf T$ goes to infinity, satisfies 
\begin{align*}
\mu(\partial^{\Phi} S_\zeta)&\geq \limsup_{\mathsf T}\mu_{\mathsf T}(\partial^{\Phi} S_\zeta),\\
&\geq 2\zeta \times \limsup_{\mathsf T}\frac{\sup_{x\in X}\sharp \{0<t<\mathsf T, \ \phi_t(x)\in \partial^{\Phi} S\}}{\mathsf T}.
\end{align*}
Thus if $\frac{\mu(\partial^{\Phi} S_\zeta)}{\zeta}\xrightarrow{\zeta\rightarrow 0}0$ we get 
$iii)$.
\end{proof}

 A discrete topological system  (resp. topological flow) is said to have the \textit{small boundary property} when there is a basis of neighborhoods with small boundary. We will consider the following corresponding notion for the flow boundary. 

\begin{defi}\label{dedef}
A topological flow $(X,\Phi)$ is said to have the \textit{small flow boundary property} when for any $x\in X$ and for any w.e.c. $S'$ with $x\in \Int^\Phi (S')$  there exists   a subset $S$ of $ S'$ with $x\in \Int^\Phi (S)$ such that the w.e.c. $S$ has a small flow boundary. 
\end{defi}

In the above definition we may replace w.e.c. by closed cross-sections. For a w.e.c.  with small flow  boundary the associated cylinders  have a small boundary, so that a topological flow with the small flow boundary property has in particular the small boundary property. Following  the construction of R.Bowen and P.Walters any topological flow with the small flow boundary property admits a complete family of  closed   cross-sections with small flow boundary and with arbitrarily small diameter. Moreover we can assume that  each cross-section in the family is contained in the flow interior of another closed cross-section (see also Lemma 2.4 in \cite{key} for a similar construction).

\subsubsection{Essential and small boundary  partitions }
For a  discrete topological system $(X,T)$ (resp. topological flow $(X,\Phi)$)   a partition $P$ of $X$  is  said to have a small boundary when any atom in $P$ has a small boundary.  Such a partition of $X$ is also called an \textit{essential} partition.

\begin{lemma}\label{fdfee}For a topological system $(X,T)$ (resp. topological  flow $(X,\Phi)$) a Borel  partition $P$ of $X$ has a small boundary if and only if $\overline{A}\setminus A$ is a null set for every $A\in P$. 
\end{lemma}

\begin{proof}
The necessary condition is clear because the sets $\overline{A}\setminus A$ is contained in the boundary of $A\in P$. To prove the equivalence it is enough to see that $A\cap \partial A \subset \bigcup_{B\in P\setminus\{A\}}\overline{B}\setminus B$. Take  $x\in A\cap \partial A$, in particular $x\notin \Int (A)$. Therefore there is a sequence $(x_n)_n$ in the complement set of $A$ going to $x$. By extracting a subsequence we may assume all $x_n$  are in $B$ for some $P\ni B\neq A$, so that $x$ belongs to  $\overline{B}\setminus B$.\end{proof}

 The partition,  generated by a finite cover of  sets with small boundary, has  itself a small boundary. Consequently a  system  with the small boundary property  admits  partitions with small boundary and  arbitrarily small diameter.  Similar properties also hold true for the small flow boundary property. A partition $P$ of a w.e.c. $S$ is said to have a \textit{small flow boundary} when every atom  in $ P$  defines a w.e.c. with small flow boundary (in this case $S$ has itself a small flow boundary).

\begin{lemma}\label{mini}
Let $(X, \Phi)$ be a topological flow with the small flow boundary property. Let $S\subset S'$ be  w.e.c.'s with $\overline{S}\subset \Int^{\Phi}(S')$ such that $S$ has a  small flow boundary.  Then there are  partitions  of $S$ into w.e.c.'s with small flow boundary and  arbitrarily small diameter. 
\end{lemma}

\begin{proof}
For all $x\in \overline{S}$ there is  a  w.e.c. $S_x\subset S'$ with small flow boundary and arbitrarily small diameter satisfying   $x\in \Int^{\Phi} (S_x) \subset \Int^{\Phi}(S')$. The sets $(\overline{S}\cap \Int^\Phi S_x)_{x\in \overline{S}}$ define an open cover of $\overline{S}$. Let  $E$ be a finite subset of $\overline{S}$ such that  $(\overline{S}\cap \Int^\Phi( S_x))_{x\in E}$ is a finite open subcover.  The partition of $S$ generated by the finite cover $(S\cap S_x)_{x\in E}$ of $S$ has a small flow boundary according to  Lemma \ref{ter} (2).
\end{proof}

We define now the corresponding notion for  global Borel cross-sections of a topological flow.  Let $S$ be  a  global Borel cross-section   of a topological flow $(X,\Phi)$. For $A\subset S$ we let $\mathsf T_A$ be the tower above $A$ defined as \[\mathsf T_A:=\{\phi_t(x), \ x\in A \text{ and }0\leq t< t_S(x) \}.\] For a Borel partition $P$ of $S$, the towers $T_A$ for $A\in P$  define a Borel partition $\mathsf T_P$ of $X$. 

\begin{defi}
With the above notations, 
a Borel partition $P$ of $S$ is said essential when the associated partition $\mathsf T_P$ of $X$ in towers is essential. 
\end{defi}

When $Q$ is a partition of $S$ and $P$ is an  essential partition of $S$ finer than $Q$, then $Q$ is also essential.

\begin{lemma}Let $\mathcal{S}$ be a complete family of  cross-sections  with small flow boundary. Then the partition $\mathcal{S}$ of the global closed cross-section  $\mathfrak{S}=\bigcup_{S\in \mathcal{S}}S$ is essential. 
\end{lemma}

\begin{proof} It is enough to show  the joined partition $P=\mathcal{S}\vee T_{\mathfrak S}^{-1}\mathcal{S}$ of $\mathfrak{S}$  is essential. For any $S\in \mathcal{S}$ the w.e.c. $T_{\mathfrak S}^{-1}(\partial^{\Phi} S )$ has a small flow boundary  according to Lemma \ref{sss} :
\begin{eqnarray*}
\frac{1}{\mathsf T}\sharp \{0<t<\mathsf T, \ \phi_t(x)\in T_{\mathfrak S}^{-1}(\partial^{\Phi} S) \}& =& \frac{1}{\mathsf T}\sharp \{0<t<\mathsf T, \ \phi_{t+t_\mathfrak S(\phi_t(x))}(x)\in \partial^{\Phi} S \},\\
&\leq & \frac{1}{\mathsf T}\sharp \{0<t<\mathsf T+ \eta_\mathcal{S}, \ \phi_t(x)\in \partial^{\Phi} S \}\xrightarrow{\mathsf T\rightarrow +\infty}0 \text{ uniformly in }x\in X.
\end{eqnarray*}

Each $E\in P$ has  therefore a small flow boundary by Lemma \ref{tour} and Lemma \ref{ter} (2). Moreover  the restriction of $t_\mathfrak{S}$ to  $E$ extends  on $\overline{E}$ to a continuous function $t_\mathfrak{S}^E$ by Lemma \ref{osc}.
 Therefore  $\overline{T_{E}}\setminus T_{E}$ is contained in the union of $\overline{E}$, $\{ \phi_{t^E_\mathfrak{S}(x)}(x), \ x\in \overline{E}\}$ and $\{\phi_{t}(x), \ 0\leq t\leq \sup t^E_\mathfrak{S} \text{ and }x\in \overline{E}\setminus E\}$. The first two sets are subsets of the  global closed cross-section $\mathfrak{S}$ so that their measure is zero for any probability measure invariant by the flow. The last set is contained in $\partial^{\Phi}E_\zeta$ with $\zeta=\sup t_\mathfrak{S}$ and therefore it  is also a null set because $E$ has a small flow boundary.
\end{proof}

\begin{rem}\begin{itemize}
\item There is no statement similar to Lemma \ref{fdfee} for  a complete family $\mathcal{S}$ of closed cross-sections with small flow boundary. Indeed,   for every $S\in \mathcal{S}$ the set $\overline{S}\setminus S$ is empty  but the cross-section $S$ has  not necessarily a small flow boundary.  
\item  For a Borel global cross-section, essential partitions are more general than partitions with small flow boundary, whose atoms are necessarily w.e.c.'s  (the flow boundary makes only sense for a w.e.c.).
\end{itemize}
\end{rem}



\subsubsection{Small boundary for $\Phi$ and $\phi_t$}

The small flow boundary property
for a closed cross-section may be related with the small boundary property
of the associated cylinders for the discrete time-$t$ maps as follows. 

\begin{lemma}Let $(X,\Phi)$ be a topological flow with a  closed cross-section $S$  and let $t\neq 0$. The following properties are equivalent.
\begin{enumerate}[i)]
\item  $S$ has a small flow boundary for the flow $\Phi$;  
\item for some $\eta>0$ the flow boundary  $\partial^{\Phi} S_\eta$   is a null set  for the topological discrete system $(X,\phi_t)$.
\end{enumerate}
\end{lemma}
However  there may exist $\phi_t$-invariant probability measures with $t\neq 0$ supported on $\phi_\eta S$ and thus on $\partial S_\eta$.
\begin{proof}
As any $\Phi$-invariant measure is $\phi_t$-invariant we have trivially $ii) \Rightarrow i)$. Assume $i)$ and let us prove $ii)$.  Let $\mu$ be a $\phi_t$-invariant measure.  
The  $\Phi$-invariant measure $\theta_t(\mu)$ satisfies
 $0=\theta_t(\mu)(\partial S_\eta)=\frac{1}{t}\int_0^t\phi_s\mu(\partial S_\eta)\, ds$. Therefore we have  $\phi_s\mu(\partial S_\eta)=\mu\left(\partial\phi_{[-\eta-s,\eta-s]}S\right)=0$ for Lebesgue almost $s\in [0,t]$. This concludes the proof as we have $\partial^{\Phi} S_{\eta-s}\subset  \partial\phi_{[-\eta-s,\eta-s]}S$ for $0\leq s\leq \eta$.
\end{proof}

\subsubsection{Suspension flows}
Let $(X,T)$ be a topological discrete system. We consider the suspension flow  $(X_r,\Phi_r)$ over the base $(X,T)$ under a positive continuous roof function $r:X\rightarrow\mathbb{R}^{+}$.
\begin{lemma}\label{suspe}
With the above notations, the flow  $(X_r,\Phi_r)$ satisfies the small flow boundary property if and only  if  $(X,T)$ satisfies the small boundary property. 
\end{lemma}
\begin{proof}
Assume firstly $(X_r,\Phi_r)$ has the small flow boundary property. Let $(x,t)\in X_r$ with $0\leq t<r(x)$. For any $R>0$ the set $S'=\{y\in X,\ d(y,x)\leq R\}\times \{t\}$ is a closed cross-section containing $(x,t)$ in its flow interior (with $d$ being the distance on $X$). Therefore there exists another cross-section $S=U\times \{t\}\subset S'$ with  a small flow boundary for the flow $\Phi_r$ and $(x,t)\in \Int^{\Phi_r}(S)= \Int(U)\times \{t\}$, i.e. $x\in \Int(U)$. For small enough $\zeta>0$ we have $\Theta(\mu)(\partial^{\Phi_r}S_\zeta)=\frac{2\zeta\times \mu(\partial U)}{\int r \, d\mu}=0$ for all $\mu\in \mathcal{M}(X,T)$. Thus $U\subset X$ is a neighborhood of $x$ with  small boundary for $(X,T)$ and with diameter less than $R$. Therefore $(X,T)$ has the small boundary property. 

 Conversely we consider a topological system $(X,T)$ on the base 
 with the small boundary property. Let $S'$ be a closed cross-section containing $(x,t)$ in its flow interior. Let $U$ be a closed neighborhood of $x$ with small boundary for $(X,T)$ and with diameter less than $R$. We let $S$ be the intersection of $S'$ with the $\xi$-cylinder  of the closed cross-section $U\times \{t\}$. For $\xi>0$ fixed  and $\zeta$, $R$ small enough we have $\partial^{\Phi_r}S_\zeta \subset \partial U\times [-\xi+t, \xi+t]$. Therefore  $(x,t)$ belongs to $\Int^{\Phi_r}(S)$ and  $\Theta(\mu)(\partial^{\Phi_r}S_\zeta)\leq \frac{2\xi \times \mu(\partial U)}{\int r\,d\mu}=0$ for all $\mu\in \mathcal{M}(X,T)$, thus $S$ has a small flow boundary.
\end{proof}



\subsubsection{Orbit equivalence}

\begin{lemma}\label{floeq}The small flow boundary property is preserved by orbit equivalence for regular flows.
\end{lemma}
\begin{proof}
Let $(X,\Phi)$ and $(Y,\Psi)$ be two topological orbit equivalent flows, via a homeomorphism  $\Lambda$ from $X$ onto $Y$. Clearly it is enough to show that the image by $\Lambda$ of a closed cross-section  with small flow  boundary is a closed cross-section with small flow boundary.  By continuity of $\Lambda^{-1}$, we have 
\begin{align*}\forall \epsilon> 0 \ \exists \delta>0 \ & \forall y\in Y, \  \psi_{[0,\delta]}(y)\subset \Lambda\left(\phi_{[0,\epsilon]}\left(\Lambda^{-1}y\right)\right).
\end{align*}
If $S$ is a closed cross-section of $(X,\Phi)$ with time $\epsilon$, then $\Lambda(S)$ defines a closed cross-section  of $(Y,\Psi)$ with time $\delta$. The homeomorphism $\Lambda^{-1}$ maps any $\Psi$-orbit of length $\mathsf T$ on a $\Phi$-orbit of length at most $\frac{\epsilon \mathsf T}{\delta}$ so that for any $y\in Y$ 
\begin{align*}
\frac{1}{\mathsf T}\sharp \{0<t<\mathsf T, \ \psi_t(y)\in \partial^{\Psi} \Lambda(S) \}& \leq \frac{1}{\mathsf T}\sharp \{0<t<\mathsf T, \ \psi_t(y)\in\Lambda( \partial^{\Phi} S) \},\\
&\leq  \frac{1}{\mathsf T}\sharp \{0<t<\mathsf T, \ \Lambda^{-1}(\psi_t(y))\in \partial^{\Phi} S \},\\
&\leq \frac{1}{\mathsf T}\sharp \{0<t<\frac{\epsilon \mathsf T}{\delta}, \ \phi_t(\Lambda^{-1}y)\in \partial^{\Phi} S \}. 
\end{align*}
It follows then from  Lemma \ref{sss} iii), that  if $S$ satisfies the small flow boundary property for $(X,\Phi)$ then so does  $\Lambda(S)$ for $(Y,\Psi)$.
\end{proof}
\subsubsection{The case of $C^2$ smooth regular flows.}

Building on works of E.Lindenstrauss and J.Kulesza we prove the small flow boundary property for $C^2$ smooth  (regular) flows on compact manifolds (in fact our proof also applies to $C^1$ discrete  systems).  The closed cross-sections with small flow boundary obtained in our proof are given by smooth discs (in the previous works of  E.Lindenstrauss and J.Kulesza we do not know if one can choose the neighborhoods with small boundary as topological balls).  However we only deal with $C^2$ smooth flows on a compact manifold as we use differential transversality tools, whereas  E.Lindenstrauss and J.Kulesza proved the small boundary property for  homeomorphisms of a finite dimensional compact set. 

A $C^2$ smooth flow is said to have the \textit{smooth small flow boundary property} when  Definition \ref{dedef} holds true with   closed cross-sections $S$ and $S'$  given by $C^1$ smooth discs  transverse to the $C^1$ vector field generating the flow.
\begin{prop}\label{SMB}
Let $(X,\Phi)$ be a $C^2$ smooth  flow on a  compact manifold $X$. We assume that for any $t>0$ the number of periodic orbits  with period less than $t$ is finite. Then $(X,\Phi)$ has the smooth small flow boundary property.
\end{prop}
\begin{proof}Let $d+1$ be the dimension of $X$.   Fix $x\in  X$ and let $S$ be a $C^1$ smooth embedded disc with $x\in \Int^{\Phi}(S)$. 
One easily builds a finite family of $C^1$ smooth embeddings  $(h_i:\mathsf B_d\rightarrow X)_{i=0,...,N}$ with $h_0(0)=x$ and $h_0(B_d)\subset S$, such that the family $\mathcal{S}=(S_i)_i$ and $\mathcal{S}'=(S'_i)_i$ with $S_i= h_i(\mathsf B_d/2)$ and $S'_i=h_i(\mathsf B_d)$   both define  complete families of closed cross-sections with $\eta_\mathcal{S}=\eta_{\mathcal{S}'}$. For $0\leq i,j\leq N$ we let $t_{i,j}$ be the first  hitting time from $S_i$ to $S_j$ : 
$$\forall x\in S_i, \ t_{i,j}(x)=\min\{t>0, \ \phi_t(x)\in S_j\}.$$
We define similarly the first hitting time $t'_{i,j}$ from $S'_i$ to $S'_j$.  Let $U_{i,j}=\{t_{i,j}\leq \eta_\mathcal{S}\}$ and $U'_{i,j}=\{t'_{i,j}<+\infty\}$. Finally we denote by $T_{i,j}:U_{i,j}\rightarrow S_j$ and $T'_{i,j}:U'_{i,j}\rightarrow S'_j$ the first associated hitting maps (see Figure 1). The following properties hold :
\begin{itemize}
\item the compact subset $h_i^{-1}(U_{i,j})$ of $\mathsf B_d/2$ is contained in    $\Int\left(h_i^{-1}(U'_{i,j})\right)$, 
\item  $h_j^{-1}\circ T'_{i,j}\circ h_i$ is a local $C^1$ diffeomorphism on $\Int\left(h_i^{-1}(U'_{i,j})\right)$, because the vector field $\mathsf X$ generating $\Phi$ is $C^1$ smooth, 
\item $T'_{i,j}=T_{i,j}$ on $U_{i,j}\subset U'_{i,j}$. 
\end{itemize} 

\begin{figure}[!ht]

\hspace{-5cm}
\includegraphics[scale=0.2]{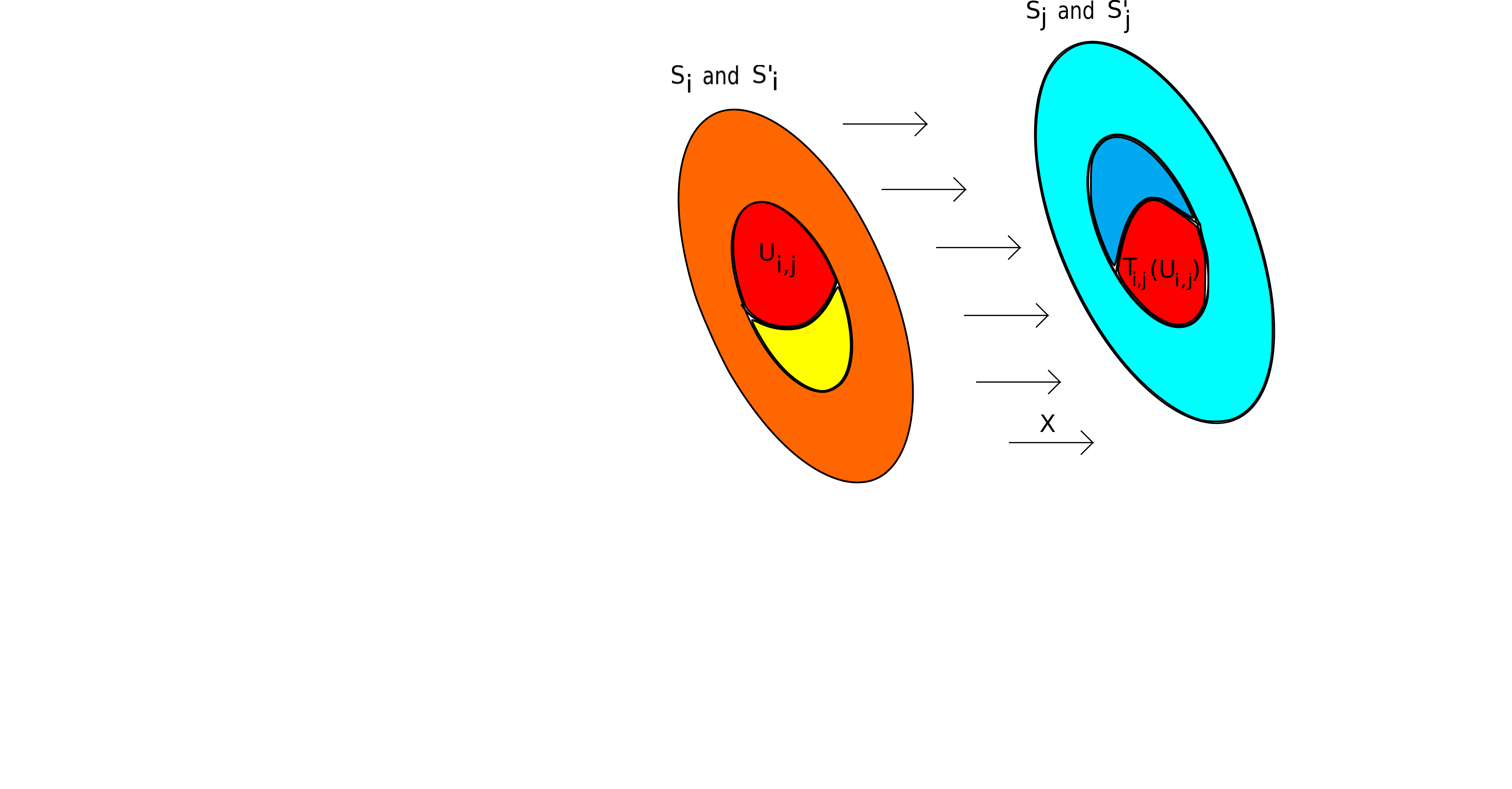}
\centering
\vspace{-2,5cm}
\caption{\label{bl}\textbf{The set $U_{i,j}\subset S_i$ and its image by $T_{i,j}$ in red.} }
\end{figure}

Let $\mathcal{C}_{i,j }$ be a finite collection of closed balls contained  in $\Int\left(h_i^{-1}(U'_{i,j})\right)\subset\mathbb{R}^d$, which covers  $h_i^{-1}U_{i,j}$, such that any $C\in\mathcal{C}_{i,j }$ is contained in an open ball,  where $h_j^{-1}\circ T'_{i,j}\circ h_i$ is a $C^1$ diffeomorphism onto its image. Then for any $C\in \mathcal{C}_{i,j }$ the restriction of $h_j^{-1}\circ T'_{i,j}\circ h_i$ to $C$ extends $C^1$ smoothly to a diffeomorphism $T_C$ of $\mathbb{R}^d$. 
For all $(i_2,...,i_n)\in \{0,...,N\}^{n-1}$  and for all $C^n=(C_1,...,C_{n-1})\in  \prod_{j=1}^{n-1}\mathcal{C}_{i_j,i_{j+1}}$ with $i_1=0$, we let $T_{C^n}^k:=T_{C_k}\circ ...\circ T_{C_1}$    for $1\leq k\leq n-1$ (let also $T_{C^n}^{-k}=(T_{C^n}^k)^{-1}$ for such $k$ and $T^0=\Id_{C_1}$). We denote by $I_{C^n}$ the subset consisting of $0$ and the integers $k$ in $\{1,...,n-1\}$ with $i_{k+1}=0$. Finally we let $F_{C^n}$ be the closed subset of $C_1$ given by $F_{C^n}=\bigcap_{k=0}^{n-1}T_{C^n}^{-k}\left(C_{k+1}\right)$.\\

 
We  build a closed cross-section $\tilde{S}$ (given by a subdisc of $S_0$)  with small flow boundary, arbitrarily small diameter and $x\in \Int^{\Phi}(\tilde S)$.   We let $E:=\{f\in C^1(\mathsf S_d,\mathbb{R}^+), \ \|f\|_{\infty}\leq 1 \}$ endowed with the usual $C^1$ topology (it is a Baire space).   The cross-section $\tilde{S}$ will be of the form $\tilde{S}=h_0(S_f)$ with $S_f:=\{rx, \ x\in \mathsf S_d, \ r\leq f(x)\}\subset \mathbb{R}^d$ for some positive $f\in E$.   Observe that for such an $f$, the boundary $\partial S_f$ is a submanifold of dimension $d-1$ of $\mathbb{R}^d$. Moreover any $C^1$ small enough perturbation of this submanifold, i.e. any submanifold $H(\partial S_f)$ for a $C^1$ diffeomorphism $H$ of $\mathbb{R}^d$ close to the identity,  is of the form $S_g$ for $g \in E$  close to $f$.

   We will say that a map $f\in E$ is \textit{$n$-transverse},  when for all $C^n$ as above,
  the  submanifolds $\left(T_{C^n}^{-i}(\partial  S_f)\right)_{i\in I_{C^n}}$ are mutually transverse  on an open neighborhood of $F_{C^n}$. 
   We   recall that two submanifolds $M$ and $N$ of $\mathbb{R}^d$ are \textit{mutually transverse}    when either $M
  \cap N=\emptyset$ or $T_xM+T_xN=\mathbb{R}^d$ for all $x\in M\cap N$. In this last case the intersection $M
  \cap N$ is itself a submanifold with  $\codim(M\cap N)=\codim M+\codim N
  \geq 0$.  For $n>2$ a family $(M_1,\cdots, M_n)$ of $n$  submanifolds  of $\mathbb{R}^d$ is said 
  \textit{mutually  transverse} when all proper subfamilies are  mutually  transverse  and  the submanifolds    $M_i$ and $
 \bigcap_{j\neq i}M_j$ are transverse  for any   (some\footnote{Indeed this condition does not depend on $i$ because $\dim (T_xM_i+
  T_x(\bigcap_{j\neq i}M_j))=\dim(T_xM_i)+d-\sum_{j\neq i}\codim(T_xM_j)- \dim  (\bigcap_{j}T_xM_j)$ is equal to $d$ if and only if $
 \codim(\bigcap_{j}T_xM_j)
  =\sum_{j} \codim(T_xM_j)$}) $i\in \{1,\cdots,n\}$. Then $ \bigcap_{1=1,\cdots, n}M_i$ is either 
  empty or a submanifold of codimension $\sum_{j=1,\cdots,n}
  \codim(M_j)$. Let $M_1,\cdots,M_{n-1}$ be mutually  transverse submanifolds and let $M_n$ be an other 
  submanifold. Then  $M_1,\cdots,M_n$ are mutually  transverse submanifolds if and only if 
  $M_n$ and $\bigcap_{j\in J}M_j$ are mutually  transverse for any $J\subset \{1,\cdots, n-1\}$. The submanifolds $(M_1,\cdots, M_n)$ are said mutually transverse  on an open set $U$ when $(M_1\cap U,\cdots ,M_n\cap U)$ are mutually transverse.  

\begin{CClaim}\label{dded}
For any $n$ the subset $E_n$ of $E$ consisting of the $n$-transverse maps is open and dense.  
\end{CClaim}
We postpone the proof of  Claim \ref{dded}. 
Let $f\in \bigcap_n E_n$. Any orbit of the flow hits  at most $d-1$ times the set $\partial S_f$. Indeed for any $x\in S_f$ and any 
positive integer $n$ there is  a $n$-uple $C^n$ 
such that (recall $T_{\mathfrak S}$  denotes the return map in the global cross-section $\mathfrak S=\bigcup_{S\in \mathcal{S}}
S$) : \begin{itemize}\item  $x\in F_{C^n}$,
\item    $\forall 0\leq k\leq n, \ T_{\mathfrak S}^k \left(h_0(x)\right)=h_{i_{k+1}}\circ T_{C^n}^k 
(x).$
\end{itemize} The manifolds $\left(T_{C^n}^{-i}(\partial  S_f)\right)_{i\in I_{C^n}}$ being mutually transverse 
submanifolds of dimension $d-1$ on an open neighborhood of $F_{C^n}$, any intersection of $d$'s of them with $F_{C^n}$ is empty. In particular $T_{C^n}^k(x)\in \partial S_f$ and  therefore $T_{\mathfrak S}^k \left(h_0(x)\right)\in h_0(\partial S_f)$ for at most $d$ integers $k$ with $0\leq k\leq n$. As it holds for any $n$, there at most $d$-many  positive times $t
$ with $\phi_t\left(h_0(x)\right)\in \partial S_f$.    By Lemma \ref{sss} $iii)$   the closed cross-section $h_0(S_f)$ has a small flow 
boundary. 
\end{proof}


\begin{proof}[Proof of the Claim \ref{dded}] 
By taking a finite intersection it is enough to consider a single $n$-uple $C^n=(C_1,...,C_n)$ in the definition of $n$-transversality. To simplify the notations we then let $T^k=T_{C^n}^k$ for $0\leq k\leq n-1$, $F_n=F_{C^n}$ and $I_n=I_{C^n}$. 
The tranversality being a $C^1$-stable property (see e.g. Proposition A.3.15 in \cite{KH}), the set  $E_n$ is an open subset of $E$. We prove now by induction on $n$ the density of  $E_n$ in $E$. Let  $f\in E$. By induction hypothesis we may find a positive function $g\in E_{n-1}$ arbitrarily close to $f$. As the set of periodic orbits of $\Phi$ with period less than $t$ is finite for all $t>0$ we may  assume that  the spheres $(T^{-k}\partial S_g)_{k\in I_n}$ avoids the fixed points of $T^k$, $k\in I_n$  on $F_n$.  Therefore there exist two finite families $(B_i)_{1\leq i\leq K}$ and $(B'_i)_{1\leq i\leq K}$ of open balls in $\mathbb{R}^d$ such that 
\begin{itemize}
\item $\overline{B_i}\subset  B'_i$ for all $i$,
\item  $B'_i$, $T^1B'_i,...T^{n-1}B'_i$ are pairwise disjoint for all $i$, 
\item $\bigcup_i B_i\supset F_n\cap \left(\bigcup_{k\in I_n}T^{-k}\partial S_g \right)$. 
\end{itemize}
For $h\in E$ close enough to $g$ the  property of the last item also holds true  for $h$. 

  By induction on  $j=1,\cdots,  K+1$ we produce  a map $g^j$  arbitrarily close to $g$ in $E_{n-1}$ such that $T^{-k}\partial S_{g^j}$, $k\in I_n$, are mutually transverse on an open neighborhood $V_j$ of $F_n \cap \left(\bigcup_{i<j}\overline{B_i}\right)$. 
 Finally we will get  $g^{K+1}\in E_n$ for $g^{K+1}$ close enough to $g$ : if $O$ denotes an open neighborhood of 
 $F_n\cap \left(\mathbb{R}^d\setminus \left(\bigcup_iB_i\right)\right)$ with $T^{-k}\partial S_{g^{K+1}} \cap O=\emptyset$ for all $k\in I_n$, then $(T^{-k}\partial S_{g^{K+1}})_{k\in I_n}$ are mutually transverse on the open neighborhood $O\cup V_{K+1}$ of $F_n$. 

 Take $g^1=g$ and proceed to the inductive step by assuming $g^j\in E_{n-1}$ already built. 
 In particular the submanifolds $(T^{-k}\partial S_{g^{j+1}})_{ k\in I_n\setminus \{0\}}$ (resp. $(T^{-k}\partial S_{g^{j+1}})_{ k\in I_n}$) are mutually transverse on an open  neighborhood  of $F_n$ (resp. of $F_n \cap \left(\bigcup_{i<j}\overline{B_i}\right)$) for $g^{j+1}$ $C^1$-close enough to $g^j$. It is therefore enough to find $g^{j+1}$ arbitrarily close to $g^j$ such that $T^{-k}\partial S_{g^{j+1}}$, $k\in I_n$, are mutually transverse on an open neighborhood of $F_n \cap \overline{B_j}$. 
When  one only  perturbs $\partial S_{g^j}$ on 
$B'_j$, it does not change the submanifolds $T^{-k}\partial 
S_{g^j}$, for $0\neq k\in I_n$, on $B'_j$. By Theorem A.3.19  \cite{KH} there is a $C^1$ small perturbation $\partial S_{g^{j+1}}$ of $\partial S_{g^j}$, supported on $U_n^j\cap B'_j$ (i.e. $\partial S_{g^{j+1}}=H(\partial S_{g^j})$ for a diffeomorphism $H$ of $\mathbb{R}^d$ $C^1$-close to $\Id$ with $H=\Id$ apart from $U_n^j\cap B'_j$),  such that $\partial S_{g^{j+1}}$ and 
$ \bigcap_{k\in I'_n}T^{-k}\partial S_{g^{j}}$ are mutually transverse for any $I'_n\subset I_n\setminus \{0\}$ on a neighborhood of $F_n\cap \overline{B_j}$. But $ \bigcap_{k\in I'_n}T^{-k}\partial S_{g^{j}}$ coincides with $ \bigcap_{k\in I'_n}T^{-k}\partial S_{g^{j+1}}$ on $B'_j$, so that the submanifolds $\partial S_{g^{j+1}}$ and 
$ \bigcap_{k\in I'_n}T^{-k}\partial S_{g^{j+1}}$ satisfy the same transversality property. As it holds for any $I'_n\subset I_n\setminus \{0\}$, the submanifolds $T^{-k}\partial S_{g^{j+1}}$, $k\in I_n$, are mutually transverse on an open neighborhood of $F_n \cap \overline{B_j}$. 
 \end{proof}

\begin{rem}
 We strongly believe the nonautonomous approach through the return maps developed above could be used to prove the small flow boundary property for general topological flows on finite-dimensional compact spaces by adapting the more sophisticated methods of E.Lindenstauss and J.Kulesza.
\end{rem}

\begin{rem}In \cite{DN} Proposition 4.1 it is was proved  that a given finite partition of a smooth compact manifold, whose atoms have piecewise smooth boundaries, has a small boundary with respect to $C^r$ generic diffeomorphisms ($r\geq 1$). Here we use a``dual" approach by fixing the diffeomorphism (in fact the flow in our statement) and  perturbing the boundaries of the partitions.    
\end{rem} 

\subsection{Representation by suspension flow}\label{subs}

Let $(X,\Phi=(\phi_t)_t)$ and $(Y,\Psi=(\psi_t)_t)$ be two continuous flows. We say $(Y,\Psi)$ is a topological extension of $(X,\Phi)$ when there is a continuous surjective map $\pi:Y\rightarrow X$ with $\pi\circ\psi_t=\phi_t\circ \pi$ for all $t$. The topological extension is said \footnote{(Principal, isomorphic, strongly isomorphic) extensions are defined similarly for any group actions.}: \begin{itemize}
 \item \textit{principal} when it preserves the entropy of invariant measures, i.e. 
 $h(\mu)=h(\pi\mu)$ for all $\mu\in \mathcal{M}(Y,\Psi)$,
 \item \textit{isomorphic } when 
the map  induced by $\pi$ on the sets of invariant Borel probability measures 
 is bijective and 
$\pi:(Y,\Psi,\mu)\rightarrow (X,\Phi, \pi\mu)$  is a measure theoretical isomorphism for any $\mu\in \mathcal{M}(Y,\Psi)$,
\item \textit{strongly isomorphic} when there is a full set $E$ of $X$ such that the restriction of $\pi$ to $\pi^{-1}E$ is one-to-one. \end{itemize}
Any strongly isomorphic extension is isomorphic and any isomorphic extension is principal. \\

For a discrete topological system $(X,T)$, to any nonincreasing sequence of  partitions $(P_k)_{k\in \mathbb{N}}$ with $\diam(P_k)\xrightarrow{k}0$  we may associate  a zero dimensional extension $(Y,S)$ given by the closure $Y$ of 
$\{\left(P_k(T^nx)\right)_{k,n}, \ x\in X\}$ in $\prod_k P_k^{\mathbb{Z}}$ (with the shift  acting on each $k$-coordinate) mapping $(A_{k,n})_{k,n}\in Y $ to $\bigcap_{k,n}\overline{T^{-n}A_k^n}\in X$. When the partitions $P_k$ have small boundary then the extension is strongly isomorphic   \cite{dow}.  

We build now in a similar way zero-dimensional extensions for a continuous flow $(X,\Phi)$. Let $\mathcal{S}$ be a complete family of closed cross-sections. 
 As the first return  map $T_\mathfrak{S}$ in $\mathfrak{S}=\bigcup_{S\in \mathcal{S}}S$ is a priori not continuous on $\mathfrak{S}$ we can not directly apply the previous construction to the system $(\mathfrak{S}, T_\mathfrak{S})$. We consider the skew product 
$$\mathcal{S}^\mathbb{Z}\ltimes \mathfrak{S}:=
\left\{ ((A_n)_n,x)\in \mathcal{S}^{\mathbb{Z}}\times \mathfrak{S}, \forall n \ T_{\mathfrak{S}}^nx\in A_n  \right\}.$$
 As a  consequence of Lemma \ref{osc}  the map $\mathcal{S}^\mathbb{Z}\ltimes \mathfrak{S}  \ni \left((A_n)_n,x\right)  \mapsto t_\mathfrak{S}(x)$ extends  continuously on  the closure of  $\mathcal{S}^\mathbb{Z}\ltimes \mathfrak{S}$ in $\mathcal{S}^{\mathbb{Z}}\times \mathfrak{S}$ (this extension will be again denoted by $t_\mathfrak{S}$). Moreover the map 
 $T:\mathcal{S}^\mathbb{Z}\ltimes \mathfrak{S}\circlearrowleft $, $ \left((A_n)_n,x\right)\mapsto \left((A_{n+1})_n,T_\mathfrak{S}x\right)$ extends to a homeomorphism of  $\overline{\mathcal{S}^\mathbb{Z}\ltimes \mathfrak{S}}$. We let $(\overline{\mathcal{S}^\mathbb{Z}\ltimes \mathfrak{S}}, T)$ be  the skew product system given by  this extension.

 \begin{lemma} \label{prec}
 The suspension flow over $(\overline{\mathcal{S}^\mathbb{Z}\ltimes \mathfrak{S}}, T)$ with roof function given by $t_\mathfrak{S}$ is 
 a topological extension of $(X,\Phi)$. Moreover when every  $S\in\mathcal{S}$ has a small flow boundary, this extension is strongly isomorphic.
  \end{lemma}
  \begin{proof}Let $\left(\mathcal{S}^\mathbb{Z}\ltimes \mathfrak{S}\right)_{t_\mathfrak{S}}$ be the invariant dense set above $\mathcal{S}^\mathbb{Z}\ltimes \mathfrak{S}$ in  the suspension flow $\left( \overline{\mathcal{S}^\mathbb{Z}\ltimes \mathfrak{S}}\right)_{t_\mathfrak{S}}$.
  The map $\pi:\left(\mathcal{S}^\mathbb{Z}\ltimes \mathfrak{S}\right)_{t_\mathfrak{S}}\rightarrow (X,
  \Phi)$ with $\pi\left( \left((A_n),x\right), t\right)=\phi_t(x)$ defines an equivariant surjective  function continuously extendable   on  $\left( \overline{\mathcal{S}^\mathbb{Z}\ltimes \mathfrak{S}}\right)_{t_\mathfrak{S}}$.  Therefore its continuous  extension,  again denoted by $\pi$, defines a topological extension of $(X,\Phi)$. 
  
 The $\Phi$-orbit of a point with multiple $\pi$-preimages hits the set $\mathcal{C}_\mathfrak{S}$.  
 Therefore by Lemma \ref{firt} the extension $\pi$ is one-to-one above the (residual) set of points whose $\Phi$-orbits do not visit the flow boundary of the closed cross-sections in $\mathcal{S}$. When these cross-sections have small flow boundaries this set has zero measure for any $\Phi$-invariant measure and thus the extension is strongly isomorphic. Indeed 
  assume by contradiction  there is $S\in \mathcal{S}$ with $\mu(\{x, \exists t\in [0,T] \ \phi_t(x)\in \partial^\Phi S\})>0$  for some $\mu\in \mathcal{M}(X,\Phi)$ and for some $T>0$.  By using the ergodic decomposition we may assume $\mu$ ergodic.  From the ergodic theorem it follows then that the $\Phi$-orbit of $\mu$-almost every $x$ visits $\partial^{\Phi}S$ with a positive frequency contradicting item $iii)$ of Lemma \ref{sss}.
  \end{proof}
  In particular the skew product system $(\overline{\mathcal{S}^\mathbb{Z}\ltimes \mathfrak{S}}, T)$  is a principal extension of $(X,\Phi)$ when the closed cross-sections in $\mathcal{S}$ have a small flow boundary. This answers in this case an  open question of R.Bowen and P.Walters (aforementioned in the introduction).  \\

 A suspension flow over a zero-dimensional topological discrete system will be called a \textit{zero-dimensional suspension flow} and a topological extension by a zero-dimensional suspension flow is said to be a \textit{zero-dimensional extension}.
 
\begin{prop}\label{pro}
A topological flow with the small flow  boundary property admits a zero-dimensional strongly isomorphic extension.
\end{prop}

\begin{proof}   Let $\mathcal{S}_0$ be a complete family of closed cross-sections with small flow boundary, such that  each $S\in \mathcal{S}_0$  is contained  in the flow interior of another closed  cross-section.    By Lemma \ref{mini}  there is a nonincreasing sequence of  partitions with small  flow boundary $(\mathcal{S}_k)_{k\geq 1}$ of $\mathfrak{S}=\bigcup_{S\in \mathcal{S}_0}S$ finer than $\mathcal{S}_0$ satisfying $\diam(\mathcal{S}_k)\xrightarrow{k}0$. Then we can follow the proof 
of Lemma \ref{prec}, by replacing the skew product system $(\overline{\mathcal{S}^\mathbb{Z}\ltimes \mathfrak{S}}, T)$ by the closure $Y^{\mathcal{S}}$ in $\prod_{k\in \mathbb{N}} \mathcal{S}_k^{\mathbb{Z}}$ of $\left\{\left(\mathcal{S}_k(T_{\mathfrak{S}}^nx)\right)_{k,n}  \right\}$ with the shift  acting on each $k$-coordinate, to get the desired zero-dimensional  extension of $(X,\Phi)$ (the roof function is again given by $t_{\mathcal{S}_0}$). The extension is one-to-one above the set of points whose $\Phi$-orbits do not visit the flow boundary of the closed cross-sections in $\mathcal{S}_0$ and the boundary of $A$ in $S$ for every $(A,S)\in (\bigcup_{k\geq 1}\mathcal{S}_k, \mathcal{S}_0)$ with $A\subset S$. But this last boundary is a subset of the flow boundary of $A$ by Lemma \ref{zer}. As the partitions  $\mathcal{S}_k$, $k\geq 0$, have  a small flow boundary,  the extension is strongly isomorphic.
\end{proof}

For two orbit equivalent flows, the base of the suspension flows of the zero-dimensional extensions  built in  Proposition \ref{pro} may be chosen to be topologically conjugate. More precisely consider $(X,\Phi)$ and $(X',\Phi')$ two orbit equivalent topological flows via a homeomorphism $\Lambda:X\rightarrow X'$. Let $\mathcal{S}=(\mathcal{S}_k)_k$ be  a sequence of complete families of closed cross-sections for $(X,\Phi)$ as in the proof of Proposition \ref{pro} and let $\Lambda(\mathcal{S})=(\Lambda(\mathcal{S}_k))_k$ be the associated sequence for $(X',\Phi')$. Then the map  $$\left(\mathcal{S}_k(T_{\mathfrak{S}}^nx)\right)_{k,n}\mapsto\left(\Lambda (\mathcal{S}_k)(T_{\Lambda(\mathfrak{S})}^n(\Lambda x))\right)_{k,n}$$ extends to a topological conjugacy of $Y^{\mathcal{S}}$ and $Y^{\Lambda(\mathcal{S})}$ (with the notations of the  above proof).\\

Any topological system $(X,T)$ has a principal zero-dimensional extension \cite{DH}. By   Proposition \ref{pro} any topological flow with the small flow  boundary property admits a strongly isomorphic, therefore  principal, zero-dimensional extension.

\begin{ques}
Does a topological flow always admits a principal zero-dimensional extension?
\end{ques}

\section{Symbolic extensions and uniform generators for flows}
In this section we develop a theory of symbolic extensions and uniform generators for flows. From now on we only consider topological flows and discrete systems with finite topological entropy. We recall that the topological entropy of a flow is given by the topological entropy of its time $1$-map. 

\subsection{Definitions}
For a topological system $(X,T)$ a \textit{symbolic  extension} is a topological extension $\pi:(Y,S)\rightarrow (X,T)$ where $(Y,S)$ is a subshift over a finite alphabet. \textit{A symbolic extension with an embedding}  is a symbolic extension $\pi:(Y,S)\rightarrow (X,T)$ endowed with a Borel  embedding $\psi:(X,T)\rightarrow (Y,S)$ satisfying $\pi\circ \psi=\Id_{X}$. A \textit{uniform generator } is a Borel partition $P$ of $X$ such that the diameter of $\bigvee_{k=-n}^nT^{-k}P$ goes to zero when $n$ goes to infinity.  
The following statement follows from Theorem 1.2 in \cite{bdo}. The characterization for clopen uniform generators  was first proved in \cite{ky}.
\begin{prop}\label{super}A topological  system $(X,T)$ admits a uniform generator (resp. essential, resp. clopen) if and only if it admits a symbolic extension with an embedding (resp. a strongly isomorphic symbolic extension, resp. is topologically conjugate to a subshift). 
\end{prop}
 For a partition $P$ of $X$ we let $\psi^T_P:(X,T)\rightarrow (P^\mathbb{Z},\sigma)$ be the equivariant map which associates to $x$ its $P$-name, i.e.  $\psi^T_P:x\mapsto (P(T^kx))_{k\in \mathbb{Z}}$. Given a uniform generator $P$  we may in fact build explicitly a symbolic extension with an embedding. Indeed in this case the map $\psi^T_P$ is a Borel embedding and $\pi^T_P:\left(\overline{\psi^T_P(X)},\sigma\right)\rightarrow (X,T)$, $(A_k)_k\mapsto \bigcap_{n\in \mathbb{N}}\overline{\bigcap_{|k|\leq n}T^{-k}A_k}$ is a symbolic extension satisfying $\pi^T_P\circ \psi^T_P=\Id_X$. When moreover the uniform generator $P$ 
   has a small boundary (resp. is clopen), then  the symbolic extension $\pi^T_P$ is a strongly isomorphic extension (resp. a topological conjugacy).

Conversely   when $\pi:(Y,S)\rightarrow (X,T)$ is a symbolic extension with an embedding $\psi$ then the partition $\psi^{-1}Q$ with $Q$ being the zero-coordinate partition of $(Y,S)$ defines a uniform generator. When this symbolic extension is a strongly isomorphic extension (resp. a topological conjugacy) then the corresponding uniform generator has a small boundary (resp. is clopen). Let us check this last point which did not appear in \cite{bdo}.

\begin{lemma}\label{gege}
Let $\pi:(Y,S)\rightarrow (X,T)$ be a strongly isomorphic symbolic extension. Then  the partition $\psi^{-1}Q$ with $Q$ being the zero-coordinate partition of $(Y,S)$ defines an essential uniform generator.
\end{lemma}

\begin{proof}
Let $E$ be a full set of $X$ such that $\pi$ is one-to-one on $\pi^{-1}E$. Let $A\in Q$. A point $x\in \partial \psi^{-1}A$ is a limit of a sequence $(x_n)_n$ in $\psi^{-1}B$ for some $B\neq A\in Q$. Let $y_n=\psi(x_n)$ for all $n$. As $B$ is closed one can assume by extracting a subsequence that $(y_n)_n$ is converging to $y\in B$. Then $\pi(B)\ni \pi(y)=\lim_n\pi(y_n)=\lim_n\pi\circ \psi(x_n)=\lim_nx_n=x$. But we may also write $x$  as the limit of a sequence in $\psi^{-1}A$, therefore $x\in \pi(A)$. Thus $x$ has at least two preimages  under $\pi$. In particular $\partial \psi^{-1}A\subset X\setminus E$ is a null set.  
\end{proof}

We consider now similar notions for a topological regular flow $(X, \Phi)$. A \textit{symbolic extension} $\pi:(Y_r,\Phi_r)\rightarrow (X,\Phi)$ of $(X,\Phi)$ is a topological extension, where $(Y_r, \Phi_r)$ is a suspension flow over a subshift $(Y,S)$ with a positive continuous roof function $r$. \textit{A symbolic extension with an embedding} of $(X,\Phi)$  is a symbolic extension $\pi:(Y_r,\Phi_r)\rightarrow (X,\Phi)$ endowed with a Borel  embedding $\psi:(X,\Phi)\rightarrow (Y_r,\Phi_r)$ satisfying $\pi\circ \psi=\Id_{X}$. 
 \textit{A uniform generator of $(X,\Phi)$} is a Borel global cross-section $S$ together a Borel partition $P$ of $S$ such that $\sup_{y\in P_{T_S}^{[-n,n]}(x)} d(y,x)$ and 
$\sup_{y\in P_{T_S}^{[-n,n]}(x)} |t_S(y)-t_S(x)|$ both go to zero uniformly in $x\in X$  when $n$ goes to infinity.
 We will say that this uniform generator $P$ : \begin{itemize}
\item is \emph{essential}, when the partition $P$  is essential,
\item is \emph{clopen}, when any atom in $P$ is closed with empty flow boundary.\footnote{In this case the cross-section $S$ is  a Poincar\'e  cross-section.}
\end{itemize}


\begin{prop}\label{zut}A topological flow $(X,\Phi)$ admits a uniform generator (resp. essential, resp. clopen) if and only if it admits a symbolic extension with an embedding (resp. a strongly isomorphic symbolic extension, resp. is topologically conjugate to a suspension flow over a subshift). 
\end{prop}

\begin{proof}
Here again we make explicit the symbolic extension with an embedding for a given uniform generator, thus proving   the necessary condition. 
 For a Borel global cross-section $S$ and and a Borel partition $P$ of $S$  we  let $\psi^\Phi_P:X\rightarrow P^{\mathbb{Z}}\times \mathbb{R}$ be the function which maps $x\in X$ to $\left(\psi^{T_S}_{P}(Tx), t(x)\right)$ with $t(x)\in \mathbb{R}^+$ and $T(x)\in S$ being respectively the last hitting time and the last hitting  of $x$ in $S$  (in particular we have $x=\phi_{t(x)}(Tx)$). Assume $P$ defines a uniform generator. For any $(A_k)_k= \psi_P^{T_S}(x)$ with $x\in S$ we let $r\left((A_k)_k\right)=t_S(x)$. As the diameter of $t_S\left(P_{T_S}^{[-n,n]}(x)\right) $ goes to zero uniformly in $x$, the map 
$r$ extends continuously on $\overline{\psi_P^{T_S}(X)}$. The closure $\overline{\psi^{\Phi}_P(X)}$ is then just the suspension flow over $\left(\overline{\psi_P^{T_S}(X)},\sigma\right)$ with roof function $r$. The map $\psi_P^\Phi$ is then a Borel embedding in this flow and $\pi^\Phi_P:\overline{\psi^{\Phi}_P(X)}\rightarrow X$, $((A_k)_k,t)\mapsto \phi_t\left(\bigcap_{n\in \mathbb{N}}\overline{\bigcap_{|k|\leq n}T_S^{-k}A_k}\right)$ is a symbolic extension  satisfying $\pi^\Phi_P\circ\psi^\Phi_P=\Id_X$. When moreover the uniform generator $P$  is clopen then the maps $t$, $T$, $t_S$ and $T_S$ are continuous and therefore the Borel embedding $\psi^\Phi_P$ is a topological embedding. Assume now $P$ is essential, i.e. the partition $\mathsf T_P$ of $X$ in towers is essential. The extension $\pi^{\Phi}_P$ is one-to-one above points, whose orbit by the flow only lies in $S\cup \bigcup_{A\in P}\Int(\mathsf T_A)$. The complement  of these points being a null set, the extension is strongly isomorphic.

Conversely, let $\pi:(Y_r,\Phi_r)$ be a symbolic extension of $(X,\Phi)$ with an embedding $\psi$, given by the suspension flow over the subshift $(Y,\sigma)$ with a positive continuous roof function $r$.  Let $S$ be the global Borel cross-section given by $S=\psi^{-1}(Y\times \{0\})$. We denote by $Q$ the zero-coordinate partition of $Y$. We show now that  the partition $P=\psi^{-1}(Q\times \{0\})$ of $S$ defines a uniform generator of $(X,\Phi)$. Firstly we have $P_{T_S}^{[-n,n]}(x)=\psi^{-1}(Q_{\sigma}^{[-n,n]}(\psi(x)))\subset \pi(Q_\sigma^{[-n,n]}(\psi(x)))$ for all $n\in \mathbb{N}$ and $x\in X$, so that  $\diam(P_{T_S}^{[-n,n]})\xrightarrow{n}0$ by (uniform) continuity of $\pi$. Then for $x\in X$ we have $t_S^\Phi(x)=t^{\Phi_r}_{Y\times \{0\}}(\psi(x))=r(\psi(x))$. Therefore  $\sup_{y\in P_{T_S}^{[-n,n]}(x)} |t^{\Phi}_S(y)-t^{\Phi}_S(x)|\leq \sup_{z\in Q_{\sigma}^{[-n,n]}(\psi(x))} |r(z)-r(\psi(x))|$  goes to zero uniformly in $x\in X$  when $n$ goes to infinity by (uniform) continuity of $r$. When $\psi$ is a topological embedding, the cross-section $S$ is a Poincar\'e cross-section and $(S,T_S)$ is topologically conjugate to $(Y,\sigma)$ through $\psi$. Therefore $(X,\Phi)$ is topologically conjugate to a suspension flow over a subshift in this case. When the symbolic extension $\pi$ is strongly isomorphic,  by imitating the proof of Lemma \ref{gege} any point in the boundary of $\mathsf T_P$ lies in $\pi(\mathsf T_A)\cap \pi(\mathsf T_B)$ for some $A\neq B\in Q$ and thus does not belong to the full set $E$ for which   $\pi:\pi^{-1}E\rightarrow E$ is one-to-one.  The uniform generator $P$ is thus essential.
 \end{proof}
 

\subsection{Entropy structure of topological flows}


We investigate now  the theory of entropy structures for (regular) topological flows. To relate the entropy structure of a discrete system  with the entropy structure of an associated suspension flow, we need to work with an a priori non monotone sequence. In order to deal with this case we generalize below  the abstract theory of convergence developed by T.Downarowicz.

\subsubsection{Abstract theory of convergence}\label{pourr}
Existence of symbolic extensions and uniform generators for discrete systems are related with some subtle properties of convergence of the entropy of measures computed at finer and finer scales. In \cite{dow} T.Downarowicz introduced an abstract framework to study the pointwise convergence of a nondecreasing sequence of functions. We recall now this theory with a slight generalization.

For  a compact metric space $\mathfrak{X}$ we consider    the set  $\mathfrak F_{\mathfrak{X}}$ of all bounded  sequences of real functions on $\mathfrak{X}$, i.e. the set of all sequences $\mathcal{H}=(h_k:\mathfrak{X}\rightarrow \mathbb{R})_{k\in \mathbb{N}}$ with $-\infty<\inf_k\inf_\mu h_k(\mu)\leq \sup_k\sup_\mu h_k(\mu)<+\infty$. 
 Let $\mathcal{H}=(h_k)_{k\in \mathbb{N}}$ and $\mathcal{G}=(g_k)_{k\in \mathbb{N}}$ be two sequences in $\mathfrak F_{\mathfrak{X}}$. Following \cite{dow} we say $\mathcal{G}$ \textit{uniformly dominates} $\mathcal{H}$ and we write $\mathcal{G}\succ\mathcal{H}$ when 
\[ \limsup_k\limsup_l \sup_{\mu\in \mathfrak{X}} (h_k-g_l)(\mu)\leq 0.\]

Similarly, for $\Gamma=(\gamma_k)_k$ and $\Theta=(\theta_k)_k$ in  $\mathfrak F_{\mathfrak{X}}$ we say $\Gamma$ \textit{uniformly yields} $\Theta$ and we write $\Gamma\prec\Theta$ when 
\[ \limsup_k\limsup_l \sup_{\mu\in \mathfrak{X}} (\gamma_l-\theta_k)(\mu)\leq 0.\]
The relation $\succ$ (idem for $\prec$) is preserved by translation : when $f:\mathfrak{X}\rightarrow \mathbb R$ is bounded and 
$(g_k)_k\succ(h_k)_k$  then $(g_k+f)_k\succ(h_k+f)_k$. Moreover $(g_k)_k\succ(h_k)_k$  if and only if   $(-g_k)_k\prec (-h_k)_k$, but in this case  $ (-h_k)_k \succ(-g_k)_k$  does not hold true in general.

The  binary relation $\succ$ is  transitive on $\mathfrak F_{\mathfrak{X}}$. 
In particular when $\mathcal{H}\succ \mathcal{G}$ and $\mathcal{G}\succ \mathcal{H}$ then $\mathcal{H}\succ \mathcal H$. Thus  the transitive  relation $\succ$ induces an equivalence relation $\sim^\succ$ on $\mathfrak{G}^\succ_{\mathfrak{X}}:=\{\mathcal{H}\in \mathfrak F_{\mathfrak{X}}, \ \mathcal{H }\succ \mathcal{H}\}$ by letting \[ [\mathcal{G}\sim^\succ \mathcal{H}]   \Leftrightarrow [\mathcal{G}\succ\mathcal{H}\text{ and }\mathcal{H}\succ\mathcal{G}].\] In an obvious way  we define  similarly $\mathfrak{G}^\prec_{\mathfrak{X}}$ and  $\sim^\prec$.
\begin{lemma}\label{bol}
\begin{itemize}
\item Any nonincreasing (resp. nondecreasing) sequence in $\mathfrak F_{\mathfrak{X}}$ belongs to $\mathfrak{G}^\succ_{\mathfrak{X}}$ (resp. $\mathfrak{G}^\prec_{\mathfrak{X}}$);
\item Any sequence $\mathcal{H}$  in $\mathfrak{G}^\succ_{\mathfrak{X}}$ (resp. $\mathfrak{G}^\prec_{\mathfrak{X}}$) is converging pointwisely to a limit function $\lim \mathcal{H}$.
\item Let $\mathcal{H}=(h_k)_k\in \mathfrak{G}^\succ_{\mathfrak{X}}$  and $\mathcal{G}=(g_k)_k\in  \mathfrak F_{\mathfrak{X}}$ with $\lim_k\sup_{\mu\in \mathfrak X} |h_k-g_k|(\mu)=0$ then $\mathcal{G}$ belongs to $\mathfrak{G}^\succ_{\mathfrak{X}}$ and $\mathcal{H}\sim^\succ\mathcal{G}$. 
\end{itemize}
\end{lemma}
\begin{proof}
The first point follows directly from the definitions. Let us check any sequence $(h_k)_k$  in $\mathfrak{G}^\succ_{\mathfrak{X}}$ is converging pointwisely. It is enough to check $\limsup_k h_k(\mu)=\liminf_lh_l(\mu)$ for all $\mu\in \mathfrak X$ :
\begin{eqnarray*}
 \limsup_k h_k(\mu)-\liminf_lh_l(\mu)&=& \limsup_k\limsup_l  (h_k-h_l)(\mu)\leq 0.
\end{eqnarray*}
Finally we consider  $\mathcal{H}=(h_k)_k\in \mathfrak{G}^\succ_{\mathfrak{X}}$  and $\mathcal{G}=(g_k)_k\in  \mathfrak F_{\mathfrak{X}}$ with $\lim_k\sup_{\mu\in \mathfrak X} |h_k-g_k|(\mu)=0$. Then we have  $\mathcal{G}\succ\mathcal{H}$ :
\begin{eqnarray*}
\limsup_k\limsup_l \sup_{\mu\in \mathfrak{X}} (h_k-g_l)(\mu)&\leq &  \limsup_k\limsup_l \sup_{\mu\in \mathfrak{X}}\left( (h_k-h_l)(\mu)+(h_l-g_l)(\mu)\right), \\
&\leq & \limsup_k\limsup_l \sup_{\mu\in \mathfrak{X}} (h_k-h_l)(\mu)+\limsup_l \sup_{\mu\in \mathfrak{X}}(h_l-g_l)(\mu), \\
&\leq &\limsup_k\limsup_l \sup_{\mu\in \mathfrak{X}} (h_k-h_l)(\mu)\leq 0,
\end{eqnarray*} 
and one proves similarly $\mathcal{H}\succ\mathcal{G}$.

\end{proof}

For a real function $f:\mathfrak X \rightarrow \mathbb{R}$ we let   $f^{\tilde{}}$ denote the upper semicontinuous envelope, i.e.   $f^{\tilde{}}=\inf\{g, \ g\geq f \text{ and }g\text{ is upper semicontinuous}\}$ if $f$ is bounded from above and  $f^{\tilde{}}$ is the constant function equal to $+\infty$ if not. 

\begin{lemma}\label{solo}Let $\Theta=(\theta_k)_k\in \mathfrak{G}^\prec_{\mathfrak{X}}$. Then the following identity holds true : 
$$\limsup_k\sup_\mu \theta_k(\mu)=\sup_\mu\limsup_k\theta_k^{\tilde{}}(\mu).$$
\end{lemma}

\begin{proof}The inequality $\limsup_k\sup_\mu \theta_k(\mu)\geq \sup_\mu\limsup_k\theta_k^{\tilde{}}(\mu)$ is trivial.  Let $M$ be the supremum of $\limsup_k\theta_k^{\tilde{}}$. 
Argue by contradiction by assuming $\limsup_k\sup_\mu \theta_k(\mu)> M$. Therefore we have  for infinitely many $l$ and for some $\mu_l$ 
$$\theta_l(\mu_l)>M.$$
Then for any fixed  $l'\leq l$ we have 
\begin{eqnarray*}
\theta_{l'}(\mu_l) +(\theta_l-\theta_{l'})(\mu_l)&>&M,\\
\theta_{l'}(\mu_l) +\sup_\mu(\theta_l-\theta_{l'})(\mu) &>&M.
\end{eqnarray*}
We get for some weak-$*$ limit $\mu_\infty$ of $(\mu_l)_l$ after taking the limsup in $l$
\begin{eqnarray*}
\theta_{l'}^{\tilde{}}(\mu_\infty) +\limsup_l\sup_\mu(\theta_l-\theta_{l'})(\mu) &>&M.
\end{eqnarray*}
We let now $l'$ go to infinity.  By using  $\Theta \in\mathfrak{G}^\prec_{\mathfrak{X}}$  we obtain the following  contradiction 
\begin{eqnarray*}
\limsup_k\theta_k^{\tilde{}}(\mu_\infty)>M=\sup_\mu \limsup_k\theta_k^{\tilde{}}(\mu).
\end{eqnarray*}
\end{proof}

\begin{lemma}\label{jol}Let $\mathcal{H}=(h_k)_k\in \mathfrak{F}_{\mathfrak{X}}$ and $\mathcal{G}=(g_l)_l\in \mathfrak{G}^{\succ}_{\mathfrak{X}}$. If $\limsup_l  (h_k-g_l)^{\tilde{}}(\mu)\leq 0$ for all $k\in \mathbb{N}$ and for all $\mu\in \mathfrak X$, then  $\mathcal{G}$ uniformly dominates $\mathcal{H}$.
\end{lemma}
\begin{proof} 
For all $k$ the sequence $(h_k-g_l)_l$ belongs to $\mathfrak{G}^\prec_{\mathfrak{X}}$. By Lemma \ref{solo} we get for all $k$   
\begin{eqnarray*}\limsup_l\sup_\mu (h_k-g_l)(\mu)&=& \sup_\mu \limsup_l (h_k-g_l)^{\tilde{}}(\mu),\\
&\leq & 0.
\end{eqnarray*}
By taking the limsup in $k$ we conclude $\mathcal{G}\succ\mathcal{H}$.
\end{proof}

A \textit{superenvelope} of $\mathcal H =(h_k)_k\in \mathfrak{G}^\succ_{\mathfrak{X}}$  is an upper semicontinuous  function 
$E:\mathfrak X \rightarrow \mathbb{R}^+\cup\{+\infty\}$ satisfying 
\begin{eqnarray}\label{enve}
\lim_k(E-h_k)^{\tilde{}}=E-\lim\mathcal{H}.
\end{eqnarray}

Theorem 2.3.2 in \cite{dow} may be stated in our slightly general context as below.  The proof follows the same lines.

\begin{lemma}\label{supere}
Let $\mathcal{H}=(h_k)_k,\mathcal{G}=(g_k)_k\in \mathfrak{G}^\succ_{\mathfrak{X}}$ with $\mathcal{H}\sim^\succ \mathcal{G}$, then 
\begin{itemize}
\item $\lim \mathcal{H}=\lim \mathcal{G}$,
\item $\mathcal{H}$ is uniformly convergent if and only so is $\mathcal{G}$,
\item $\limsup_k (\lim\mathcal{H}-h_k)^{\tilde{}}=\limsup_k (\lim\mathcal{G}-g_k)^{\tilde{}}$,
\item $\mathcal{H}$ and $\mathcal{G}$ have the same superenvelopes. 
\end{itemize}
\end{lemma}

\subsubsection{Entropy structure for flows, definition}\label{brin}

For a discrete topological system $(X,T)$, we let $\mathfrak{G}_T$ be the subset of $\mathfrak{G}^\succ_{\mathcal{M}(X,T)}$ consisting of sequences $\mathcal{H}$ with $\lim\mathcal{H}$ equal to the metric entropy function $h_T=h$  on $\mathcal{M}(X,T)$. The \textit{entropy structure of $(X,T)$} is the equivalence class for the equivalence relation $\sim^{\succ}$ on $\mathfrak{G}_T$  of  the sequence $\mathcal{H}^{Leb}\in \mathfrak{G}_T$ defined below.  By abuse of language any representative of this class is  called an entropy structure of $(X,T)$.  The product of $(X,T)$ with a circle rotation $(\mathsf{R}_\alpha,\mathbb{S}^1)$ by an angle $\alpha\notin \mathbb{Q}$ has the small boundary property (see Theorem 6.2 in \cite{lin}). Fix a nonincreasing  sequence $(R_k)_k$ of partitions with  small boundary satisfying $\diam(R_k)\xrightarrow{k}0$. Let  $\lambda$ be the Lebesgue measure on the circle $\mathbb{S}^1$. Then we define the sequence $\mathcal{H}_{Leb}=(h_k)_k$ by  $h_k:\mu\mapsto h_{T\times \mathsf{R}_\alpha}(\mu\times \lambda,R_k)$ for all $k\in \mathbb{N}$.

For a topological flow $(X,\Phi)$ we  define the entropy structure  similarly. First we recall that the metric entropy $h(\mu)$ of $\mu\in \mathcal{M}(X,\Phi)$ is  the entropy of $i_1(\mu)$ for the time $1$-map $\phi_1$ of $\Phi$. We denote by $\mathfrak{G}_\Phi$ the 
subset of $\mathfrak{G}^\succ_{\mathcal{M}(X,\Phi)}$ consisting of sequences $\mathcal{H}$ with $\lim\mathcal{H}$ 
equal to the metric entropy function $h$ on $\mathcal{M}(X,\Phi)$. Then we define the \textit{entropy structure of $(X,
\Phi)$} as the equivalence class for  $\sim^{\succ}$ (denoted simply by $\sim$ in the following) on $\mathfrak{G}_\Phi$   of  the sequence $\mathcal{H}^\Phi_{BK}
$ defined below.  For $x\in X$, $\epsilon>0$ and $\tau>0$ we let $B_\Phi(x,\epsilon,\tau)$ be the $\Phi$-dynamical  ball:
\[B_\Phi(x,\epsilon,\tau):=\{y\in X, \ d(\phi_s(x),\phi_s(y))<\epsilon \ \forall 0\leq s\leq \tau \}.\]
For a $\Phi$-invariant measure $\mu$ we let for all $x\in M$
\[h^\Phi(\mu,\epsilon,x):=\limsup_{\tau\rightarrow +\infty}-\frac{1}{\tau}\log\mu\left(B_\Phi(x,\epsilon,\tau)\right),\]
and then 
\[h^\Phi(\mu,\epsilon):=\int h^\Phi(\mu,\epsilon,x)\,d\mu(x).\]
M.Brin and A.Katok \cite{bk} have shown this quantity  is converging to $h(\mu)$ when $\epsilon$ goes to zero. 
For a fixed decreasing sequence $(\epsilon_k)_k$ with $\lim_k\epsilon_k=0$ we let $\mathcal{H}^\Phi_{BK}:=\left(h^\Phi(\cdot,\epsilon_k)\right)_k$.



\subsubsection{Relations with the entropy structure of the time $t$-maps.}
For any $t>0$ we recall that the map  $\theta_t:\mathcal{M}(X,\phi_t)\rightarrow \mathcal{M}(X,\Phi)$ defined by  $\mu\mapsto \frac{1}{t}\int_{0}^t\phi_s\mu\,  ds$ is a (affine) retraction, i.e.  we have $\theta_t \circ i_t=\Id_{\mathcal{M}(X,\Phi)}$ with $i_t$ being the inclusion of $\mathcal{M}(X,\Phi)$ in $\mathcal{M}(X,\phi_t)$. 

For a map $\pi:\mathcal{N}\rightarrow \mathcal{M}$ and   a sequence $\mathcal{H}=(h_k)_{k\in\mathbb{N}}$  of real functions on $\mathcal{M}$, we let $\mathcal{H}\circ \pi$ be the sequence $(h_k\circ \pi)_k$ on $\mathcal{N}$.

\begin{lemma}\label{timet}
\begin{enumerate}[i)]
\item If $\mathcal{H}^{\phi_t}$ is an entropy structure of $(X,\phi_t)$ with $t> 0$ then $\frac{1}{t}\mathcal{H}^{\phi_t}\circ i_t$ is an entropy structure of $\mathcal{M}(X,\Phi)$,
\item If $\mathcal{H}^\Phi$ is an entropy structure of $(X,\Phi)$ then  $t \mathcal{H}^\Phi\circ \theta_t$ is an entropy structure of $\mathcal{M}(X,\phi_t)$ for $t> 0$. 
\end{enumerate}
\end{lemma}

\begin{proof}\begin{enumerate}[i)]
\item
For a fixed $t>0$ we may define the Brin-Katok entropy of $\phi_t$ by considering the $\phi_t$-dynamical ball
\[B_{\phi_t}(x,\epsilon_k,n):=\{y\in X, \ d(\phi_{lt}(x),\phi_{lt}(y))<\epsilon_k \ \forall 0\leq l<n\}.\] More precisely 
we let for all $x\in X$, for  all $\mu\in \mathcal{M}(X,\phi_t)$ and for all $k\in \mathbb{N}$
\[h^{\phi_t}(\mu, \epsilon_k,x):=\limsup_{n\rightarrow +\infty}-\frac{1}{n}\log\mu\left(B_{\phi_t}(x,\epsilon_k,n)\right),\]
and then
\[h^{\phi_t}(\mu, \epsilon_k):=\int h^{\phi_t}(\mu, \epsilon_k,x)\, d\mu(x). \]
 This defines an entropy structure $\mathcal{H}_{BK}^{\phi_t}=\left(h^{\phi_t}(\cdot, \epsilon_k)\right)_k$ of the discrete system $(X,\phi_t)$ (see  Appendix \ref{BKK}).  For all $\epsilon>0$ there exists $\tilde\epsilon>0$ such that  for all $\tau>0$ we have 
\begin{eqnarray*}\forall x\in X, \  B_{\phi_{t}} (x,\tilde \epsilon,[\tau/t])&\subset & B_\Phi(x,\epsilon,\tau)\subset B_{\phi_{t}}(x,\epsilon,[\tau/t]).\end{eqnarray*}
Let $\mu$ be a $\Phi$-invariant measure and let $\epsilon>0$. From the above inclusions  we get  :
\begin{eqnarray*}
\forall x\in X, \ \frac{1}{t}h^{\phi_t}(\mu, \epsilon,x)\leq h^{\Phi}(\mu,\epsilon,x) \leq \frac{1}{t}h^{\phi_t}(\mu, \tilde{\epsilon},x), 
\end{eqnarray*} 
and then by integrating with respect to $\mu$ 

\begin{eqnarray*}
\frac{1}{t}h^{\phi_t}(\mu, \epsilon)\leq h^{\Phi}(\mu,\epsilon) \leq \frac{1}{t}h^{\phi_t}(\mu, \tilde{\epsilon}).
\end{eqnarray*}

In particular  $\frac{1}{t}\mathcal{H}_{BK}^{\phi_t}\circ i_t$ and $\mathcal{H}_{BK}^{\Phi}$ are equivalent. \\

\item According to the first item  we can assume the entropy structure $\mathcal{H}^\Phi$ is $\frac{1}{t}\mathcal{H}_{BK}^{\phi_t}\circ i_t$.  As the functions in $\mathcal{H}_{BK}^{\phi_t}$ are harmonic\footnote{A 
real function $f$ defined on the Choquet simplex of probability invariant measures is said \textit{harmonic} when for any invariant  probability measure $\mu$  we have  $f(\mu)=\int f(\nu_x)d\mu(x)$ with $\mu=\int \nu_x\, d\mu(x)$ being  the ergodic decomposition of $\mu$. In particular, harmonic functions are affine.}, the sequence $t\mathcal{H}^{\Phi}\circ \theta_t=\mathcal{H}^{\phi_t}_{BK}\circ \theta_t$ is just the  sequence of  functions  $\mu\mapsto 
\frac{1}{t}\int_0^t h_{BK}^{\phi_t}\left(\phi_{s}\mu,\epsilon_k\right)ds$, $k\in \mathbb{N}$. 

 For all $\epsilon>0$ there exists $g(\epsilon)\leq \epsilon$ such that $d(x,y)<g(\epsilon)\Rightarrow d(\phi_s(x), \phi_s(y))<\epsilon$ for any $s$ with $|s|\leq t$. In particular we have for all $x\in X$, $n\in \mathbb{N}$ and $\epsilon>0$ :
$$\phi_{-s}B_{\phi_t}(x,g(\epsilon),n)\subset B_{\phi_t}(\phi_{-s}(x),\epsilon,n)$$

and for all $\mu\in \mathcal{M}(X,\phi_t)$,
\begin{eqnarray*}
  h^{\phi_t}(\phi_s\mu,g(\epsilon),x)&=&\limsup_{n\rightarrow +\infty}-\frac{1}{n}\log\mu\left(\phi_{-s}B_{\phi_t}(x,g(\epsilon),n)\right),\\
 &\geq & \limsup_{n\rightarrow +\infty}-\frac{1}{n}\log\mu\left(B_{\phi_t}(\phi_{-s}x,\epsilon,n)\right),\\
 &\geq & h^{\phi_t}(\mu,\epsilon,\phi_{-s}x).
 \end{eqnarray*}
 By integrating this last inequality with respect to $\phi_s\mu$ we get for all $\mu\in \mathcal{M}(X,\phi_t)$ and for all $|s|\leq t
$ :
 $$h_{BK}^{\phi_t}\left(\phi_{s}\mu,g(\epsilon)\right) \geq h_{BK}^{\phi_t}\left(\mu,\epsilon\right).$$ 
Consequently  the sequence $t\mathcal{H}^{\Phi}\circ \theta_t=\left(\frac{1}{t}\int_0^t h_{BK}^{\phi_t}\left(\phi_{s}\cdot,\epsilon_k\right)ds\right)_k$ is equivalent to the Brin-Katok entropy structure  $\mathcal{H}_{BK}^{\phi_t}=\left(h^{\phi_t}(\cdot, \epsilon_k)\right)_k$.  \end{enumerate}
\end{proof}




For discrete dynamical systems entropy structures are preserved by principal extensions (see Theorem 5.0.3 (2) in \cite{dow}).
A principal extension  between two topological flows induces  a principal extension between their time $t$-maps. Indeed the entropy function being harmonic we have $h_\Phi(\theta_t(\mu))=h_{\phi_t}(\mu)$ for  any $\mu\in \mathcal{M}(X, \phi_t)$ for a topological flow $(X,\Phi)$. Then for a principal extension $\pi:(Y,\Psi)\rightarrow (X,\Phi)$
 we get for any $\mu \in \mathcal{M}(Y,\psi_t)$ : 
 \begin{align*}
 h_{\phi_t}(\mu)&=h_\Phi(\theta_t(\mu)),\\
 &=h_\Psi(\pi \theta_t(\mu)),\\
 &=h_\Psi(\theta_t(\pi\mu))=h_{\psi_t}(\pi\mu).
 \end{align*}  As a consequence of Lemma \ref{timet} we obtain then the result analogous to Theorem 5.0.3 (2)  \cite{dow} for topological flows : 
\begin{cor}\label{mami} Entropy structures of flows are preserved by principal extensions, i.e. if $\pi:(Y,\Psi)\rightarrow (X,\Phi)$ is a principal extension then $\mathcal{H}\circ \pi$ is an entropy structure of $(Y,\Psi)$ if and only if $\mathcal{H}$ is an entropy structure of $(X,\Phi)$.\end{cor}

For a discrete topological system the entropy with respect to a nonincreasing sequence of partitions with small boundary defines an entropy structure. In our context we have :
\begin{cor}\label{magic}Let $(X,\Phi)$ be a topological flow.
If $\mathcal{P}=(P_k)_k$ is  a sequence of partitions of $X$ with small boundary such that $\mathcal{H}_{\mathcal{P}}:=(h(\cdot, P_k))_k$ belongs to $\mathfrak{G}_\Phi$, then $\mathcal{H}_{\mathcal{P}}$ defines an entropy structure of $(X,\Phi)$. 
\end{cor}
\begin{proof}Let $(R_k)_k$ be a nonincreasing sequence of partitions of $X$ defining the entropy structure $\mathcal{H}_{Leb}$  for  $T=\phi_1$ (see  Subsection \ref{brin}). Then we have  for any $\mu\in \mathcal{M}(X,\phi_1)$   
\begin{align*}
 h(\mu,P_k)-h(\mu\times \lambda,R_l)& \leq h\left(\mu\times \lambda, (P_k\times \mathbb{S}^1)\vee R_l | R_l\right),\\
h(\mu\times \lambda,R_l)- h(\mu,P_k)&\leq  h\left(\mu\times \lambda, (P_k\times \mathbb{S}^1)\vee R_l | P_k\times \mathbb{S}^1)\right).
\end{align*}
For $k$ (resp. $l$) fixed the first (resp. second)  right member defines an  upper semicontinuous function on $\mathcal{M}(X,\Phi)$ going pointwisely to zero when $l$ (resp. $k$)  goes to infinity. The sequence $\mathcal{H}_{\mathcal P}$  and the restriction of  $\mathcal{H}_{Leb}$ to $\mathcal{M}(X,\Phi)$ being both in $\mathfrak{G}_\Phi$ they are equivalent by Lemma \ref{jol}. By Lemma \ref{timet}  this restriction  is an entropy structure of the flow. Therefore $\mathcal{H}_{\mathcal P}$ is also an entropy structure of the flow.  
\end{proof}

\subsubsection{Entropy of suspension flows}
In this paragraph we consider a suspension flow $(X_r, \Phi_r=(\phi_t)_t)$ over a topological system $(X,T)$ with a continuous positive roof function $r$.   To simplify the notations we will denote by $\nu_\mu$ the $\Phi_r$-invariant measure $\Theta(\mu)$ associated to the $T$-invariant measure $\mu$ and $\mu_\nu$ the  $T$-invariant measure $\Theta^{-1}(\nu)$ associated to the $\Phi_r$-invariant measure $\nu$, where $\Theta$ is the homeomorphism defined in Subsection \ref{sssup}.  The  entropy of $\mu$ and $\nu_\mu$ are related  by the following formula due to L.M.Abramov \cite{ab}:
$$h(\nu_\mu)=\frac{h(\mu)}{\int r \ d\mu}.$$
Abramov formula holds for any measurable suspension flow over a measurable system. It follows from the formula  for  the entropy of an induced system. We recall below the corresponding formula for the entropy with respect to a given partition. 

\begin{lemma}\label{ac}
Let $(Y,f,\mathcal{B},\nu)$ be a measure preserving system and let $A\subset Y$ with $\nu(A)>0$. Then for any finite Borel partition $P$ of $A$ we have 
\[\nu(A)h(\nu_A,f_A,P\vee R_A)=h(\nu,f,\overline{P}),\]
where $\overline{P}$ is the partition of $Y$ given by $\overline{P}:=\{ Y\setminus A, \ B \ :  \ B\in P\}$  and $R_A:=\{\tau_A=k,\ k\in \mathbb{N}\setminus \{0\}\}$ is the partition of $A$ with respect to the first return time $\tau_A$ in $A$.
\end{lemma}

\begin{proof} As both terms in the above equality is preserved by the ergodic decomposition\footnote{If the ergodic decomposition of $\nu$ is given by  $\nu=\int_Y\nu^x\, d\nu(x)$, then $\nu_A=\int_A\nu^x(A)\nu^x_A\, d\nu_A(x)$ is the ergodic decomposition of $\nu_A$.}, one can assume without loss of generality the ergodicity of $\nu$. 
The induced measure $\nu_A$ on $A$ is then also ergodic.  From the Birkhof ergodic theorem, we get
\begin{eqnarray}\label{ein}  
\forall \nu-\text{a.e. }x,\  \frac{1}{n}\sum_{k=0}^{n-1}\tau_A(f_A^k(x))& \xrightarrow{n} &\int \tau_A\, d\nu_A=\frac{1}{\nu(A)}.
\end{eqnarray}

 Then by Shanon-McMillan-Breiman  formula we have:
\begin{eqnarray}\label{zwei}  
\forall \nu_A-\text{a.e. }x,\ h(\nu_A,f_A,P\vee R_A)&=&\lim_n-\frac{1}{n}\log \nu_A\left((P\vee R_A)^n_{f_A}(x)\right),
\end{eqnarray}

\begin{eqnarray}\label{drei}  
 \forall \nu-\text{a.e. }x, \ h(\nu,f,\overline P)&=\lim_{n'}-\frac{1}{n'}\log \nu\left(\overline{P}^{n'}_f(x)\right),
\end{eqnarray}

where   $(P\vee R_A)^n_{f_A}(x)$ and  $\overline{P}^{n'}_{f}(x)$ denote respectively the atom of the iterated  partition $\bigvee_{k=0}^{n-1}f_A^{-k}(P\vee R_A)$ and  $\bigvee_{k=0}^{n'-1}f^{-k}\overline{P}$ containing $x$. But we have $(P\vee R_A)_{f_A}^n(x)=\overline{P}_f^{n'_x}(x)$  with $n'_x:=\sum_{k=0}^{n-1}\tau_A(f_A^k(x))$.  By taking a  point $x$ satisfying the three above properties (\ref{ein}), (\ref{zwei}), (\ref{drei}), we get :
\begin{align*}
h(\nu_A,f_A,P\vee R_A)&=\lim_n-\frac{1}{n}\log \nu_A\left((P\vee R_A)^n_{f_A}(x)\right),\nonumber\\
& =\lim_n-\frac{n'_x}{n}\frac{1}{n'_x}\log \nu_A\left(\overline{P}^{n'}_{f}(x)\right), \nonumber\\
h(\nu_A,f_A,P\vee R_A)&=\frac{h(\nu,f,\overline{P})}{\nu(A)}.
\end{align*}
\end{proof}

We return now to our suspension flow $(X_r, \Phi_r)$. Let $\underline{r}:=\inf_{x\in X}r(x)>0$. We may deduce  from  Lemma \ref{ac} the following inequalities for the entropy of suspension flows.

\begin{lemma}\label{ab}
Let $P$ be a Borel partition of $X$, then 
 for all  $\delta\in ]0,\underline{r}[$ and for all $ \mu\in \mathcal{M}(X,T)$  \begin{eqnarray}\label{dcd} \frac{h(\nu_\mu,\phi_{\delta},\overline{P_\delta})}{\delta}\geq \frac{h(\mu,T,P)}{\int r \ d\mu},
 \end{eqnarray}
 where $\overline{P_{\delta}}$ is the partition of $X_r$ given by $\overline{P_{\delta}}:=\{ X_r\setminus (X\times [0,\delta[), \ B\times [0,\delta[ \ :  \ B\in P\}$.
When we moreover assume   $|r(x)-r(y)|<\delta$ for all $x,y$ in the same atom of   $P$, then 
\[ \frac{h(\nu_\mu,\phi_{\delta},\overline{P_\delta})}{\delta}\leq \frac{h(\mu,T,P)+\log 3}{\int r \ d\mu},\]
 
\end{lemma}

\begin{proof} Let  $0<\delta<\underline{r}$. Let $A_\delta$ be the subset of $X_r$ given by $A_\delta= X \times [0,\delta[$. For the partition $P$ of $X$ we first denote by  $P_{\delta}=\{B\times [0,\delta[, \ B\in P\}$ the partition induced on $A_\delta$.  We also let $R_\delta=R_{A_\delta}$ be the partition with respect to the first return time in $A_\delta$.  By applying Lemma \ref{ac}  to $\nu_\mu$, $\phi_{\delta}$, $X_r$  and $P_{\delta}$
we get \begin{eqnarray}\label{t}
h\left((\nu_\mu)_{A_\delta},(\phi_{\delta})_{A_\delta},P_{\delta}\vee R_{\delta}\right)&=&\frac{h\left(\nu_\mu,\phi_{\delta},\overline{P_{\delta}}\right)}{\nu_\mu(A_\delta)},\nonumber\\ 
&=&\frac{\int r\, d\mu}{\delta}h\left(\nu_\mu,\phi_{\delta},\overline{P_{\delta}}\right).
\end{eqnarray}
But  the partition $\bigvee_{k=0}^{n-1}(\phi_\delta)_{A_\delta}^{-k}P_{\delta}$ of $A_\delta$ is just the partition 
$\bigvee_{k=0}^{n-1}T^{-k}P\times [0,\delta[$ and therefore we get, with $H_\mu(Q)=\sum_{C\in Q}-\mu(C)\log \mu(C)$ :
\begin{align*}
h\left((\nu_\mu)_{A_\delta},(\phi_{\delta})_{A_\delta},P_\delta\right)&=\lim_n\frac{1}{n}H_{(\nu_\mu)_{A_\delta}}\left(\bigvee_{k=0}^{n-1}T^{-k}P\times [0,\delta[\right),\\
&= \lim_n\frac{1}{n}\sum_{C\in \bigvee_{k=0}^{n-1}T^{-k}P\times [0,\delta[ }-(\nu_\mu)_{A_\delta}(C)\log (\nu_\mu)_{A_\delta}(C).
\end{align*}
For any $B\in \bigvee_{k=0}^{n-1}T^{-k}P$ and $C=B\times [0,\delta[$ we have
$ (\nu_\mu)_{A_\delta}(C)=\frac{\nu_\mu(C)}{\nu_\mu(A_\delta)}=\mu(B)$.
Therefore we obtain finally :
\begin{align*}
h\left((\nu_\mu)_{A_\delta},(\phi_{\delta})_{A_\delta},P_\delta\right)&=h(\mu,T,P),
\end{align*}
which implies the first inequality. 

Then by using again Equality (\ref{t}) we get 
\begin{align*}
\left|\frac{1}{\delta}h\left(\nu_\mu,\phi_{\delta},\overline{P_\delta}\right)-\frac{h(\mu,T,P)}{\int rd\mu}\right|&=
\frac{1}{\int rd\mu}\left| h\left((\nu_\mu)_{A_\delta},(\phi_{\delta})_{A_\delta},P_\delta\vee R_{\delta}\right)-h\left((\nu_\mu)_{A_\delta},(\phi_{\delta})_{A_\delta},P_\delta\right) \right|,\\
&=\frac{1}{\int rd\mu}h\left((\nu_\mu)_{A_\delta},(\phi_{\delta})_{A_\delta},P_\delta\vee R_{\delta}| P_{\delta}\right).
\end{align*}
Under the additional assumption of  small oscillation of $r$, any element of $P_\delta$ has a non empty intersection with at most $3$ elements of $R_\delta$ so that the conditional entropy 
$h\left((\nu_\mu)_{A_\delta},(\phi_{\delta})_{A_\delta},P_\delta\vee R_{\delta}| P_{\delta}\right)$ is bounded from above by $\log 3$. 
\end{proof}

\subsubsection{Relations with the entropy structure of the base dynamics for a suspension flow.}
 Let $(X_r, \Phi_r)$ be a suspension flow over a zero-dimensional discrete system $(X,T)$. We relate the entropy structure of the flow   $(X_r, \Phi_r)$  with the entropy structure of  $(X,T)$. To any  sequence $\mathcal{H}=(h_k)_k\in \mathfrak{G}_T$ we associate the  sequence $\mathcal{H}_r\in \mathfrak{G}_\Phi$  defined by $\mathcal{H}_r=\left(\nu\mapsto\frac{h_k(\mu_\nu)}{\int r \, d\mu_\nu}\right)_k$. 
 
\begin{lemma} \label{bidule}
 The map $\mathcal{H}\mapsto \mathcal{H}_r$ is well-defined and  compatible with the equivalence relation $\sim$, i.e. $[\mathcal{H}\sim \mathcal{G}]\Leftrightarrow [\mathcal{H}_r\sim \mathcal{G}_r]$.  
\end{lemma}

The proof follows  from the continuity of the  map given by $\mathcal{M}(X,T)\ni \mu\mapsto \frac{1}{\int r \, d\mu}$ and the continuity of $\Theta$ and $\Theta^{-1}$. The details are left to the reader.

\begin{lemma}\label{poly}Assume $(X,T )$ is an aperiodic zero-dimensional system.
There  exist a nonincreasing sequence $(Q_k)_k$ of clopen partitions of  $X$  with $\diam(Q_k)\xrightarrow{k\rightarrow +\infty}0$ and 
and a  sequence of  partitions $(P_k)_k$ of $X_r$ with small boundary (for the flow $\Phi_r$) such that 
\begin{eqnarray*}
\sup_{\mu\in \mathcal{M}(X,T)}\left|h(\nu_\mu,P_k)-\frac{h(\mu,Q_k)}{\int r\,d\mu}\right|&\xrightarrow{k\rightarrow +\infty}&0.
\end{eqnarray*}
\end{lemma}
\begin{proof}

 We consider a  sequence $(U_k)_{k\in \mathbb{N}}$ of nested topological Rohlin towers (see Lemma 8.5.4  \cite{dowb}):
\begin{itemize}
\item $U_k$ is a clopen set for every $k$,
\item $U_{k+1}\subset U_k$ for every $k$,
\item $X=\bigcup_{n\in \mathbb{N}}T^nU_k$,
\item$\underline{\tau_{U_k}}:=\min_{x\in U_k}\tau_{U_k}(x)\xrightarrow{k\rightarrow +\infty}+\infty$. 
\end{itemize}

Let $k\in \mathbb{N}$. By Kac's formula we have $\mu(U_k)\leq \frac{1}{\underline{\tau_{U_k}}}$. The flow $(X_r,\phi_r)$  may be represented as a suspension flow over $U_k$ with roof function $r_k:=\sum_{0\leq l<\tau_{U_k}}r\circ T^l$. Note that $\int r_k\,d\mu_{U_k} =\frac{\int r\,d\mu}{\mu(U_k)}$. A clopen partition $R_k$ of $U_k$ finer than $R_{U_k}:= \left\{\{\tau_{U_k}=l \}\ | \ l\in \mathbb{N}\setminus \{0\}\right\}$  induces a clopen partition $Q_k$ of $X$ by letting $Q_k=\{T^m\left(\{\tau_{U_k}=l\}\cap A\right) \ | \  A\in R_k, \ l\in \mathbb{N}\setminus \{0\}, \ 0\leq m<l\}$. Observe that $h(\mu,T, \overline{R_k})=h(\mu,T,Q_k)$ where $\overline{R_k}$ denotes the partition of $X$ given by $\overline{R_k}:=\{X\setminus U_k, B \ | \  B\in R_k\}$ (indeed for any positive  integer $n$ the partition $Q_k^{n+M}$ is finer than $\overline{R_k}^{n+M}$, which is itself finer than $Q_{k}^n$, with $M$ being a fixed    integer larger than $\max_{x\in U_k}\tau_{U_k}(x)$).

  We may choose such a sequence $(R_k)_k$  that the induced partitions
$(Q_k)_k$ satisfy  $\diam(Q_k)<1/k$ and  $Q_{k+1}$  finer than $Q_k$ for all $k$. Moreover we may assume the diameter of $R_k$ so small that $|r_k(x)-r_k(y)|<\delta$ for any points $x$ and $y$ in the same atom of $R_k$. Since the partition $R_k$ is finer than $R_{U_k}$ we get according to Lemma \ref{ac} :
\[\mu(U_k)h(\mu_{U_k}, T_{U_k},R_k)=h(\mu,T,\overline{R_k})=h(\mu, T,Q_k).\]
Fix  $p\in \mathbb{N}^*$ with   $\delta:=1/p<\underline{r}$. By applying Lemma \ref{ab} for $\delta$ to the  suspension flow over $U_k$ we get :
\begin{eqnarray*}
 \frac{h(\mu_{U_k}, T_{U_k}, R_k)}{\int r_k\,d\mu_{U_k}}&\leq \frac{1}{\delta}h(\nu_{\mu_{U_k}},\phi_\delta, \overline{(R_k)_\delta})&\leq \frac{h(\mu_{U_k}, T_{U_k}, R_k)+\log 3 }{\int r_k\,d\mu_{U_k}},\\
\frac{h(\mu, T, Q_k)}{\int r\,d\mu}&\leq \frac{1}{\delta}h(\nu_\mu, \phi_{\delta},\tilde{\overline{(R_k)_\delta}})& \leq \frac{h(\mu, T,Q_k)+\mu(U_k)\log 3}{\int r\,d\mu},
\end{eqnarray*} 
where $ \tilde{\overline{(R_k)_\delta}}$ is the partition of $X_r$ obtained from the partition $\overline{(R_k)_{\delta}}$ of $X_{r_k}$ through the natural topological conjugacy between the two 
suspension flows. We  let $P_k$ be the partition given by  $P_k:=\bigvee_{l=0}^{p-1}\phi_{l/p}^{-1}
\left(\tilde{\overline{(R_k)_\delta}}\right)$ so that we have (recall $\delta=1/p$) :
\[\frac{h(\mu, T,Q_k)}{\int r\,d\mu}\leq h(\nu_\mu, \phi_{1},P_k) \leq \frac{h(\mu,T, Q_k)+\mu(U_k)\log 3}{\int r\,d\mu}.\] The partitition $R_k$ of $U_k$ being clopen, the sets $B\times [0,\delta[$ for $B\in R_k$ have a small  boundary  for $\Phi_r$. Consequently  $\overline{(R_k)_{\delta}}$, and then $P_k$,  is a partition of $(X_r, \Phi_r)$ with small boundary.

\end{proof}

The sequence $(P_k)_k$ built in the above lemma is a priori not nonincreasing. That is why  we have generalized 
the theory of entropy structures in Subsection \label{pourr}.

\begin{cor}\label{ssssu}
With the above notations the following assertions are equivalent:
\begin{enumerate}
\item  $\mathcal{H}$ is an entropy structure of $(X,T)$, 
\item  $\mathcal{H}_r$ is an entropy structure of  $(X_r,\Phi_r)$.
\end{enumerate}
\end{cor}

\begin{proof}
We first prove $(1)\Rightarrow(2)$ for  an aperiodic zero-dimensional system $(X,T)$.  Let $\mathcal{P}:=(P_k)_k$ and $\mathcal{Q}=(Q_k)_k$ be as in Lemma \ref{poly}. The sequence $\mathcal{H}_\mathcal{Q}$ is  an entropy structure of $(X,T)$ (see \cite{dowb}). Then if $\mathcal{H}$ is an entropy structure of $(X,T)$, we have $\mathcal{H}_\mathcal{Q}\sim \mathcal{H}$ and therefore $(\mathcal{H}_\mathcal{Q})_r\sim \mathcal{H}_r$ by Lemma \ref{bidule}.  By the last item of Lemma  \ref{bol}, the sequence $\mathcal{H}_\mathcal{P}$ belongs to $\mathfrak{G}_{\Phi_r}$ and  $(\mathcal{H}_\mathcal{Q})_r\sim \mathcal{H}_\mathcal{P}$. But  the sequence $\mathcal{H}_\mathcal{P}$ defines also an entropy structure of  $(X,\Phi)$ according to  Corollary \ref{magic}.  Thus  $\mathcal{H}_r$ is  an entropy structure of $(X,\Phi)$.

We deal now with the general case. Consider an aperiodic principal zero-dimensional extension $\pi:(Y,S)\rightarrow(X,T)$. Let $\mathcal{H}=(h_k)_k$ be an entropy structure of $(X,T)$.  As entropy structures are preserved by principal extensions, the sequence $\mathcal{H}\circ \pi$ is an entropy structure of $(Y,S)$.
Let $(Y_{r'}, \Phi_{r'})$ be the suspension flow of $(Y,S)$ under the roof function $r'=r\circ \pi$. The map $\pi':(Y_{r'},\Phi_{r'})\rightarrow (X_r,\Phi_r)$, $(y,t)\mapsto (\pi(y),t)$, defines a principal extension.
From the aperiodic  case  the sequence
$(\mathcal{H}\circ \pi)_{r'}=\mathcal{H}_r\circ \pi'$ defines an entropy structure of $(Y_{r'}, \Phi_{r'})$.
But if $\mathcal{F}=(f_k)_k$ is an entropy structure of $(X_r,\Phi_r)$ then 
$\mathcal{F}\circ \pi'$ is also an entropy structure of $(Y_{r'}, \Phi_{r'})$ by Corollary \ref{mami}. Thus $\mathcal{F}\circ \pi'$ is  equivalent to $\mathcal{H}_r\circ \pi'$.  Therefore  $\mathcal{H}_r\sim \mathcal{F}$ is an entropy structure of  $(X_r,\Phi_r)$.

The other implication    $(2)\Rightarrow(1)$ follows easily from $(1)\Rightarrow(2)$.  Indeed let $\mathcal{G}$ be   an entropy structure  of $(X,T)$. Then  $\mathcal{G}_r$ is an entropy structure of $(X_r, \Phi_r)$. 
Let $\mathcal{H}$ be a sequence  in $\mathfrak{G}_T$  such that  $\mathcal{H}_r$ is an entropy structure of  $(X_r,\Phi_r)$. We have   $\mathcal{H}_r\sim \mathcal{G}_r$, therefore
 $\mathcal{H}\sim \mathcal{G}$ and $\mathcal{H}$ is also an entropy structure of $(X,T)$.
\end{proof}

\subsubsection{Superenvelope of flows}

For a discrete dynamical system (or a topological flow),   a  \textit{superenvelope of the entropy structure} (or simply a \textit{superenvelope}) is a superenvelope of a given entropy structure (seen as a representative sequence in the equivalence class). This definition does not depend on the choice of the representative sequence by Lemma \ref{supere}.


The statement below follows easily  from  Corollary \ref{mami} and the definition of superenvelopes  (see Lemma 8.4.8 in \cite{dowb} for the analogous result for discrete systems) : 
\begin{lemma}\label{mamip} Superenvelopes of flows are preserved by principal extensions, i.e. if $\pi:(Y,\Psi)\rightarrow (X,\Phi)$ is a principal extension then $E\circ \pi:\mathcal{M}(Y,\Psi)\rightarrow \mathbb{R}^+\cup \{+\infty\}$ is a superenvelope of $(Y,\Psi)$ if and only if $E:\mathcal{M}(X,\Phi)\rightarrow \mathbb{R}^+\cup \{+\infty\}$ is a superenvelope  of $(X,\Phi)$.\end{lemma}

We now relate the superenvelopes of the flow with those of its time-$t$ map for $t>0$.
\begin{lemma}\label{bz}
For any $t>0$ the map $E\mapsto tE\circ \theta_t$ defines an injective map from the set of (affine) superenvelopes for the flow to the corresponding set for the time $t$-map $\phi_t$. 
\end{lemma}
\begin{proof}
The injectivity follows from the retraction property of  $\theta_t$ (the map $E\mapsto \frac{1}{t}E\circ i_t$ defines a right inverse). It remains to check the image of a superenvelope is a superenvelope. Let $\mathcal{H}^\Phi=(h_k^\Phi)_k$ be an entropy structure of the flow. By Lemma  \ref{timet} the sequence  $t\mathcal{H}^\Phi\circ \theta _t$ is an entropy structure of $\phi_t$. Then $t\left(E\circ\theta_t-h_k^\Phi\circ \theta_t\right)=t\left(E-h_k^\Phi\right)\circ \theta_t$ and by continuity of $\theta_t$ we get 
$$\lim_k \left(t\left(E-h_k^\Phi\right)\circ \theta_t\right)^{\tilde{}}=\lim_k t\left(E-h_k^\Phi\right)^{\tilde{}}\circ \theta_t=t\left(E-h^\Phi\right)\circ \theta_t.$$ 
 \end{proof}
 
 We consider now superenvelopes of suspension flows.
\begin{lemma}\label{sus}
Let $(X_r,\Phi_r)$ be a zero-dimensional  flow given by a suspension flow over a zero-dimensional system $(X,T)$ with a positive continuous roof function $r:X\rightarrow \mathbb{R}^{+}$. The map
\[\Gamma:E\mapsto E_r:=\frac{E(\mu_\nu)}{\int r\  d\mu_\nu}\]
is a bijection between the (affine) superenvelopes of $(X,T)$ and  the (affine) superenvelopes of $(X_r,\Phi)$.
\end{lemma}

\begin{proof} Let $\mathcal{H}=(h_k)_k$ be an entropy structure of $(X,T)$ and let $\mathcal{H}_r=(g_k)_k$.  By
 continuity of  $\mathcal{M}(X,\Phi)\ni\nu\mapsto \frac{1}{\int r \, d\mu_\nu}$ we have for all $\nu\in \mathcal{M}(X,\Phi)$: 
\begin{align*}
(E_r-g_k)^{\tilde{}}(\nu)&=\frac{(E-h_k)^{\tilde{}}(\mu_\nu)}{\int r\ d\mu_\nu},\\
\lim_k (E_r-g_k)^{\tilde{}}(\nu)&=  \frac{\lim_k(E-h_k)^{\tilde{}}(\mu_\nu)}{\int r\ d\mu_\nu}.
\end{align*}
Thus $E$ is a superenvelope of $(X,T)$ if and only if $E_r$ is a superenvelope of $(X_r,\Phi_r)$. Note finally that the map $\Gamma$ is invertible with  $\Gamma^{-1}(E_r):\mu\mapsto \int r \ d\mu \times E_r(\nu_\mu)$ for any superenvelope $E_r$ of $(X_r,\Phi_r)$.

 Assume now $E$ is affine. Let us denote  by $E$ its affine extension  on the set $\mathcal{N}(X,T)$ of $T$-invariant positive finite measures (not necessarily probability ones). Similarly we denote by $\mathcal{N}(X_r,\Phi_r)$ the set of $\Phi_r$-invariant positive finite measures. 
The map $\nu\mapsto \frac{\mu_\nu}{\int r\, d\mu_\nu}$ being an affine bijection from $\mathcal{N}(X_r,\Phi_r)$ into $\mathcal{N}(X,T)$ (the inverse is given by $\mu\mapsto \mu \times \lambda$), the map $E_r$ defines an affine function on $\mathcal{N}(X_r,\Phi_r)$ and thus on the simplex $\mathcal{M}(X_r,\Phi_r)$ by restriction. Similarly $E$ is affine when $E_r$ is affine. 

\end{proof}

\subsection{Periodic structure}

For a topological flow $(X,\Phi)$ we let $Per(\Phi)$ be the set of $\Phi$-periodic orbits.  We denote  by $t(\gamma)$ the minimal period of $\gamma\in Per(\Phi)$.   We define the \emph{global periodic growth $p(\Phi)$} of $(X,\Phi)$ as follows : $$p(\Phi)=\sup_{t>0}\frac{1}{t}\log \sharp \{\gamma\in Per(\Phi), \ t(\gamma)\leq t\}.$$
To estimate the local exponential growth of periodic orbits, we introduce  \textit{the periodic structure} as the equivalence class for $\sim^{\prec}$ on $\mathcal{M}(X,\Phi)$ of  the following nonincreasing sequence $\mathcal{P}=(p_k)_k$ of nonnegative functions  on $\mathcal{M}(X,\Phi)$ (again we  call periodic structure any representative in this class). Recall we have fixed a nonincreasing sequence $(\epsilon_k)_k$ with $\lim_k\epsilon_k=0$. We let $D=D^X$ be a convex distance on the set $\mathcal{M}(X)$ of Borel probability measures on $X$ inducing the weak-$*$ topology, e.g. with a dense countable family $(f_n)_n$ of real continuous nonzero  functions on $X$
\[\forall \mu, \nu\in \mathcal{M}(X,\Phi),\  \ D(\mu, \nu)=\sum_n\frac{\left| \int f_n\, d\mu-\int f_n\, d\nu\right|}{2^n\sup_x|f_n(x)|}.\] We let $\nu_\gamma$ be the  periodic measure associated to $\gamma\in Per(\Phi)$.  Then we let  for all $k$
$$p_k^\Phi(\nu_\gamma)=\frac{1}{t(\gamma)}\log\sharp \{\gamma'\in Per(\Phi), \ D(\nu_\gamma, \nu_\gamma')<\epsilon_k\text{ and } t(\gamma')\leq t(\gamma)\}.$$
The functions $p^\Phi_k$ are then extended harmonically on the simplex $\mathcal{M}(X,\Phi)$ by letting $p^\Phi_k(\nu)=0$ for any aperiodic measure $\nu$. We get in this way a nonincreasing sequence $\mathcal{P}=(p_k^\Phi)_k$ of nonnegative functions on $\mathcal{M}(X,\Phi)$.  \emph{The tail periodic function $u^{\Phi}_1$} is then defined as 
$$u_1^{\Phi}= \lim_k\widetilde{p_k^{\Phi}}.$$
When the global periodic growth   $p(\Phi)$ is finite,  the sequence $(p_k^\Phi)_k$ is converging pointwisely to zero and $$u_1^\Phi\leq p(\Phi)<+\infty.$$ The equivalence class of $\mathcal{P}$  (and thus $u_1^\Phi$ by Lemma \ref{supere}) depend  neither on the choice of the sequence $(\epsilon_k)_k$ nor on the distance $D$.

Similarly we define $u_1^T$ for a discrete system $(X,T)$ as $u_1^{T}= \lim_k\widetilde{p_k^{T}}$ with $p_k^T$ harmonic, vanishing on aperiodic measures and   $p_k^T(\mu_x)=\frac{1}{n}\log\sharp \left\{\mu_{x'}, \ D(\mu_x',\mu_x)<\epsilon_k\text{ and }T^{n}x'=x'\right\}$ for any periodic point $x\in X$ with minimal period 
$n$ (where  $\mu_x$ denotes here the periodic measure associated to $x$). A similar quantity  $\mathfrak u_1^T$  was first defined in \cite{bdo} by letting 
$\mathfrak u_1^{T}= \lim_k\widetilde{\mathfrak p_k^{T}}$
 with   $\mathfrak p_k^T$ harmonic satisfying $\mathfrak p_k^T(\mu_x)=\frac{1}{n}\log\sharp \left\{x', \ x'\in B(x,\epsilon_k,n) \text{ and }T^{n}x'=x', \, T^kx'\neq x'  \text{ for }k<n\right\}$ for any periodic point $x\in X$ with minimal period $n$. Obviously we have 
 $\mathfrak u_1^{T}\leq  u_1^{T}$. 
  For a subshift $(X,T)$ we clearly have $\mathfrak u_1^{T}=0$. In this case $u_1^{T}=0$ also holds true (see Lemma \ref{fdf} in Appendix C).  For a topological discrete system $(X,T)$ we also let $p(T)$ be \textit{the global periodic growth} $p(T)=\sup_{n>0}\frac{1}{n}\log\sharp\{x\in X, \ T^{n}x=x \}$.

\begin{lemma}\label{last}Let $\pi:(Y,\Psi)\rightarrow (X,\Phi)$ be an   isomorphic extension then 
$$u_1^{\Phi}\circ \pi=u_1^\Psi.$$
\end{lemma}
\begin{proof}
The proof follows directly from the fact, that the induced map $\pi:\mathcal{M}(Y, \Psi)\rightarrow \mathcal{M}(X, \Phi)$ is a homeomorphism preserving the periodic measures and their periods. 
\end{proof}

For a topological flow $(X,\Phi)$ one checks easily that the time $t$-map, $t\neq 0$, satisfies $u_1^{\phi_t}(\mu)=0$ for $\mu\notin \mathcal{M}(X,\Phi)$ and  $\frac{u_1^{\phi_t}}{t}(\mu)\leq u_1^{\Phi}(\mu)$ for $\mu\in \mathcal{M}(X,\Phi)$. However this last inequality may be strict. We investigate now the behaviour of $u_1$ under suspensions.

\begin{lemma}\label{sd}
Let $(X_r,\Phi_r)$ be a zero-dimensional flow given by a suspension flow over a zero-dimensional system $(X,T)$ with a positive continuous  roof function $r$. Then we have  for all $\nu \in \mathcal{M}(X_r,\Phi_r)$:
$$u_1^{\Phi_r}(\nu)=\frac{u_1^T(\mu_\nu)}{\int  r \,d\mu_\nu}.$$
In particular $u_1^{\Phi_r}=0$ when $(X,T)$ is a subshift.
\end{lemma}
\begin{proof}
Let us show $u_1^{\Phi_r}(\nu_\mu)\leq \frac{u_1^T(\mu)}{\int  r \,d\mu}$, 
the other inequality being proved similarly by reversing the roles of $T$ and $\Phi_r$. For all $\mu\in \mathcal{M}(X,T)$ we 
let  $\overline{p_k^T}(\mu)= (1+\epsilon_k)\sup  \{
p_k^T(\mu'),\ D(\mu,\mu')\leq \epsilon_k\}$. From  Lemma 5.4 in \cite{bdo} it follows that 
$u_1^T=\lim_k\overline{p^T_k}$. As the functions $\mu\mapsto \overline{p_k^T}
(\mu)$  are upper semicontinuous  it is enough to show that for any fixed $l$
there is $k$  with 
\begin{eqnarray}\label{dddkhkr}
\forall \nu\in \mathcal{M}(X_r,\Phi_r), \ \ \ p_k^{\Phi_r}(\nu)\int  r 
\,d\mu_\nu&\leq \overline{p_l^T}(\mu_\nu).
\end{eqnarray} The distance $D^{X}$ being convex and the function $p_l^T$ being affine,  the function $\overline{p_k^T}$
is concave and therefore superharmonic (as a  concave upper semicontinuous function).  We may extend  $p_k^{\Phi_r}$ on $\mathcal{N}(X_r,\Phi_r)$  so that
$\mu\mapsto p_k^{\Phi_r}(\nu_\mu)\int  r 
\,d\mu=p_k^{\Phi_r}(\mu\times \lambda)$ is harmonic. Therefore to prove the inequality (\ref{dddkhkr})
 we can assume $\mu_\nu$ (therefore $\nu$) to be periodic without loss of generality. Let $\gamma$ be the associated  periodic orbit of the flow. We let $k$ be so large that for any $\mu, \mu'\in \mathcal{M}(X,T)$ with $D^{X_r}(\nu_\mu,\nu_{\mu'})<\epsilon_k$ we have $D^X(\mu,\mu')<\epsilon_l/2$ and $\frac{\int r\, d\mu}{\int r\, d\mu'}<1+\epsilon_l$. 

To simplify the notations we let $\mu_\gamma$ be the $T$-periodic measure $\mu_{\nu_\gamma}$ for a $\Phi_r$-periodic orbit $\gamma$. We pick up  a periodic orbit $\overline{\gamma'}$ in the set of periodic orbits $
\gamma'$ with $t(\gamma')\leq t(\gamma)$ and $D^{X_r}(\nu_{\gamma'},
\nu_{\gamma})<\epsilon_k$ such that the period $n_{\mu_{
\overline{\gamma'}}}$ of $\mu_{\overline{\gamma'}}$ maximizes the 
period of the $\mu_{\gamma'}$'s.
 The homeomorphism $\Theta : \mathcal{M}(X, T) \rightarrow  \mathcal{M}(X_r, 
 \Phi_r)$ maps the $T$-periodic measures $\mu_x$  of period $n$ to 
 the $\Phi_r$-periodic measures $\mu_\gamma$  of period $n\int r \,d
 \mu_x$.  Therefore we get \begin{align*}
 p_k^{\Phi}(\nu_\gamma)t(\gamma)&\leq 
 p_l^T(\mu_{\overline{\gamma'}})n_{\overline{\gamma'}},\\
 &\leq 
 p_l^T(\mu_{\overline{\gamma'}})\frac{t(\overline{\gamma'})}{\int r\, d
 \mu_{\overline{\gamma'}}},\\
 & \leq p_l^T(\mu_{\overline{\gamma'}})
 \frac{t(\gamma)}{\int r\, d\mu_{\overline{\gamma'}}},\\
 \text{ and thus }p_k^{\Phi}(\nu_\gamma)\int r\, d\mu_{\gamma}& \leq  
(1+\epsilon_l) p_l^T(\mu_{\overline{\gamma'}})\leq \overline{p_l^T}(\mu_\gamma).
 \end{align*}

\end{proof}

\begin{rem}
One can show $p(\Phi)\leq h_{top}(\Phi)+\sup_{\mu}u^\Phi_1(\mu)$. We did not provide a proof as this inequality will not be used in the following. We refer to Section 6 in \cite{bdo} for the analogous inequality in the discrete case. 
\end{rem}
\subsection{Expansiveness and asymptotical expansiveness}
Following R.Bowen and P.Walters a topological flow $(X,\Phi)$ is said \textit{expansive} when  $\forall \epsilon>0$ $\exists \delta>0$ such that if $d(\phi_t(x), \phi_{s(t)}(y))<\delta$ for all $t\in \mathbb{R}$ and for a continuous map $s: \mathbb{R} \rightarrow  \mathbb{R}$ with $s(0)=0$, then $y=\phi_t(x)$ with $|t|<\epsilon$. Recall  that a discrete dynamical system is expansive whenever there is $\epsilon>0$ with $\bigcap_{n\in \mathbb{Z}}T^{-k}B(T^kx,\epsilon)=\{x\}$ for all $x\in X$. The expansiveness property is invariant under topological conjugacy. The global periodic growth and the topological entropy of an expansive discrete system  (resp. expansive flow)   is finite (see Theorem 5 in \cite{BW}). 

R.Ma\~n\'e  has proved that expansive dynamical systems only act on finite dimensional  metric spaces \cite{man}. This result was extended to flows by H.R.Keynes and M.Sears \cite{key}. By Proposition \ref{SMB}  any $C^2$ smooth expansive flow satisfies the smooth small flow boundary property. In this case Lemma \ref{prec} gives a positive answer to the aforementioned open question of R.Bowen and P.Walters.  

Furthermore expansiveness is preserved by suspension :
\begin{theo}\label{gras}(Theorem 6 in  \cite{BW})
Let $(X_r,\Phi_r)$ be a zero-dimensional flow given by a suspension flow over a zero-dimensional system $(X,T)$ with a positive continuous  roof function $r$. The flow $(X_r,\Phi_r)$ is expansive if and only if $(X,T)$ is expansive. 
\end{theo}

M.Misiurewicz  introduced in \cite{mis} the asymptotic $h$-expansiveness property for topological systems. A topological system $(X,T)$ is \textit{asymptotically $h$-expansive} when $$\lim_{\epsilon\rightarrow 0}\sup_{x\in X} h_{top}\left(B_T(x,\epsilon,\infty)\right)=0.$$ Asymptotical $h$-expansiveness is invariant under topological conjugacy \cite{mis}, even under principal extensions \cite{led}. The metric entropy of an asymptotical $h$-expansive system is upper semicontinuous. In particular such a system always admits a measure of maximal entropy.  T.Downarowicz characterizes asymptotical $h$-expansiveness in terms of entropy structure as follows :
\begin{theo}\label{vieu}(Theorem 9.0.2 in \cite{dow}) A topological system is asymptotical $h$-expansiveness if and only if any (some) entropy structure of $(X,T)$ is converging uniformly to the entropy function $h$.
\end{theo}

For a topological flow $(X,\Phi)$ it is easily seen that $(X,\phi_1)$ is asymptotically $h$-expansive if and only if  so does $(X,\phi_t)$ for any $t\neq 0$. In this case the flow will be said \textit{asymptotically $h$-expansive}.\footnote{For topological flows, R.F.Thomas has defined and studied another notion of $h$-expansiveness in \cite{Rom}.}

A topological system $(X,T)$ (resp. flow $(X,\Phi)$) is said to be \textit{asymptotically expansive} when it is asymptotically $h$-expansive and $u_1^T=0$ (resp. $u_1^\Phi=0$). From the definitions one easily checked that $u_1^T=0$ (resp. $u_1^\Phi=0$) if and only if the periodic structure $\mathcal{P}=(p_k)_k$ of $(X,T)$ (resp. $(X,\Phi)$) is converging uniformly to zero. As for discrete systems,  asymptotical $h$-expansiveness may be also characterized by the uniform convergence of entropy  for flows with the small flow boundary property. 

\begin{lemma}\label{neww}
Let $(X,\Phi)$ be a topological flow with the small flow boundary property.  The following properties are equivalent:
\begin{enumerate}[i)]
\item $(X,\Phi)$ is asymptotically $h$-expansive (resp. asymptotically expansive),
\item any (some) entropy structure $\mathcal{H}=(h_k)_k$ is converging uniformly  to $h$ (resp. and any (some)  periodic structure is converging uniformly to zero). 
\end{enumerate}
\end{lemma}
\begin{proof}It is enough to deal with the asymptotical $h$-expansiveness because we already observed that $u_1^\Phi$ is equal to zero if and only if any (some) periodic structure is converging uniformly to zero. The proof then follows from the following equivalences : 
\begin{center}$(X,\Phi)$  is asymptotically $h$-expansive, \\
$\Longleftrightarrow$,\\
$(X,\phi_1)$  is asymptotically $h$-expansive,\\
$\Longleftrightarrow$\\
the entropy structure of $(X,\phi_1)$ is converging uniformly to $h$,\\
$\stackrel{\text{Lemma \ref{timet}}}{\Longleftrightarrow}$\\
the entropy structure of $(X,\Phi)$ is converging uniformly to $h$.
\end{center}

\end{proof}

We show now that asymptotical expansiveness  is also preserved by suspension. 

\begin{lemma}\label{susp}
Let $(X_r,\Phi_r)$ be a suspension flow over a topological system $(X,T)$ under a positive continuous roof function $r$. Then the following properties are equivalent :
\begin{enumerate}[i)]
\item $(X_r,\Phi_r)$ is asymptotically $h$-expansive (resp. asymptotically expansive),
\item $(X,T)$ is asymptotically $h$-expansive (resp. asymptotically expansive).
\end{enumerate}
\end{lemma}
\begin{proof}The proof follows from the  equivalences  :
\begin{center}
$(X,T)$ is asymptotically $h$-expansive (resp. asymptotically expansive),\\
$\stackrel{\text{Theorem \ref{vieu}}}{\Longleftrightarrow}$\\
any (some) entropy structure (resp. and any (some)  periodic structure) of $(X,T)$ is converging uniformly,\\
$\stackrel{\text{Corollary \ref{ssssu} (resp. and Lemma \ref{sd})}}{\Longleftrightarrow}$\\
any (some) entropy structure (resp. and any (some) periodic structure) of $(X_r,\Phi_r)$ is converging uniformly,\\
$\stackrel{\text{Lemma \ref{neww}}}{\Longleftrightarrow}$\\
 $(X_r,\Phi_r)$ is asymptotically $h$-expansive (resp.  asymptotically expansive).
\end{center}
\end{proof}


\subsection{ Relating symbolic extensions and uniform generators with expansiveness properties}

\subsubsection{The case of expansive systems}

Krieger's embedding theorem characterizes systems with a clopen uniform generator, or equivalently by Proposition \ref{super} systems topologically conjugate to a subshift :

\begin{theo}\label{balon}(Krieger's topological embedding theorem \cite{kr})
A discrete topological system $(X,T)$ is  topologically conjugate to a subshift if and only if the following properties hold:
\begin{itemize}
\item $X$ is zero-dimensional, 
\item $(X,T)$ is expansive. 
\end{itemize}
\end{theo}
We recall that the expansiveness property  implies the finiteness of the topological entropy and of the global periodic growth, which are both invariant under topological conjugacy. We state and  show now the analogous result for topological flows.

\begin{theo}A topological flow $(X,\Phi)$ is  topologically conjugate to a suspension flow over a subshift if and only if the following properties hold:
\begin{itemize}
\item $X$ is one-dimensional, 
\item $(X,\Phi)$ is expansive. 
\end{itemize}
\end{theo}

\begin{proof}
The necessary conditions are clear.  Conversely we assume  that $X$ is one-dimensional and $(X, \Phi)$ is expansive. As already mentioned such a flow is conjugate to a suspension flow $(Z_r, \Phi_r)$ over a zero-dimensional system $(Z,R)$ under a positive continuous roof function $r$. By Theorem \ref{gras} this zero-dimensional discrete  system is expansive.  The system $(Z,R)$ is therefore topologically conjugate to a subshift according to Theorem \ref{balon}. 
\end{proof}

\subsubsection{Symbolic extensions of a suspension flow}
We are in position to express the existence of symbolic extensions and uniform generators for a topological flow  in terms of superenvelopes. 

 For a topological extension $\pi:(Y,\Psi)\rightarrow (X,\Phi)$   and a function $g:\mathcal{M}(Y,\Psi)\rightarrow \mathbb{R}$ we let 
\begin{eqnarray*}g^{\pi}:\mathcal{M}(X,\Phi)&\rightarrow &\mathbb{R}^+,\\
  \mu &\mapsto & \sup_{\nu, \ \pi\nu=\mu}g(\nu).
\end{eqnarray*}
This notation was introduced earlier for a topological extension between discrete topological systems (see e.g. \cite{BD}).

\begin{lemma}\label{mieux}
Let $(X_r,\Phi_r)$ be a zero-dimensional flow given by a suspension flow over a zero-dimensional system $(X,T)$ with a positive continuous  roof function $r$. 

For  any symbolic extension (resp. with an embedding) $\pi:(Y,S)\rightarrow (X,T)$ of $(X,T)$ the suspension flow $(Y_{r'}, \Phi_{r'})$ over $(Y,S)$ under $r':=r\circ \pi$ defines a symbolic extension $\pi'$ (resp. with an embedding)  of $(X_r,\Phi_r)$  with $\pi'(y,t)=(\pi(x),t)$ for all $(y,t) \in Y_{r'}$ satisfying 
$$\forall \nu\in \mathcal{M}(X_r,\Phi_r),\ h^{\pi'}(\nu)=\frac{h^{\pi }(\mu_\nu)}{\int r\, d\mu_\nu}. $$ 

Moreover for any symbolic extension $\tau:(Z,\Psi)\rightarrow(X_r,\Phi_r)$ (resp. with an embedding)  there is a symbolic extension $(Y,S)$ of $(X,T)$ (resp. with an embedding) such that the suspension flow $(Y_{r'},\Phi_{r'})$,  as defined above, is topologically conjugate to $(Z,\Psi)$.
\end{lemma}

\begin{proof}
As the first part of the statement is easily checked, we  focus on the last part. Let $\tau:(Z,\Psi)\rightarrow (X, \Phi)$ be a topological  extension between two flows. If $S$ is a Poincar\'e cross-section of $(X,\Phi)$  then $S'=\tau^{-1}S$ is also a Poincar\'e cross-section of $(Z, \Psi=(\psi_t)_t)$ by Lemma \ref{poinc}. Indeed $S'$ is firstly a closed global cross-section (with a continuous return time  $t_{S'}=t_S\circ \tau$). Then, for $\zeta>0$,  the set  $\phi_{]-\zeta,\zeta[}S$ is open because $S$ has an empty flow boundary. Therefore $\psi_{]-\zeta,\zeta[}(S')= \tau^{-1}(\phi_{]-\zeta,\zeta[}S)$ is open by continuity of $\tau$ and the cross-section  $S'$ has an empty flow boundary.  Consequently $(Z,\Psi)$ is topologically conjugate to the suspension flow over $(S',T_{S'})$ under the continuous positive roof function $t_{S'}$. 
 
 In our context when $\tau:(Z,\Psi)\rightarrow (X_r,\Phi_r)$ is a symbolic extension (resp. with an embedding $\psi$) we let $S=X\times \{0\}$. According to  Theorem \ref{gras} the system $(S',T_{S'})$ is expansive (with $S'=\tau^{-1}S$). Then by Theorem \ref{balon} this system  is (topologically conjugate to) a subshift,  which defines a topological extension of $(X,T)$ (resp. with an embedding) via the projection map $\pi'=\pi|_{S'}$ (resp. and the embedding $\psi|_{S}$).
\end{proof}

\subsubsection{Characterization of Symbolic Extensions}

The main result in the entropy theory of symbolic extensions, known as the Symbolic Extension Entropy Theorem, may be stated as follows :

\begin{theo}\label{sexy}(Theorem 5.5 in \cite{BD}) Let $(X,T)$ be a topological system. The systems admits a symbolic extension (resp. principal) if and only if there exists a finite superenvelope $E$ (resp. $(X,T)$ is asymptotically $h$-expansive).

More precisely a function $E$ on $\mathcal{M}(X,T)$ equals 
$h^\pi$ for some symbolic extension $\pi$ if and only if $E$ is an affine superenvelope of the entropy structure of $(X,T)$.\end{theo}

 We show now the  corresponding statement for topological flows :

\begin{theo}\label{SEXX}Let $(X,\Phi)$ be a topological flow with the small flow boundary property. The flow admits a symbolic extension (resp. principal) if and only if there exists a finite superenvelope $E$ (resp. $(X,\Phi)$ is asymptotically $h$-expansive).

More precisely a function $E$ on $\mathcal{M}(X,\Phi)$ equals 
$h^\pi$ for some symbolic extension $\pi$ if and only if $E$ is an affine superenvelope of the entropy structure of $(X,\Phi)$.
\end{theo}
\begin{proof}
We first consider the case of a suspension flow $(X_r,\Phi_r)$ over a zero-dimensional system $(X,T)$ under a positive continuous roof function $r$.
  By Lemma \ref{sus} and  Lemma \ref{mieux} the  map 
\[f\mapsto \left[\nu\mapsto \frac{f(\mu_\nu)}{\int r\, d\mu_\nu}\right] \]
defines a bijection between affine superenvelopes   on one hand and  the entropy functions  in symbolic extensions $h^\pi$ on the other hand, for $(X,T)$ and $(X_r,\Phi_r)$. But according to Theorem \ref{sexy} the affine superenvelopes  are exactly  the functions $h^\pi$ for the discrete system $(X,T)$. Therefore the same holds for the suspension flow $(X_r,\Phi_r)$.

We deal now with the general case. By Proposition \ref{pro} any topological flow $(X, \Phi)$ with the small flow boundary property admits a principal  extension $\pi$ by a flow $(X_r,\Phi_r)$ of the previous form. Then if $E$ is an affine superenvelope of $(X, \Phi)$, it follows from Lemma \ref{mamip} that $E\circ \pi$ is also a superenvelope of $(X_r,\Phi_r)$. According to  the previous case there exists a symbolic  extension $\pi':(Y,\Psi)\rightarrow (X_r,\Phi_r)$ with $h^{\pi'}=E\circ \pi$. Then $\pi'\circ \pi$ is a symbolic extension of $(X, \Phi)$ with $h^{\pi'\circ \pi}=E$. Conversely,  from Theorem 7.5 in \cite{BD} (which applies to topological flows with the same proof), for  any symbolic extension $\pi'$ of $(X, \Phi)$ there exists a symbolic extension $\pi''$ of $(X_r,\Phi_r)$ with the same entropy function, i.e.
$h^{\pi''}=h^{\pi'}\circ \pi$. Since $h^{\pi''}$ is an affine superenvelope of $(X_r,\Phi_r)$, the entropy function $h^{\pi'}$ is an affine  superenvelope of $(X,\Phi)$.
\end{proof}

Together with Lemma \ref{bz} we get :
\begin{lemma}
Let $(X,\Phi)$  be a topological flow with the small flow boundary property. The flow admits a  symbolic extension (resp.  principal) if and only if so does its times $t$-map for some (any) $t\neq 0$.
\end{lemma}

 The fact, that $\phi_t$ admits a symbolic extensions does not depend on $t\neq 0$, was first proved by T.Downarowicz and M.Boyle in Theorem 3.4 of \cite{bdd}. For rational $t$ is was done by just considering a standard power rule for the entropy whereas  for an irrational $t$ they build explicitly a symbolic extension of $\phi_t$ from a symbolic extension of $\phi_1$ by using the coding of irrational rotations via Sturmian sequences.


\subsubsection{Characterization of Uniform generators}
In \cite{bdo} T.Downarowicz and the author  also characterize the entropy function in a symbolic extension with an embedding. The case of strongly isomorphic symbolic extension follows from \cite{bumo} as detailed in the Appendix \ref{dav}, whereas the characterization of systems topologically conjugate to a subshift first appeared in \cite{kr}.
\begin{theo}(Theorem 55 in \cite{bdo}, Main Theorem in \cite{bumo}, Krieger's topological embedding Theorem \cite{kr})\label{durud}
 Let $(X,T)$ be a topological system with the small boundary property. The system admits  a  uniform generator (resp.  essential, resp. clopen) if and only if there exists a finite superenvelope $E$ and $p(T)<+\infty$ (resp. $(X,T)$ is asymptotically expansive, resp. $T$ is expansive and $X$ is zero-dimensional).

More precisely a function $E$ on $\mathcal{M}(X,T)$ equals 
$h^\pi$ for some symbolic extension $\pi$ with an embedding  if and only if $E$ is an affine superenvelope of the entropy structure of $(X,T)$ with $E\geq h^\pi+\mathfrak u_1^T$.\end{theo}

By following the proof of Theorem \ref{SEXX} with making use of Lemma \ref{sd} we get :

\begin{theo}\label{Gener}
Let $(X,\Phi)$ be a topological flow with the small flow boundary property. 
The flow admits a  uniform generator (resp. essential) if and only if there exists a finite superenvelope $E$ and $p(\Phi) <+\infty$ (resp. the flow is asymptotically expansive).

More precisely a function $E$ on $\mathcal{M}(X,\Phi)$ equals $h^\pi$ for some symbolic extension $\pi$ with an embedding if and only if $E$ is a superenvelope of the entropy structure of $(X,\Phi)$ and $E\geq h+u_1^\Phi$.
\end{theo}

We are now in position to prove  Theorem \ref{intr} stated in  the Introduction.
\begin{proof}[Proof of Theorem \ref{intr}]
Fix $t>0$. By Lemma \ref{bz} the system $(X,\Phi_t)$ admits a (finite affine) superenvelope  (resp. the entropy function $h$ is a super envelope) if and only if so does the flow $(X,\Phi)$. Then, by Theorem \ref{sexy} and Theorem \ref{SEXX}, the time $t$-map admits a symbolic extension if and only if so does the flow. The corresponding statement for uniform generators follows from Theorem \ref{durud} and Theorem \ref{Gener}. The invariance of these properties under orbit equivalence is proved below in Theorem \ref{equi}.

\end{proof}

In general there is no relation between  uniform generators for the flow and  uniform generators for the time $t$-maps.
Indeed consider the standard suspension  (i.e. with roof function $r=1$) of the identity of a compact metrizable space $X$. Then the flow $\Phi=(\phi_t)_t$  admits a uniform generator if and only if the base $X$ of the suspension is a finite set, whereas  $\phi_t$  admits a uniform generator if and only if $t$ is irrational. Indeed when $X$ is infinite the flow $\Phi$ (resp. the times $t$-map $\phi_t$ with $t=\frac{p}{q}\in \mathbb{Q}$) has  infinitely many  periodic orbits with period $1$ (resp. $q$) and thus can not be embedded in a symbolic flow (resp. subshift). When $t$ is irrational, then $\phi_t$ has the small boundary property by the aforementioned  result of E.Lindenstrauss (Theorem 6.2 in \cite{lin}). Also $\phi_t$ is clearly aperiodic and asymtotically $h$-expansive. By Theorem 31 in \cite{bdo} it admits an (strongly isomorphic) uniform generator.

\subsection{Invariance by orbit equivalence}
R.Bowen and P.Walters have proved that   expansiveness is invariant under orbit equivalence for topological flows. Here we show :
\begin{theo}\label{equi}
The asymptotic ($h$-)expansiveness, the existence of symbolic extensions and the existence of uniforms generators are also dynamical properties invariant by orbit equivalence for topological  flows with the small flow boundary property. 
\end{theo}

\begin{proof} Let us consider two orbit equivalent topological flows via a homeomorphism $\Lambda$.  As already mentioned the orbit equivalence induces a topological conjugacy of the base systems $Y^{\mathcal{S}}$ and $Y^{\Lambda(\mathcal{S})}$ of the zero-dimensional strongly isomorphic extensions built in Proposition \ref{pro}   (but the roof functions may differ). For discrete systems,  the existence of  (principal) symbolic extensions (with an embedding), is invariant under topological conjugacy. Therefore,  both zero-dimensional flows admit such symbolic extensions  or not by Lemma \ref{susp}. 
\end{proof}

\begin{rem}As the topological entropy, the infimum of the entropy of symbolic extensions  and the minimal cardinality of  uniform generators may be modified by a change of the  time scale. In particular these quantities are not invariant under orbit  equivalence. 
\end{rem}




\section{Representation of symbolic flows}
After Ambrose's representation theorem, D.Rudolph  showed that a suspension flow over an ergodic transformation is always isomorphic to another one where the new roof function takes only two values (in general one can not hope the roof to be constant as  such suspension flows are not mixing). 

In the same spirit we wonder what is the ``simplest model" for the  roof function of a \textit{symbolic flow}, i.e. a  suspension flow $(Y_r,\Phi_r)$ over a subshift $(Y,\sigma)$ with a positive continuous roof function $r$. In this section  we will only consider   aperiodic flows. In this case there is a very nice  topological version of  Rokhlin towers for the subshift $(Y,S)$ :

\begin{lemma}(Lemma 7.5.4 in \cite{dowb}) Let $(X,T)$ be an aperiodic zero-dimensional system. For any integer $n>0$ there exists a clopen set $U_n$ such that :
\begin{itemize}
\item $\bigcup_{k=0}^{n+1} T^k U_n=X$,
\item $U_n,TU_n,..., T^{n-1}U_n$ are pairwise disjoint.
\end{itemize}
\end{lemma} 

Such a clopen set $U_n$ will be called a \textit{ $n$-marker} of $(X,T)$. We first show the roof function may be chosen almost constant.

\begin{lemma}\label{bog}Any aperiodic symbolic flow $(Y_r,\Phi_r)$ is topologically conjugate to a symbolic flow over a subshift of $\{0,1\}^{\mathbb{Z}}$ under a roof function arbitrarily close to $\frac{\log 2}{h_{top}(\Phi_r)}$. 
\end{lemma}
Equivalently  $(Y_r,\Phi_r)$ admits a Poincar\'e cross-section $S$ with return time $t_S$ arbitrarily close to $\frac{\log 2}{h_{top}(\Phi_r)}$ and with $h_{top}(T_S)\leq \log 2$.

\begin{proof}
Fix $1/2>\epsilon>0$. Let $a=\frac{\log 2 }{ h_{top}(\Phi_r)}+\epsilon$ and let $N$ be an integer larger then  $3/\epsilon$. We take $N'>N$ so large that any integer larger than $N'$ belongs to $[Na]\mathbb{N}+([Na]+1)\mathbb{N}$. 
Let $U_n$ be a $n$-marker of $(Y,S)$ with $Nn\underline{r}>N'$. The set $V_n:=U_n\times \{0\}\subset Y_r$ defines a Poincar\'e cross-section with return time $t_{V_n}$ larger than $n\underline{r}$.  In particular for any $u\in V_n$ there are positive integers $k,l$ such that $|Nt_{V_n}(u)-k[Na]-l([Na]+1)|<1/2$ and therefore $|t_{V_n}(u)-k[Na]/N-l([Na]+1)/N|<\epsilon/6$. There is a partition $P$ of $U_n$ in clopen sets such that for any two points $x$ and $y$  in the same atom of the induced partition of $V_n$  we have  $|t_{V_n}(x)-t_{V_n}(y)|<\epsilon/6$. In particular we may choose the above integers $k$ and $l$ independently of $u\in A$ for $A\in P$, i.e. there are nonnegative integers $k_A$ and $l_A$ such that $|t_{V_n}(u)-k_A[Na]/N-l_A([Na]+1)/N|<\epsilon/3$ for any $u\in A$. Finally we let $S$  be the union of $\phi_t A$ over $A\in P$ and $t\in\{k'[Na]/N, \ 0\leq k'< k_A\}\cup\{k_A[Na]/N+l'([Na]+1)/N,\ 0\leq l'< l_A\}$. The set $S$ is a Poincar\'e cross-section with return time $|t_S-a|< 1/N +\epsilon/3<2\epsilon/3$ and therefore $\frac{\log 2 }{ h_{top}(\Phi_r)}+\epsilon/3<t_S<\frac{\log 2 }{ h_{top}(\Phi_r)}+2\epsilon$. By Abramov entropy formula the topological entropy of the  first return map $T_S$ in $S$ is less than $\log 2$. As it is an (aperiodic) subshift it may be  topologically embedded in the full shift with two symbols by Krieger's topological embedding theorem \cite{kr}.
 \end{proof}

For a discrete topological system $(X,T)$  the orbit capacity $\ocap(E)$  of a subset $E$ of $X$ is defined as follows:
$$\ocap^T(E)=\lim_{n\rightarrow +\infty}\frac{1}{n} \sup_{x\in X}\sharp\{0\leq k\leq n, \ T^kx\in E\}.$$
Similarly for a topological flow $(X,\Phi)$ we let 
$$\ocap^\Phi(E)=\lim_{\tau\rightarrow +\infty}\frac{1}{\tau}\sup_{x\in X}\lambda\left(\{t\in [0,\tau], \ \phi_t(x)\in E\}\right).$$
Note that the limits are well defined by Fekete and Hille subadditive Lemma.
When $E$ is a closed subset of $X$ we have  $\ocap(E)=\sup_{\mu}\mu(E)$ where the supremum holds over all invariant probability measures $\mu$.  By Lemma \ref{sss} a closed cross-section $S$ of time $\eta$ has a small flow boundary if and only if $\ocap^\Phi(\partial^\Phi S_\eta)=0$.\\

We may refine Lemma \ref{bog} under the following form (similar to the representation in Rudolph's theorem). For a subshift $Y$ over a finite  alphabet $\mathcal{A}$ and for $a_{-k},\cdots, a_0\in \mathcal{A}$, we let 
$[a_{-k}\cdots a_0]$  be the cylinder set 
$$[a_{-k}\cdots a_0]:=\{(y_n)_n\in Y, \ y_{-l}=a_{-l}\text{ for }l=0,\cdots, k\}.$$
Moreover for $a\in \mathcal{A}$ and $l\in \mathbb{N}\setminus \{0\}$, we let $a^l$ be the subword given by $a^l:=\underbrace{a\cdots a}_{l\text{ times}}$.

\begin{lemma}\label{dex}
Let $(Y_r,\Phi_r)$ be an aperiodic symbolic flow over a subshift $(Y,\sigma)$. Then  for any rationally independent positive real numbers $p$ and  $q$ with $h_{top}(\Phi_r)<\frac{2\log 2}{p+q}$, for any $\epsilon>0$ and for any $\delta\in ]0,\min(p,q)[$ the flow is topologically conjugate to  a symbolic flow over a subshift $(Z,T)$ \footnote{the shift map on $Z$ is denoted here by $T$ to avoid any confusion with $(Y,\sigma)$.} of $\{0,1,2\}^{\mathbb{Z}}$ over a roof function $r'$ satisfying for $z=(z_n)_n$:
\begin{itemize}
 \item $\{r'=p\}=[0]$ and $\ocap^T( [0] )\leq\frac{1}{2}$,
 \item $\{r'=p\}=[1]$ and $\ocap^T( [1] ) \leq \frac{1}{2}+\epsilon$,
 \item $\{0<r'<\delta\}=[2]$ and $\ocap^T([2] )<\epsilon$. 
\end{itemize}
\end{lemma}

Equivalently $(Y_r,\Phi_r)$ admits a Poincar\'e cross-section $S$ together a clopen uniform generator $\{\mathfrak P,\mathfrak Q,\mathfrak R\}$ of $(S,T_S)$ with  $t_S=p$ on $\mathfrak P$, $t_S=q$ on $\mathfrak Q$ and $t_S<\delta$ on $\mathfrak R$ and with $\ocap^{T_S}(\mathfrak P)\leq\frac{1}{2}$, $\ocap^{T_S}(\mathfrak Q) \leq\frac{1}{2}+\epsilon$ and  $\ocap^{T_S}(\mathfrak R)<\epsilon$.  The condition on the orbit capacity implies in particular that $\frac{1}{2}-2\epsilon\leq \mu([i])\leq \frac{1}{2} +\epsilon$ for any $\mu\in \mathcal{M}(Z,T)$ and for $i=0,1$.

\begin{proof}
Fix $\delta>0$, $\epsilon>0$ and rationally independent positive real numbers $p$ and  $q$ with $h_{top}(\Phi_r)<\frac{2\log 2}{p+q}$. 
 By Lemma \ref{bog} we may assume
$h_{top}(\sigma)<\log 2$ and $r\simeq \frac{\log 2}{h_{top}(\Phi_r)}$.

For any $x\in \mathbb{R}$ we let  $D(x)=\min\{x-(kp+lq)\geq 0\ :  \ k,l\in\mathbb{N} \text{ with }\frac{1}{1+\epsilon}\leq k/l\leq 1 \}$.  
As $p$ and $q$ are rationally independent we have $\lim_{x\rightarrow +\infty}D(x)=0$. Fix $\epsilon\in ]0,\frac{\log 2-h_{top}(\sigma) }{2}]$ small and take $N$ with $D(x)<\delta/2$ for $x>N$. We argue then  as in the proof of Lemma \ref{bog}. Let $U_n$ be a $n$-
marker of $(Y,\sigma)$ with $n>\max( N/\underline{r},1/\epsilon )$. The set $V_n=U_n\times \{0\} \subset Y_r$ defines a Poincar\'e cross-section with return time $t_{V_n}$ larger than $n\underline{r}$. In particular for any $u\in V_n$ there are positive integers $k,l
$ such that $|t_{V_n}(u)-kp-lq|<\delta/2$. There is a partition $P$ of $U_n$ in clopen sets such that for any two points $x$ and $y$  in 
the same atom of the induced partition of $V_n$ we have  $|t_{V_n}(x)-t_{V_n}(y)|<\delta/2$. In particular we may choose the above 
integers $k$ and $l$ independently of $u\in A$ for $A\in P$, i.e. there are nonnegative integers $k_A$ and $l_A$ such that $0<t_{V_n}
(u)-k_Ap-l_Aq<\delta$ for any $u\in A$. We may assume $P$ is finer than $Q^n$ with $Q$ being the zero coordinate of $Y$. Moreover 
we choose $\epsilon>0$ small and $n$ large enough so that the cardinality of the set of $n$-words in $Y$ is  less than $e^{n(h_{top}
(\sigma)+\epsilon)}\leq \binom{k_A}{2k_A}$ for any $A\in P$. Indeed $\binom{m}{2m}\sim\frac{2^{2m}}{\sqrt{\pi m} }$ and we have for 
small $\epsilon>0$
\begin{align*}
k_A(p+q)& \simeq k_Ap+l_Aq,\\
& \simeq nr,\\ 
&\simeq n\frac{\log 2}{h_{top}(\Phi_r)},\\
& > n(p+q)/2.
\end{align*}

Following D.Rudolph we may now  encode the system $(Y,\sigma)$ by ordering the subdivision of $[0,k_Ap+l_Aq]$ into $k_A$ intervals of length $p$ and $l_A$ intervals of length $q$.   For any $k_A$ we fix a bijection between the set of $n$-words of $Y$ and the $2k_A$-uple of $0$ and $1$ with exactly $k_A$ terms equal to $0$ and $1$. 
To any $A\in P$ we let $w^A=(w_{k''}^A)_{k''}$ be the $2k_A$-uple associated to  the element of $Q^n$ containing $A$. For any $k'\leq 2k_A$ we let\footnote{For $i=0,1$ we let  $\delta_i(x)=1$ if $x=i$ and $0$ if not.} $t_{k'}(A)=\sum_{ k''\leq k'}\left(p\delta_1(w^A_{k''})+q\delta_0(w^A_{k''})\right)$ and for any $2k_A<k'\leq k_A+l_A$ we let $t_{k'}(A)=t_{2k}(A)+(k'-2k_A)q$. Finally we let $S$  be the union of $\phi_t A$ over $A\in P$ and $t\in\{t_{k'}(A), \ k'\leq k_A+ l_A\}$. 
The set $S$ is a Poincar\'e cross-section with return time $t_S\in \{p,q\}\cup ]0,\delta[$. By construction the associated clopen partition of $S$ consisting of  $\mathfrak P=\{t_S=p\}$, $\mathfrak Q=\{t_S=q\}$ and 
$\mathfrak R=\{t_S\in ]0,\delta[\}$ is generating. Moreover $\ocap(T_S\in \mathfrak R)\leq \min_A\frac{1}{2k_A}<\frac{1}{n}<\epsilon$. As $\frac{1}{1+\epsilon}\leq k_A/l_A\leq 1$ we also have $\ocap(T_S\in \mathfrak P)\leq 1/2$ and $\ocap(T_S\in \mathfrak Q)\leq \frac{1}{2}+\epsilon$. 
\end{proof}

\begin{rem}
We would like to remove the remaining set $\mathfrak{R}$ by using a multiscale approach for a sequence of nested $n$-markers as D.Rudolph did for the ergodic case. One can follow this procedure. In this way one  gets a Borel section $S$ with return time $t_S$ in $\{p,q\}$, such that the partition $\{t_S=p\}$, $\{t_S=q\}$ of $S$ is a generator for the induced Borel system on $S$. Unfortunately the obtained generator is not uniform. Indeed to approach the base of the $k^{th}$ tower with an error term of size $\epsilon_k$ one needs to reencode a piece of orbit of length $l_k$ with $l_k\rightarrow +\infty$ when $\epsilon_k\rightarrow 0$, so that  the limit map does not admit a priori a continuous inverse.  
\end{rem}

\begin{ques}
Does an aperiodic symbolic flow admit a symbolic extension with an embedding given by a suspension flow over a subshift of $\{0,1\}^{\mathbb{Z}}$ with a roof function constant on the two atoms of the zero-coordinate partition? Note that  if $(X,\Phi)$ has periodic orbits one can not always ensure the roof function is two-valued. Indeed any period should then  belong to $\mathbb{N}p+\mathbb{N}q$ where $p$ and $q$ are the values of the roof function. 
\end{ques}

\begin{lemma}\label{dep}
Let $(Y_r,\Phi_r)$ be an aperiodic symbolic flow over a subshift $(Y,\sigma)$. Then  for any rationally independent positive real numbers $p<q$ with $h_{top}(\Phi_r)<\frac{2\log 2}{p+q}$, for any integer  $M\geq 2$ and for any $\delta>0$ the flow is topologically conjugate to  a symbolic flow over a subshift $(Z,T)$ of $\{0,1\}^{\mathbb{Z}}$  over a roof function $r'$ satisfying for some positive integer $K$ :
\begin{itemize}
 \item   $\{r'=p\}=[1]$, 
 \item $\{r'\in [q,  q+\delta]\}=[0]$,
 \item $\{r'>q\}=T([0^{M+K}10^K1]),$
\end{itemize}

Moreover we have    $$Z=\bigcup_{0\leq k<+\infty}  T^{k}\{r'>q\}.$$
\end{lemma}

\begin{proof}
We only have to slightly modify the construction  in Lemma \ref{dex} as follows (we keep the notations of that proof). When encoding the $2k_A$-uple $w^A$ associated to $A\in P$ we may always start and  finish with the letter  $1$, avoid a sequence of $K$ consecutive $0$'s for a large enough integer $K$ and take  only $k_A-1$ (not $k_A$) terms equal to $1$.  Moreover,  we can also assume   $l_A\geq k_A+M+K+2$. All these requirements may be established by taking $n$ large enough. Then we extend $w^A$ to a  $(k_A+l_A-1)$-uple $v^A=(v^A_{k''})_{k''}$ by adding to $w^A$ a suffix of the form $0^{L}10^K$ with $L\geq M+K$. 

  Then we consider the Poincar\'e cross-section  $S$ defined by the union $\phi_t A$ over $A\in P$ and $t\in\{ t_{k'}(A), \ k'< k_A+ l_A\}$
  with $t_{k'}(A)=\sum_{ k''\leq k'}\left(p\delta_1(v^A_{k''})+q\delta_0(v^A_{k''})\right)$. The return time in $S$ is now either equal to $p$ 
  or in  $[q,q+\delta]$. Moreover it is larger than $q$ if and only if   the first return in $S$ lies in $U_n\times \{0\}$.  Finally  the partition $\{ t_S =p \}$, $\{ t_S\geq q\}$ of $S$ defines again a clopen generator of $(S,T_S)$.
  \end{proof}


For the time $t$-map  of a topological flow  we define the following weaker notion of uniform generators.
 \begin{defi} Let $(X,\Phi=(\phi_t)_t)$ be a topological flow. For $\alpha>0$ and $t\neq 0$ a partition $P$ is said to be \emph{an $\alpha$-uniform generator of $\phi_t$} when 
$\sup_{y\in P_{\phi_t}^{[-n,n]}(x)}d\left(y, \phi_{[-\alpha,\alpha]}(x)\right)$ goes to zero uniformly in $x\in X$.\end{defi}
 The above definition does not depend on the choice of the metric, but only on the topology of $X$ (the same holds for uniform generators). In particular $\alpha$-uniform generators are preserved by topological conjugacy. Clearly any uniform generator of $\phi_t$ is an  $\alpha$-uniform generator of $\phi_t$ for all $\alpha$. 

\begin{lemma}\label{Rudf}
Let $(Y_r,\Phi_r)$ be an aperiodic symbolic flow over a subshift $(Y,\sigma)$. Then  for any $t\in ]0,\frac{\log 2 }{ h_{top}(\Phi_r)}[$ and for any $\alpha>0$, the time $t$-map $\phi_t$   
admits an $\alpha$-uniform generator given by the towers associated to a clopen $3$-partition of a Poincar\'e cross-section. 
\end{lemma}

\begin{proof}
We let $p=t\in ]0,\frac{\log 2 }{ h_{top}(\Phi_r)}[$ and we take $q$ rationally independent from $p$  with  $\alpha=q-p\in ]0,p[$   so small  that we have $h_{top}(\Phi_r)<\frac{2\log 2}{p+q}$. Without loss of generality we may assume $(Y_r,\Phi_r)$ is the model $(Z_{r'},\Phi_{r'})$ 
given by Lemma \ref{dep} with respect to $p,q,\epsilon$ and $\delta=\alpha=q-p$.    Let  $M\geq 2$ be so large  that  for all $s\in [0,q[$ there exists $0\leq u<  M$ and $0\leq v\leq u+1$  with $s+up=vq+\beta$ for $0\leq \beta<\alpha$. 
We let $\mathcal T$ be the $2$-partition $\{\mathfrak P, \mathfrak Q\}$ of the Poincar\'e cross-section $Z\times \{0\}$ given by 
$\mathfrak P=[0]\times \{0\}$ and $\mathfrak Q=[1]\times \{0\}$.  We let $\overline{r}=\sup_{y\in Y}r(y)$. We consider the compact space $\tilde{Y}=Y\times [0,\overline{r}]$ endowed with the metric $d_{\tilde{Y}}$ given by $d_{\tilde{Y}}\left((x,t), (y,s)\right)=d_Y(x,y)+|t-s|$. We also let $\pi_r: \tilde{Y}\rightarrow Y_r$ be the (uniformly) continuous map which associates to any $(y,t)\in \tilde{Y}$  the point $\phi_t^r(y,0)$ in $Y_r$ (with $(y,0)\in Y_r$ and $\Phi_r=(\phi_t^r)_t$).   Fix some    metric $d_{Y_r}$ on $Y_r$, for example the Bowen-Walters metric (see \cite{BW}). Finally we let $w:\mathbb{R}^+\rightarrow \mathbb{R}^+$ with $\lim_{\epsilon\rightarrow 0}w(\epsilon)=0$ be a modulus of uniform equicontinuity of $(\phi_\beta)_{|\beta|\leq \alpha}$ and $\pi_r$, i.e.  
\[\forall \beta\in [-\alpha, \alpha] \  \forall \mathcal{Z}, \mathcal{Z}'\in Y_r, \ \ d_{Y_r}(\phi_\beta (\mathcal{Z}),\phi_\beta (\mathcal{Z}'))<w(d_{Y_r}(\mathcal{Z}, \mathcal{Z}'))\] 
and 
\[\forall \tilde{u},\tilde{v}\in  \tilde{Y}, \ \ d_{Y_r}\left(
 \pi_r (\tilde{u}),\pi_r(\tilde{v})\right) <w(d_{\tilde{Y}}\left(
 \tilde{u},\tilde{v}\right).  \]

The $3$-partition  $\mathsf T_\mathcal{R}$ in towers of $Y_r$ associated to the partition $\mathcal{R}:=\{\mathfrak{P}, \mathfrak Q, \phi_\alpha\mathfrak Q\}$ of the Poincar\'e  cross-section $S'=S \cup\phi_\alpha \mathfrak Q$ with $S=Y\times \{0\}$ is an $\alpha$-uniform generator of $\phi_t$ (recall $\alpha=q-p$).   Indeed we claim that, for any positive integer $n$, for any $\mathcal{X}=(x,s)\in Y_r$ with $0\leq s<r(x)$ and for any $\mathcal{Y}\in \mathsf T^{ [-2n,2n]}_\mathfrak{R}(\mathcal X) $, there exists $\beta$ with $|\beta|\leq \alpha$ such that $\phi_\beta(\mathcal Y)=\pi_r(y,s)$   with $y\in Q^{[-n, n]}(x)$, where  $Q$ denotes  the zero-coordinate partition of $Y$. Then we have 
\begin{eqnarray*}
 d_{Y_r}(\mathcal{Y},\phi_{[-\alpha, \alpha]}\mathcal{X}) & \leq & w(d_{Y_r}(\phi_\beta(\mathcal{Y}),\mathcal{X}),\\
 & \leq & w\circ w\left(d_{\tilde{Y}}\left( (y,s),(x,s)\right)\right),\\
 &\leq &  w\circ w\left(\diam \left(Q^{[-n, n]}(x)\right)\right)\xrightarrow{n\rightarrow +\infty}0 \text{ uniformly in $x\in Y$, thus in $\mathcal{X}\in Y_r$.}
 \end{eqnarray*}

   We show now the above claim. Let $\mathcal{N}^{\phi_t}$ be the $\mathsf T_\mathcal{R}$-name of $\mathcal{X}$ with respect to $\phi_t$, i.e. $\mathcal{N}^{\phi_t}=\left(\mathsf T_{\mathcal{R}}(\phi_{kt}\mathcal{X} )\right)_{k\in [-2n,2n]}$.  Any letter $T_{\mathfrak Q}$ is followed by $T_{\phi_\alpha \mathfrak Q}$ in $\mathcal{N}^{\phi_t}$ but both correspond to the same return in $S$.  Then a subword of $\mathcal{N}^{\phi_t}$ of the form $ T_\mathfrak{P} \mathsf T_1... \mathsf T_{K''}\mathsf T_\mathfrak{P}$, with $\mathsf T_i=\mathsf T_{\mathfrak Q}$ or $\mathsf T_i=\mathsf T_{\phi_\alpha\mathfrak Q}$ for $i=1,\cdots, K''$, and $\sharp \{i, \,\mathsf T_i =\mathsf T_{\phi_\alpha\mathfrak Q}\}\in \{K,K+1\}$ indicates the return times  in $\{r'>q\}$. These subwords are called the marking subwords.   Then any subword  $\mathsf T^L_\mathfrak{P}$ of $\mathcal{N}^{\phi_t}$  between two such consecutive marking subwords correspond  to exactly  $L$ consecutive returns of $T_S$ in $\mathfrak P$  because  the associated return times in $S'$ are equal to $t=p$. This is also the case of the subwords  $\mathsf T^L_{\phi_\alpha\mathfrak Q}$ of  $\mathcal{N}^{\phi_t}$, whose last letter is not the penultimate letter  of a marking subword. In this sole case, the return time in $S'$ from $\phi_\alpha\mathfrak Q$ may differ from $p$, but we know the subword in $Y$ associated to a marking subword  is given by $10^K1$.  Combining these facts, the $Q$-name of $x$  is obtained from $\mathcal{N}^{\phi_t}$ by first replacing the marking subwords by $10^K1$,  then by  deleting the letters $\mathsf T_{\mathfrak{Q}}$ and finally by replacing the remaining letters $T_{\mathfrak P}$ and $\mathsf T_{\phi_\alpha \mathfrak{Q}}$ respectively by $0$ and $1$. For example, if  $\mathcal{N}^{\phi_t}$ is  the following sequence (where we write  the marking subwords in blue, the zero-coordinate in green and the large block of $0$'s before the marking subword in orange)  $$\cdots\textcolor{orange}{ \mathsf T_{\phi_\alpha\mathfrak Q}\cdots\mathsf T_{\phi_\alpha\mathfrak Q}}\ \textcolor{blue}{\mathsf T_{\mathfrak{P}}\mathsf T_{\mathfrak{Q}}\mathsf T_{\phi_\alpha\mathfrak Q}\cdots  \mathsf T_{\phi_\alpha\mathfrak Q}\mathsf T_\mathfrak{P}} \ \mathsf T_{\phi_\alpha\mathfrak Q} \textcolor{green}{\mathsf T_\mathfrak{P}}\mathsf T_\mathfrak{Q}T_{\phi_\alpha\mathfrak Q} \cdots  \mathsf T_{\mathfrak{P}} \ \textcolor{orange}{ \mathsf T_{\phi_\alpha\mathfrak Q}\cdots\mathsf T_{\phi_\alpha\mathfrak Q}}\  \textcolor{blue}{\mathsf T_{\mathfrak{P}}\mathsf T_{\phi_\alpha\mathfrak{Q}}\mathsf T_{\phi_\alpha\mathfrak Q}\cdots \mathsf T_{\phi_\alpha\mathfrak Q}\mathsf T_\mathfrak{P}} \cdots$$  we obtain the  subword of $Y$ given by  $\cdots\textcolor{orange}{0^{L}}\ \textcolor{blue}{10^K1} \ 0\textcolor{green}{1}0\cdots 1 \  \textcolor{orange}{0^{L'}}\ \textcolor{blue}{10^K1}\cdots $  for some $L,L'\geq K+M\geq K+2$. As we delete at most one letter in two, it contains $Q_T^{[-n,n]}(x)$ as a subword.  Thus for  any $\mathcal{Y}\in \mathsf T^{ [-2n,2n]}_\mathfrak{R}(\mathcal X)$ there is $(y,u)\in \tilde{Y}$ with  $0\leq u<r(y)$ and $y\in Q_T^{[-n,n]}(x) $. Let $N$ be a positive integer with $Y=\bigcup_{0\leq k< N}\sigma^{k}\{r>q\}$.
 To conclude the proof of the claim we  show that  there exists  $\beta$ with $|\beta|\leq \alpha$ and $s=u+\beta$, whenever $n$ is larger than $N$.  It is enough to see that any orbit  of $\phi_t$ visits at least one time the tower $\mathsf T_{\mathfrak{Q}}=\phi_{[0,\alpha[}\mathfrak Q$ of height $\alpha$   between two return times of the flow in  $\{r>q\}$.  But this follows easily from the choice of $M$ and the presence of more than $M$ consecutive $0$'s  before any return in $\{r>q\}$ (which corresponds to the blocks in orange in the above example). 
 \end{proof}

We are now in position to prove Theorem \ref{deux} stated in the introduction. 
\begin{proof}[Proof of Theorem \ref{deux}]
Let $(X,\Phi=(\phi_t)_t)$ be an aperiodic flow with a uniform generator given by a symbolic  extension   $\pi:(Y_r,\Phi_r=(\phi^r_t)_t)\rightarrow (X,\Phi)$ with an embedding $\psi$. We recall that it means $\psi:(X,\Phi)\rightarrow (Y_r,\Phi_r)$ is a Borel equivariant injective map with $\pi\circ \psi=\Id_X$.  Note that the flow $(Y_r,\Phi_r)$ is necessarily also aperiodic. Fix $\alpha>0$. By Lemma \ref{Rudf}  for $t$ small enough the time $t$-map $\phi^r_t$ of this symbolic flow admits an $\alpha$-uniform generator $\mathsf T_\mathcal{R}$  given by the towers above the atoms of a clopen $3$-partition $\mathcal{R}$  of a Poincar\'e  cross-section $S'$.  Then $\psi^{-1}\mathsf T_\mathcal{R}$ is an $\alpha$-uniform generator of $\phi_t$  given by the towers of the partition $\psi^{-1}\mathcal R$ of the global Borel section $\psi^{-1}S'$. Indeed we have with $d=d_X$ :
\begin{align*}
\sup_{y\in (\psi^{-1}\mathsf T_{\mathcal{R}})_{\phi_t}^{[-n,n]}(x)}d\left(y, \phi_{[-\alpha,\alpha]}(x)\right)& =\sup_{y\in \psi^{-1}\left((\mathsf T_{\mathcal{R}})_{\phi^r_t}^{[-n,n]}\left(\psi(x)\right)\right)}d\left(y, \phi_{[-\alpha,\alpha]}(x)\right),\\
&\leq   \sup_{z\in \mathsf (T_{\mathcal{R}})_{\phi^r_t}^{[-n,n]}(\psi(x))}d\left(\pi(z), \pi\left(\phi^r_{[-\alpha,\alpha]}(\psi(x)\right)\right).
\end{align*}
and this last right member goes to zero uniformly in $x$ with $n$ as $\pi$ is uniformly continuous and $\mathsf T_{\mathcal{R}}$ is an $\alpha$-uniform generator of $\phi^r_t$. 
\end{proof}


\appendix
\section{Modified Brin-Katok entropy Structure}\label{BKK}
Let $(X,T)$ be a topological system and let   $\mu\in \mathcal{M}(X,T)$. In \cite{dow} T.Downarowicz defines $h^{BK}(\mu,\epsilon)$  for an ergodic measure $\mu$ as done in Subsection \ref{brin}, but then he extends the function harmonically on the whole space $\mathcal{M}(X,T)$. 
  Let $P$ be a finite measurable partition of $X$. By Shanon-MacMillan-Breiman theorem the sequence $-\frac{1}{n}\log \mu(P^n(x))$ is converging for $\mu$-almost every $x$. Moreover the limit $h(\mu,P,x)$ satisfies $h(\mu,P,x)=h(\mu,P,Tx)$ almost everywhere and$ \int h(\mu,P, x)\, d\mu(x)=h(\mu,P)$.

\begin{theo}\label{ffdf}The Brin-Katok entropy structure for a topological system $(X,T)$ as defined in the proof of Lemma \ref{timet} is an entropy structure. 
\end{theo}

For any finite Borel partition $P$ of $X$ we have $h(\mu,P)\geq h(\mu, \diam(P))$ for all $\mu\in \mathcal{M}(X,T)$. When $P$  is a clopen partition we let $Leb(P)$ be the Lebesgue constant of the open cover $P$. Then we have also $h(\mu,P)\leq h(\mu, Leb(P))$. Consequently if $(X,T)$ is a zero-dimensional system then  the entropy structure $\left(h(\cdot,P_k)\right)_k$ for a sequence $(P_k)_k$ of clopen partitions with $\diam(P_k)\xrightarrow{k}0$ is  uniformly equivalent to the Brin-Katok entropy structure.

We deal now independently with a general topological system $(X,T)$. Let $(Y,S)$  be the product of $(X,T)$ with an irrational circle rotation $(\mathbb{S}^1,\mathsf R)$.  As already mentioned the system 
$(Y, S)$ has the small boundary property. Let $(P_k)_k$ be a nonincreasing sequence of partitions  of $Y$ with small boundary and  $\diam(P_k)\xrightarrow{k}0$. Let $\lambda$ be the Lebesgue measure on the circle. The sequence $\left(h(\cdot \times \lambda,P_k)\right)_k$ defines an entropy structure of $(X,T)$ (by definition). By  taking the distance $d_Y$ on $Y$ defined for all $y=(x,t),\, y'=(x',t')\in Y=X\times \mathbb{S}^1 $ by
$ d_Y(y,y')=\max\left(d_X(x,x'),d_{\mathbb{S}^1}(t,t')\right)$ we have 
for all $n \in\mathbb{N}$, for all $\epsilon>0$ and  for all $y=(x,t)\in Y $  
\[B_S(y,n,\epsilon)=B_T(x,n,\epsilon)\times B(t,\epsilon).\]
In particular we get $h^{S}(\mu\times \lambda,\epsilon,y)= h^T(\mu,\epsilon,x)$
and then by integrating with respect to $\mu\times \lambda$
\[h^S(\mu\times \lambda,\epsilon)=h^T(\mu,\epsilon).\]
To conclude the proof of the theorem it is enough to show the sequences $\left(h(\cdot \times \lambda,P_k)\right)_k$ and  $\left(h(\cdot \times \lambda,\epsilon_k)\right)_k$ are  equivalent for some (any) sequence $(\epsilon_k)_k$ of positive numbers with $\lim_k\epsilon_k=0$. As mentioned above we always have $h(\nu,P)\geq h(\nu,\diam(P))$ for any finite Borel partition $P$ of $Y$ and for any $\nu\in \mathcal{M}(Y,S)$.  The theorem 
follows therefore from the following proposition:
\begin{prop}\label{fdfe}
Let $(Y,S)$ be a topological system and let $P$ be a partition of $Y$ with small boundary. Then for all $\gamma>0$ there is $\delta>0$ such that 
\[\forall \nu\in \mathcal{M}(Y,S),  \ \ h(\nu,P)\leq h(\nu,\delta)+3\gamma.\]
\end{prop}
\begin{proof}
We will  show that for all $\gamma>0$ there is $\delta>0$ such that  for any $\nu\in \mathcal{M}(Y,S)$ and for 
  $\nu$-almost every $x$ :
\begin{eqnarray}\label{pipo}h(\nu,P,x)\leq  h(\nu,\delta,x) +3\gamma.\end{eqnarray}
Let $\gamma>0$. 
In \cite{bk,mane} the authors only consider a single measure $\nu$  with a partition satisfying  $\nu(\partial P)=0$.
This last condition together the ergodic theorem allows to control the number of atoms of $P^n$ intersecting a $\nu$-typical dynamical ball of length $n$. To get  uniform estimates in $\nu$ for an essential partition $P$, we use the combinatorial lemma of \cite{burgu}.
We may then apply  verbatim the proof of Ma\~n\'e \cite{mane} to get the desired inequality (\ref{pipo}). As a sake of completeness we give now the details. 

Fix $\gamma'\in ]0,\gamma/2[$ so small that $\limsup_{n}\frac{1}{n}\log \dbinom{\lceil n \gamma' /\log \sharp P\rceil}{n}<\gamma/2$.
By Lemma 6 in \cite{burgu}  there exists $\delta>0$ such that 
\begin{eqnarray}\label{burguet}\limsup_n\sup_{x\in X}\frac{1}{n}\log \sharp\{A^n\in P^n, \ B(x,n,\delta)\cap A^n\neq \emptyset\}<\gamma'.\end{eqnarray} 
In fact it follows from the proof in \cite{burgu} that 
\begin{eqnarray}\label{burgueta}\limsup_n\sup_{x\in X} \sup_{A^n}\frac{1}{n}\sharp\left\{k\in [0,n-1], \ A_k\neq P(S^kx)\right\}<\frac{\gamma'}{\log \sharp P},\end{eqnarray}
where the supremum holds over $ A^n=\bigcap_{k=0}^{n-1}S^{-k}A_k\in P^n$ with $ B(x,n,\delta)\cap A^n\neq \emptyset$.

 For  $n,k\in \mathbb{N}$ we let 
 $$E_k^n:=\{x\in X, \ \nu(P^n(x))\leq e^{-nk\gamma}\}\text{ and}$$
 $$F_{k}^n:=\{x\in E_k^n, \ \exists A^n\in P^n \text{ with } \nu(A^n)\geq  e^{-n(k-2)\gamma}\text{ and } B(x,n,\delta)\cap A^n\neq \emptyset\}.$$
  Then,  for $n$ large enough and for any  fixed $A^n\in P^{n}$, there are at most $\dbinom{\lceil n\gamma'/\log \sharp P \rceil}{n}e^{\gamma' n}$ atoms   $P^n(x)$ with $B(x,n,\delta)\cap A^n\neq \emptyset$ by (\ref{burgueta}). Therefore we have for $n$ large enough :
  \begin{eqnarray*}
  \nu(F_n^k)&\leq & \sum_{\stackrel{A^n\in P^n,}{ \nu(A^n)\geq e^{-n(k-2)\gamma }}}\sum_{\stackrel{x\in E_k^n}{B(x,n,\delta)\cap A^n\neq \emptyset}}\nu(P^n(x)),\\
  &\leq & e^{n(k-2)\gamma } \times \dbinom{\lceil n\gamma' /\log \sharp P \rceil}{n}e^{n\gamma'} \times  e^{-nk\gamma}\leq e^{-\gamma n}.\end{eqnarray*}

Therefore by Borel-Cantelli Lemma,  every $x$ in a subset $E$ of full $\nu$-measure belongs to finitely many 
$E_{n}^k$, $n,k\in \mathbb{N}$.  We may also assume that $-\frac{1}{n}\log \nu(P^n(x))$ is converging (to 
$h(\nu,P,x)$) for $x\in E$.  Let $x\in E$ and $k\in \mathbb{N}$ with $k\gamma< h(\nu,P,x)\leq (k+1)\gamma$. For $n$ large enough, $x$ belongs to $E_k^n\setminus F^n_k$. Thus the dynamical ball $B(x,n,\delta)$   only intersects atoms $A^n\in P^n$ with $\nu(A^n)\leq e^{-n(k-2)\gamma}$, therefore $\nu(B(x,n,\delta))\leq e^{-n(k-3)\gamma}$ by (\ref{burguet}). We get finally 
\begin{eqnarray*}
h(\nu,\delta,x)&\geq & (k-3)\gamma, \\
&\geq &  h(\nu,P,x)-3\gamma.
\end{eqnarray*}

\end{proof}

 

 \section{Proof of the inequality $h^\pi\geq h+u^T_1$}\label{dirs}
 From Theorem 55 in \cite{bdo} we have $h^{\pi}\geq h+ \mathfrak u_{1}^T$ for a symbolic extension $\pi$ with an embedding of $(X,T)$, but in \cite{bdo} the proof involves a delicate intermediate construction, \textit{the enhanced system}. Here we give a direct proof of $h^\pi\geq h +u_1^T\geq h+ \mathfrak u_{1}^T$.
\begin{lemma}\label{fdf}Assume $(X,T)$ is a zero-dimensional system admitting  a symbolic extension $\pi:(Y,S)\rightarrow (X,T)$ with an embedding $\psi:(X,T)\rightarrow (Y,S)$, then  we  have  $$h^\pi\geq h+u_1^T.$$
In particular $u_1^T=0$ when $(X,T)$ is a subshift.
\end{lemma}

\begin{proof}
  We let $\mathcal{Q}=(Q_k)_k$ be a nonincreasing sequence of clopen partitions with $\diam(Q_k)\xrightarrow{k}0$. The sequence of affine upper semicontinuous functions $h_k=h(\cdot, Q_k)$, $k\in \mathbb{N}$ then defines an entropy structure of $(X,T)$. Recall $D$ denotes a convex distance on $\mathcal{M}(X,T)$ inducing the weak-$*$ topology.  We let $Per_n(X,T):=\{x\in X, \ T^nx=x\}$ and $Per(X,T)=\bigcup_{n>0}Per_n(X,T)$. 
  By a standard combinatorial argument there is for any $k$  a positive number $\epsilon_k\in ]0,1/k[$   so small  that for any $x\in Per_n(X,T)$  the number of $A\in Q_k^n$, such that there exists  $y\in Per_n(X,T)\cap A$ with    $D(\mu_y,\mu_x)<\epsilon_k$, is less than $e^{n/k}$.   
  For a periodic point $x$  with minimal period $n$ we recall  $p_k(\mu_x)=\frac{1}{n}\log\sharp\{\mu_y,  \  D(\mu_y,\mu_x)<\epsilon_k \text{ and }y\in Per_n(X,T)\}$. 
The sequence $(p_k)_k$ is converging pointwisely to zero because  there are only finitely many periodic points in $(X,T)$ with a given period as in the subshift $(Y,S)$. Therefore there exists a nondecreasing sequence of  positive integers $(n_k)_k$ going to infinity  such that for all $k$ we have $p_k(\mu_x)=0$ for all periodic points $x$ with minimal period less than $n_k$. For $\mu\in \mathcal{M}(X,T)$ we let $p_k(\mu)=\int p_k(\mu_x)\, d\mu(x)$, $k\in \mathbb{N}$.  By definition  we have $u_1=u_1^T=\lim_k\tilde{p_k}$. 
By Lemma 54 in  \cite{bdo} for all $\mu\in \mathcal{M}(X,T)$ there exists a sequence of $T$-invariant probability measures $(\mu_k)_k$ converging to $\mu$ with
$p_k(\mu_k)\xrightarrow{k}u_1(\mu)$. 
For any periodic point $x$ with minimal period equal to $n$ we  let 
$\gamma^x_k$ be the  probability measure associated to $\sum_{y}\delta_{\psi(y)}$, where the sum holds  over $y\in Per_n(X,T)$ with $D(\mu_y,\mu_x)<\epsilon_k$. Finally  we let 
$\nu_k=\int \gamma^x_k\, d\mu_k(x)\in \mathcal{M}(Y,S)$. Observe that $\pi \gamma_k^x\xrightarrow{k}\mu_x$  and therefore $\pi\nu_k\xrightarrow{k}\mu$ by convexity of $D$.   Let $P$ be the zero-coordinate partition of $Y$. By superharmonicity of $\nu\mapsto H_{\nu}(R|R')$ for any given clopen partitions $R$, $R'
$ of $Y$ we have 
\begin{align*}
\frac{1}{n_k}H_{\nu_k}(P^{n_k}|\pi^{-1}Q_k^{n_k})& \geq 
\frac{1}{n_k}\int_{Per(X,T)\setminus Per_{n_k-1}(X,T)} H_{\gamma_{k}^x}(P^{n_k}|\pi^{-1}Q_k^{n_k}) d\mu_k(x).
\end{align*}
Then for any  periodic point $x$ with minimal period $n_x\geq n_k$ we have :
\begin{align*}
\frac{1}{n_k}H_{\gamma_{k}^x}(P^{n_k}|\pi^{-1}Q_k^{n_k})
& \geq  \frac{1}{n_x}H_{\gamma_{k}^x}(P^{n_x}|\pi^{-1}Q_k^{n_x}),  \text{ because $\left(\frac{1}{n}H_{\xi}(P^n|\pi^{-1}Q_k)\right)_n\searrow$  by Fact 2.2.5 in \cite{dowb}},\\
& \geq \frac{1}{n_x}\left(H_{\gamma_{k}^x}(P^{n_x})-H_{\gamma_k^x}(\pi^{-1}Q_k^{n_x})\right),\\
&\geq  \frac{1}{n_x}\log\sharp\{ y\in Per_{n_x}(X,T)
 \text{ with }D(\mu_y,\mu_x)<\epsilon_k\} \\
 & \ \ \ \  
 -\frac{1}{n_x}\log\sharp\{ A\in Q_k^{n_x}, \ \exists y\in Per_{n_x}(X,T)\cap A \text{ with }D(\mu_y,\mu_x)<\epsilon_k\},\\ 
&\geq  p_k(\mu_x)-1/k.
\end{align*}

Therefore we get 
\begin{align*}
p_k(\mu_k)=\int p_k(\mu_x)d\mu_k(x)\leq \frac{1}{n_k}H_{\nu_k}(P^{n_k}|\pi^{-1}Q_k^{n_k})+1/k.
\end{align*}

The left member goes to $u_1(\mu)$ when $k$ goes to infinity. Let us now show the limsup in $k$ of the right member is not larger  than $h^\pi(\mu)-h(\mu)$. We have for all $k''\leq k'\leq k$

\begin{align*} \frac{1}{n_k}H_{\nu_k}(P^{n_k}|\pi^{-1}Q_k^{n_k})&\leq \frac{1}{n_{k}}H_{\nu_k}(P^{n_{k}}|\pi^{-1}Q_{k''}^{n_{k}}), \text{ since $Q_k$ is  finer than $Q_{k''}$},\\ 
&\leq \frac{1}{n_{k'}}H_{\nu_k}(P^{n_{k'}}|\pi^{-1}Q_{k''}^{n_{k'}})\text{ because $\left(\frac{1}{n}H_{\xi}(P^n|\pi^{-1}Q_k^n)\right)_n\searrow$ as recalled above}.
\end{align*}
The involved partitions being clopen,  we have for any weak limit $\nu$ of $(\nu_k)_k$,  by letting $k$ go to infinity :
$$ \limsup_k\frac{1}{n_k}H_{\nu_k}(P^{n_k}|\pi^{-1}Q_k^{n_k})\leq \frac{1}{n_{k'}}H_{\nu}(P^{n_{k'}}|\pi^{-1}Q_{k''}^{n_{k'}}).$$
As it holds for all $k'$ and $\pi\nu=\mu$ we get with $h_{k''}:=h(\cdot, Q_{k''})$ :
$$ \limsup_k\frac{1}{n_k}H_{\nu_k}(P^{n_k}|\pi^{-1}Q_k^{n_k})\leq h(\nu)-h_{k''}(\pi\nu)\leq (h^{\pi}-h_{k''})(\mu).$$
We conclude the proof by letting $k''\rightarrow +\infty$.\end{proof}

\section{Uniform generators with small boundary for asymptotically expansive systems}\label{dav}

\begin{prop}\label{Gener}
Any asymptotically expansive topological system with the small  boundary property admits an  essential uniform generator. 
\end{prop}
\begin{proof}
Replacing the system by a zero-dimension strongly isomorphic extension, we can assume by Proposition \ref{super} the initial system to be zero-dimensional. Let $(X,T)$ be such a zero-dimensional asymptotically expansive system. By the Main Theorem in \cite{bumo}
 there exists, for a finite alphabet $\mathcal{A}$, a sequence of continuous equivariant maps $\psi_k:(X,T)\rightarrow (\mathcal{A}^{Z}, \sigma)$ converging pointwisely to an embedding $\psi$, such that the induced maps on $\mathcal{M}(X,T)$ are converging uniformly. Moreover it is shown that $\psi^{-1}$ extends continuously to a symbolic extension $\pi:(\overline{\psi(X)}, \sigma)\rightarrow (X,T)$.  The maps $\psi_k$ encode the orbit of $x$ at some scales $\epsilon_k$ with $\epsilon_k\xrightarrow{k}0$.  Except for the so-called \textit{free positions} corresponding to some letter $*$ in $\mathcal{A}$, which represents the coordinates  we can freely use to encode the smaller scales, the other letters are fixed once for all : 
 $$\forall a\in \mathcal{A}\setminus \{*\}, \ \psi^{-1}([a])=\bigcup_{k}\psi_k^{-1}([a]).$$
Moreover the upper asymptotic density of $*$ in any $\psi_k(x)$ goes to zero uniformly in $x\in X$ when $k$ goes to infinity. Consequently we have $\psi\mu([*])=0$ for any $\mu \in \mathcal{M}(X,T)$. 

The partition $P=\{\psi^{-1}(a), \ a\in \mathcal{A}\}$ defines a uniform generator (see  Proposition \ref{super}). Let us now show $P$ is an essential partition. The maps $\psi_k$ being continuous we have for any $\mu\in \mathcal{M}(X,T)$
\begin{eqnarray*}
\sum_{a\in \mathcal{A}}\mu\left( \Int (\psi^{-1}[a])\right)&\geq &\sum_{a\in \mathcal{A}\setminus \{*\}}\mu\left( \Int (\psi^{-1}[a])\right),\\
 &\geq &\lim_k \sum_{a\in \mathcal{A}\setminus \{*\}}\mu\left( \Int (\psi_k^{-1}[a])\right),\\
  &\geq &\sum_{a\in \mathcal{A}\setminus \{*\}}\lim_k \mu\left( \psi_k^{-1}[a]\right),\\
 &\geq &\sum_{a\in \mathcal{A}\setminus \{*\}}\lim_k \psi_k\mu( [a]),\\
 &\geq &\sum_{a\in \mathcal{A}\setminus \{*\}}\psi\mu( [a])=1.
\end{eqnarray*}
Therefore $ P$ has a small boundary.

\end{proof}


\end{document}